\documentclass[11pt,reqno]{amsart}
\usepackage{amsmath,amsfonts,amsthm,amssymb}
\usepackage{mathrsfs}

\usepackage[T2A]{fontenc}
\usepackage[cp1251]{inputenc}
\def\wh{\widehat}
\def\wt{\widetilde}

\def\R{\mathbb R}
\def\C{\mathbb C}
\def\Z{\mathbb Z}

\def\O{\mathcal O}
\def\q{\mathbf q}
\def\f{\mathbf f}
\def\x{\mathbf x}
\def\v{\mathbf v}
\def\a{\mathbf a}
\def\s{\mathbf s}
\def\p{\mathbf p}
\def\r{\mathbf r}
\def\b{\mathbf b}
\def\c{\mathbf c}

\def\w{\mathbf w}
\def\z{\mathbf z}

\def\u{\mathbf u}
\def\e{\mathbf e}
\def\s{\mathbf s}
\def\r{\mathbf r}
\def\c{\mathbf c}

\def\V{\mathbf V}

\def\bups{\boldsymbol \upsilon}
\def\btheta{\boldsymbol \vartheta}
\def\D{\mathbf D}

\def\1{\mathbf 1}

\def\M{\mathcal M}

\def\V{\mathbf V}
\def\O{\mathcal O}

\def\bzeta{\boldsymbol \zeta}
\def\bphi{\boldsymbol \phi}
\def\bpsi{\boldsymbol \psi}
\def\brho{\boldsymbol \rho}

\def\D{\mathbf D}

\def\1{\mathbf 1}

\def\M{\mathcal M}

\def\1{\bold 1}
\def\rot{\mathrm{curl}\,}
\def\div{\mathrm{div}\,}

\def\eps{\varepsilon}
\def\Dom{\mathrm{Dom}\,}

\def\bvarphi{\boldsymbol \varphi}

\def\le{\leqslant}

\theoremstyle{theorem}
\newtheorem{theorem}{Theorem}[section]
\newtheorem{proposition}[theorem]{Proposition}
\newtheorem{lemma}[theorem]{Lemma}
\newtheorem{condition}[theorem]{Condition}
\newtheorem{remark}[theorem]{Remark}

\newtheorem{definition}[theorem]{Definition}

\usepackage[paperwidth=20.5cm,paperheight=29.7cm,width=16cm,height=24cm,top=24mm,left=22mm,headsep=11mm]{geometry}

\begin{document}

\title[Homogenization of the stationary Maxwell system]{Homogenization of the stationary Maxwell system \\
with periodic coefficients in a 
bounded domain}

\author{T.~A.~Suslina}


\address{St. Petersburg State University, Universitetskaya nab. 7/9,  St.~Petersburg, 199034, Russia}

\email{t.suslina@spbu.ru}


\begin{abstract}
In a bounded domain  $\mathcal{O}\subset\mathbb{R}^3$ of class $C^{1,1}$, we consider a stationary Maxwell system with 
the perfect conductivity boundary conditions. It is assumed that the dielectric permittivity and the magnetic permeability are given by 
  $\eta(\x/\eps)$ and $\mu(\x/\eps)$,
where  $\eta(\x)$ and $\mu(\x)$ are symmetric  $(3 \times 3)$-matrix-valued functions; they are periodic with respect to some lattice, 
bounded and positive definite. Here $\eps >0$ is the small parameter. 
We use the following notation for the solutions of the Maxwell system: $\u_\eps$
is the electric field intensity, $\v_\eps$ is the magnetic field intensity, 
$\w_\eps$ is the electric displacement vector, and $\z_\eps$ is the magnetic displacement vector. 
It is known that $\u_\eps$, $\v_\eps$, $\w_\eps$, and $\z_\eps$ weakly converge in 
$L_2(\O)$ to the corresponding homogenized fields $\u_0$, $\v_0$, $\w_0$,~and $\z_0$ (the solutions of the homogenized Maxwell system with the effective coefficients), as $\eps \to 0$. We improve the classical results and find approximations for  $\u_\eps$, $\v_\eps$, $\w_\eps$, and $\z_\eps$ 
in the~$L_2(\O)$-norm. The error terms do not exceed $C \sqrt{\eps} (\|\q\|_{L_2}+\|\r\|_{L_2})$, where 
the divergence free vector-valued functions $\q$ and $\r$ are the right-hand sides of the Maxwell equations.
\end{abstract}

\thanks{Supported by Russian Science Foundation (project 17-11-01069)}

\maketitle

\numberwithin{equation}{section}

\section*{Introduction}

The paper concerns homogenization theory for differential operators (DO's) with periodic coefficients. 
The bibliography on homogenization is rather extensive; we mention the monographs  \cite{BeLPap, BaPa,Sa, ZhKO}.

\subsection{Operator error estimates}

Let $\Gamma \subset \R^d$ be a lattice. For any $\Gamma$-periodic function $f$ in $\R^d$, we denote 
$$
f^\eps(\x):= f(\x/\eps),\quad \eps >0.
$$ 

In a series of papers \cite{BSu1,BSu2,BSu3} by Birman and Suslina, 
an operator-theoretic approach to homogenization problems was suggested and developed. 
A wide class of matrix second order strongly elliptic operators $A_\eps$ 
acting in $L_2(\R^d;\C^n)$ and admitting a factorization of the form 
\begin{equation}
\label{0.1}
A_\varepsilon = b(\D)^* g^\eps(\x) b(\D)
\end{equation}
was studied.
Here a matrix-valued function $g(\x)$ is bounded, positive definite, and $\Gamma$-periodic; $b(\D)$ is a matrix first order DO of the form
$b(\D)= \sum_{j=1}^d b_j D_j$ such that its symbol has maximal rank. 
The simplest example of the operator \eqref{0.1} is the scalar  elliptic operator  
$A_\eps = - \div g^\eps(\x) \nabla = \D^* g^\eps(\x)\D$ (the acoustics operator).
The elasticity operator, as well as  an auxiliary operator  
$A_\eps = \rot a^\eps(\x) \rot - \nabla \nu^\eps(\x) \div$ arising in electrodynamics, 
also can be written as  \eqref{0.1}.

In \cite{BSu1}, it was shown that, as  $\eps \to 0$, the resolvent $(A_\eps +I)^{-1}$ converges  
to the resolvent of the \textit{effective operator} $A^0= b(\D)^* g^0 b(\D)$ in the operator norm in $L_2(\R^d;\C^n)$. 
Here $g^0$ is a constant positive matrix called the \textit{effective matrix}.
We have 
\begin{equation}
\label{0.2}
\| (A_\varepsilon +I)^{-1} - (A^0 +I)^{-1} \|_{L_2(\R^d) \to L_2(\R^d)} \le C \eps. 
\end{equation}
In \cite{BSu3}, approximation for the resolvent $(A_\eps +I)^{-1}$ in the norm of operators acting from $L_2(\R^d;\C^n)$ to the Sobolev space
$H^1(\R^d;\C^n)$ was found:
\begin{equation}
\label{0.3}
\| (A_\varepsilon +I)^{-1} - (A^0 +I)^{-1} - \eps K(\eps) \|_{L_2(\R^d) \to H^1(\R^d)} \le C \eps. 
\end{equation}
Here $K(\eps)$ is the so called \textit{corrector}. It contains a rapidly oscillating factor, so that 
$
\|K(\eps)\|_{L_2 \to H^1}= O(\eps^{-1}).
$

Estimates \eqref{0.2} and \eqref{0.3} are order-sharp. 
The results of such type are called the \textit{operator error estimates} in homogenization theory. 
Another approach to operator error estimates (the modified method of the first order approximation or the shift method)  was suggested by Zhikov. 
In \cite{Zh1, ZhPas1}, by this method, estimates \eqref{0.2} and \eqref{0.3} were proved for the acoustics and elasticity operators. 
Further results are discussed in the survey \cite{ZhPas2}.

Operator error estimates were also studied for the boundary value problems in a bounded domain $\O \subset \R^d$ with sufficiently smooth boundary; 
see \cite{ZhPas1, ZhPas2, Gr1, Gr2, KeLiS, PSu, Su13, Su_SIAM, Su15}.
Let $A_{D,\eps}$ and $A_{N,\eps}$ be the operators in $L_2(\O;\C^n)$ given by the expression $b(\D)^* g^\eps(\x) b(\D)$
with the Dirichlet or Neumann boundary conditions. Let $A_D^0$ and $A_N^0$ be the corresponding effective operators. Then we have 
\begin{align}
\label{0.4}
& \| (A_{\flat,\varepsilon} +I)^{-1} - (A_\flat^0 +I)^{-1}  \|_{L_2(\O) \to L_2(\O)} \le C \eps, 
\\
\label{0.5}
& \| (A_{\flat,\varepsilon} +I)^{-1} - (A_\flat^0 +I)^{-1} - \eps K_\flat(\eps) \|_{L_2(\O) \to H^1(\O)} \le C \eps^{1/2}. 
\end{align}
Here $\flat = D,N$, and $K_\flat(\eps)$ is the corresponding corrector. Estimate \eqref{0.4} is of sharp order $O(\eps)$ 
(the order is the same as for the similar problem in  $\R^d$).
The order of estimate \eqref{0.5} is worse compared with  \eqref{0.3}; this is explained by the boundary influence. 

In \cite{ZhPas1}, by the shift method, estimate \eqref{0.5}
and the analog of \eqref{0.4} with error term $O(\sqrt{\eps})$ were obtained for the acoustics and elasticity operators. 
Independently, Griso \cite{Gr1, Gr2} obtained similar results for the acoustics operator using the unfolding method. 
For the first time, the sharp order estimate  \eqref{0.4} was proved in  \cite{Gr2}. 
The case of matrix elliptic operators was studied in \cite{KeLiS} (where uniformly elliptic operators under some regularity assumptions on the coefficients were considered)  and in \cite{PSu, Su13, Su_SIAM, Su15} (where estimates \eqref{0.4} and \eqref{0.5} were proved 
for strongly elliptic operators described above). 

\subsection{Homogenization of the Maxwell system in~$\R^3$\label{sec0.2}}

Now, we discuss homogenization problem for the stationary Maxwell system in  $\R^3$.

Suppose that the dielectric permittivity and the magnetic permeability are given by the matrix-valued functions 
$\eta^\eps(\x)$ and $\mu^\eps(\x)$, where $\eta(\x)$ and $\mu(\x)$ are bounded, positive definite, and periodic with respect to some lattice~$\Gamma$. 
Let~$J(\R^3)$ denote the subspace of vector-valued functions ${\mathbf f} \in L_2(\R^3; \C^3)$ such that $\div {\mathbf f} =0$ 
(in the sense of distributions).
By $\u_\eps$ and $\v_\eps$ we denote the electric field intensity and the magnetic field intensity; $\w_\eps = \eta^\eps \u_\eps$ and $\z_\eps = \mu^\eps \v_\eps$ are the electric and magnetic displacement vectors. We write the Maxwell operator~$M_\eps$ in terms of the displacement vectors, 
assuming that $\w_\eps$ and $\z_\eps$ are divergence free.
Then the operator $M_\eps$ acts in the space $J(\R^3)\oplus J(\R^3)$ and is given by 
$$
M_\eps = \begin{pmatrix}
0 & i \rot (\mu^\eps)^{-1} \cr - i \rot (\eta^\eps)^{-1} & 0
\end{pmatrix}
$$
on the natural domain. The operator $M_\eps$ is selfadjoint, if $J(\R^3)\oplus J(\R^3)$ is considered as a subspace of the weighted space 
$L_2(\R^3;\C^3; (\eta^\eps)^{-1}) \oplus L_2(\R^3;\C^3; (\mu^\eps)^{-1})$. The point $\lambda =i$ is a regular point for $M_\eps$.

We discuss the problem about the behavior of the resolvent $(M_\eps - iI)^{-1}$ for small $\eps$. 
In other words, we are interested in the behavior of the solutions $(\w_\eps, \z_\eps)$ of the Maxwell system
\begin{equation}
\label{0.6}
(M_\varepsilon - iI) \begin{pmatrix} \w_\eps \cr \z_\eps \end{pmatrix} = \begin{pmatrix} \q \cr \r \end{pmatrix}, 
\quad  \q, \r \in J(\R^3;\C^3),
\end{equation}
and the fields $\u_\eps = (\eta^\eps)^{-1}\w_\eps$,  $\v_\eps = (\mu^\eps)^{-1}\z_\eps$.

The homogenized Maxwell operator $M^0$ has the coefficients $\eta^0$, $\mu^0$; it is well known that the effective matrices $\eta^0$ and $\mu^0$ 
are the same as for the scalar elliptic operators $-\div \eta^\eps \nabla$ and $-\div \mu^\eps \nabla$. Let 
$(\w_0, \z_0)$ be the solution of the homogenized Maxwell system
\begin{equation*}
(M^0 - iI) \begin{pmatrix} \w_0 \cr \z_0 \end{pmatrix} = \begin{pmatrix} \q \cr \r \end{pmatrix}, 
\end{equation*}
and let $\u_0 = (\eta^0)^{-1}\w_0$, $\v_0 = (\mu^0)^{-1}\z_0$.
The classical results (see, e.~g., \cite{BeLPap, Sa, ZhKO}) say that 
the functions $\u_\eps, \w_\eps, \v_\eps, \z_\eps$ 
\textit{weakly converge} in $L_2(\R^3;\C^3)$ to the corresponding homogenized fields $\u_0, \w_0, \v_0, \z_0$, as $\eps \to 0$.

Operator error estimates for the Maxwell system  \eqref{0.6} have been studied in  \cite[Chapter 7]{BSu1}, \cite[\S 14]{BSu2}, \cite[\S 22]{BSu3}, \cite{BSu4} (in  the case of constant permeability), and in 
  \cite{Su1, Su2} (in the general case ).
The method was to reduce the problem to the study of some auxiliary second order operator. 
The solution of system \eqref{0.6} can be written as  $\w_\eps = \w_\eps^{(\q)}+\w_\eps^{(\r)}$, $\z_\eps = \z_\eps^{(\q)}+\z_\eps^{(\r)}$,
where $(\w_\eps^{(\q)},\z_\eps^{(\q)})$ is the solution of this system with  $\r=0$, 
and $(\w_\eps^{(\r)},\z_\eps^{(\r)})$ is the solution of this system with $\q=0$. 
For instance, let us consider $(\w_\eps^{(\r)},\z_\eps^{(\r)})$. Substituting  the first equation  
$\w_\eps^{(\r)} = \rot (\mu^\eps)^{-1} \z_\eps^{(\r)}$ in the second one, we arrive at the following problem for $\z_\eps^{(\r)}$:
$$
\rot (\eta^\eps)^{-1}\rot (\mu^\eps)^{-1} \z_\eps^{(\r)} + \z_\eps^{(\r)} = i \r, \quad \div \z_\eps^{(\r)} =0. 
$$
Putting $\f_\eps^{(\r)} = (\mu^\eps)^{-1/2}\z_\eps^{(\r)}$ and lifting the divergence free condition, we see that $\f_\eps^{(\r)}$ is the solution of the second order elliptic 
equation
\begin{equation}
\label{0.8}
L_\eps \f_\eps^{(\r)}+ \f_\eps^{(\r)}= i (\mu^\eps)^{-1/2}\r,
\end{equation}
where
\begin{equation}
\label{0.9}
L_\eps = (\mu^\eps)^{-1/2}\rot (\eta^\eps)^{-1}\rot (\mu^\eps)^{-1/2} - (\mu^\eps)^{1/2} \nabla \div (\mu^\eps)^{1/2}. 
\end{equation}
 The field $\w_\eps^{(\r)}$ is expressed in terms of the derivatives of the solution: 
$$
\w_\eps^{(\r)} = \rot (\mu^\eps)^{-1/2} \f_\eps^{(\r)}.
$$

If $\mu$ is constant,  the operator \eqref{0.9} is of the form~\eqref{0.1}, which allows one to apply 
general results of \cite{BSu1, BSu2, BSu3} to  equation  \eqref{0.8}.
In the case of variable $\mu$, this is not so, however, it is possible to use the abstract scheme from \cite{BSu1, BSu2, BSu3} to study the operator 
\eqref{0.9}. By this way, in \cite{Su1, Su2}, approximation of the resolvent  ${(M_\eps - iI)^{-1}}$ was found. 
In contrast to the resolvent of the operator \eqref{0.1}, in the general case, this resolvent has no  limit in the operator norm.
However, this resolvent can be approximated by the sum of $(M^0 - iI)^{-1}$ and some zero order corrector (which weakly tends to zero);
the corresponding error estimate is of sharp order $O(\eps)$. 
In terms of the solutions, this implies approximations in the $L_2(\R^3;\C^3)$-norm for all physical fields with error estimates of order $O(\eps)$.
For instance, we write down the result for $\u_\eps$:
$$
\| \u_\eps - \u_0 - \u_\eps^{(1)}\|_{L_2(\R^3)} \le C \eps (\| \q\|_{L_2(\R^3)} + \| \r \|_{L_2(\R^3)}).
$$
Here $\u_\eps^{(1)}$ weakly converges to zero and  is interpreted as a corrector of zero order; it is given in terms of $\u_0$, the solution of some ``correction'' Maxwell system, and 
some rapidly oscillating factor.

\subsection{Statement of the problem. Main results}
In the present paper, we study homogenization of the stationary Maxwell system in a bounded domain 
$\O \subset \R^3$ of class $C^{1,1}$. We rely on the general theory of the Maxwell operator in arbitrary domains developed by Birman and Solomyak \cite{BS1,BS2}.

As above, we assume that  the dielectric permittivity and the magnetic permeability are given by the rapidly oscillating matrix-valued functions
$\eta^\eps(\x)$ and $\mu^\eps(\x)$. 
For the physical fields we use the same notation as in Subsection~\ref{sec0.2}.
The Maxwell operator $\M_\eps$ written in terms of the displacement vectors acts in the space $J(\O) \oplus J_0(\O)$,
where $J(\O)$ and $J_0(\O)$ are divergence free subspaces of $L_2(\O;\C^3)$ defined below by 
\eqref{5.1},~\eqref{5.2}. The operator $\M_\eps$ is given by 
$$
\M_\eps = \begin{pmatrix}
0 & i \rot (\mu^\eps)^{-1} \cr - i \rot (\eta^\eps)^{-1} & 0
\end{pmatrix}
$$
on the natural domain with the perfect conductivity boundary conditions (see \eqref{Dom_M_eps}  below). 
The operator $\M_\eps$ is selfadjoint if $J(\O)\oplus J_0(\O)$ is treated as a subspace of the weighted space 
$$
L_2(\O;\C^3; (\eta^\eps)^{-1}) \oplus L_2(\O;\C^3; (\mu^\eps)^{-1}).
$$ 

We study the resolvent $(\M_\eps - iI)^{-1}$. In other words, we are interested in the behavior of the solutions $(\w_\eps, \z_\eps)$ of the Maxwell system
\begin{equation}
\label{0.8a}
(\M_\varepsilon - iI) \begin{pmatrix} \w_\eps \cr \z_\eps \end{pmatrix} = \begin{pmatrix} \q \cr \r \end{pmatrix}, 
\quad \q \in J(\O),\ \r \in J_0(\O),
\end{equation}
and the fields $\u_\eps = (\eta^\eps)^{-1}\w_\eps$ and $\v_\eps = (\mu^\eps)^{-1}\z_\eps$.

Let $\M^0$ be the homogenized Maxwell operator with the coefficients $\eta^0$ and $\mu^0$.
The homogenized Maxwell system has the form
\begin{equation*}
(\M^0 - iI) \begin{pmatrix} \w_0 \cr \z_0 \end{pmatrix} = \begin{pmatrix} \q \cr \r \end{pmatrix}.
\end{equation*}
We put $\u_0 = (\eta^0)^{-1}\w_0$, $\v_0 = (\mu^0)^{-1}\z_0$.
As for the problem in  $\R^3$, the classical results  (see \cite{BeLPap, Sa, ZhKO}) give weak convergence in 
$L_2(\O;\C^3)$ of the vector-valued functions $\u_\eps, \w_\eps, \v_\eps, \z_\eps$ 
to the corresponding homogenized fields $\u_0, \w_0, \v_0, \z_0$.

We find approximations in the $L_2(\O;\C^3)$-norm for all four fields $\u_\eps, \w_\eps, \v_\eps, \z_\eps$.
These approximations are similar to each other.
For instance, for $\u_\eps$ we have 
\begin{equation}
\label{0.9a}
\| \u_\eps - \u_0 - \u_\eps^{(1)}\|_{L_2(\O)} \le C \eps^{1/2} (\| \q\|_{L_2(\O)} + \| \r \|_{L_2(\O)}).
\end{equation}
The term $\u_\eps^{(1)}$ weakly converges to zero and can be interpreted as a corrector of zero order; 
it is expressed in terms of $\u_0$, the solution of some ``correction'' Maxwell system, and some rapidly oscillating factor. 
 The order of estimate  \eqref{0.9a} deteriorates as compared with the problem in $\R^3$, this is explained by the boundary influence.

In the case where the  magnetic permeability is given by the constant matrix $\mu_0$ and $\q=0$ in the right-hand side of  
\eqref{0.8a}, the result can be improved. This case has been studied in  \cite{Su6}.
It turns out that, under such assumptions,  $\v_\eps$ and $\z_\eps$ converge in the $L_2(\O;\C^3)$-norm to  $\v_0$ and $\z_0$, respectively,  
and the error terms are estimated by $C \eps \|\r\|_{L_2(\O)}$.
For the fields $\u_\eps$ and $\w_\eps$, approximations with the error terms not exceeding  $C \eps^{1/2} \|\r\|_{L_2(\O)}$ were found.

\subsection{Method} 
As for the problem in $\R^3$, the method is based on reduction to the study of some boundary value problems for
second order equations. In the case where $\q=0$, the following problem arises:
\begin{equation}\label{0.11}
\left\{\begin{matrix}
& \rot (\eta^\eps)^{-1} \rot (\mu^\eps)^{-1} \z_\eps^{(\r)}  + \z_\eps^{(\r)} = i \r,
\quad \div \z_\eps^{(\r)} =0,
\cr
& 
(\z_\eps^{(\r)})_n \vert_{\partial \O} =0, 
\quad  ((\eta^\eps)^{-1} \rot (\mu^\eps)^{-1} \z_\eps^{(\r)})_\tau\vert_{\partial \O} =0.
\end{matrix}
\right.
\end{equation}
In the case where $\r=0$, the question is reduced to the problem
\begin{equation}\label{0.12}
\left\{\begin{matrix}
& \rot (\mu^\eps)^{-1} \rot (\eta^\eps)^{-1} \w_\eps^{(\q)}  + \w_\eps^{(\q)} = i \q,\quad
\div \w_\eps^{(\q)} =0,
\cr
& ( (\eta^\eps)^{-1} \w_\eps^{(\q)})_\tau \vert_{\partial \O} =0, 
\quad  (\rot (\eta^\eps)^{-1} \w_\eps^{(\q)})_n\vert_{\partial \O} =0.
\end{matrix}
\right.
\end{equation}
Problems \eqref{0.11} and \eqref{0.12} are similar to each other, but the boundary conditions are different.
We study these problems separately. 

For instance, let us discuss problem \eqref{0.11}.  We rely on the results in $\R^3$ and look for approximation of the solution 
$\z_\eps^{(\r)}$ as the sum of three terms: the effective field $\z_0^{(\r)}$, the corrector (similar to the corrector in $\R^3$), and the boundary layer correction term.
The last term is the solution of some boundary value problem for the equation with rapidly oscillating coefficients. 
It turns out that the error of such approximation is of sharp order $O(\eps)$ in the energy norm. 
However, it is difficult to control the boundary layer correction term.
The main technical work is related to estimation of this term. We show  that this term is of order   
$O(\sqrt{\eps})$ in the energy norm. These considerations allow us to approximate the solution $\z_\eps^{(\r)}$ 
by the sum of the effective field and the corrector with error term of order $O(\sqrt{\eps})$.

Problem \eqref{0.12} is studied similarly. Combining the results for problems \eqref{0.11} and \eqref{0.12}, 
we deduce the results for the Maxwell system.

\subsection{Plan of the paper} 
The paper contains seven sections. Preliminaries are given in Section~\ref{Sec1}. 
In Section \ref{Sec2}, we formulate the statement of the problem  and the main results.
 In Section 3, the question is reduced to the study of the boundary value problems \eqref{0.11} and \eqref{0.12}. 
 In  Section 4,  problem  \eqref{0.11}  is studied. 
 The effective problem is described, the first order approximation to the solution is defined, and the boundary layer correction term is introduced. 
 Theorem 4.6 about estimate for the boundary layer correction term is formulated; 
  the final results for the Maxwell system with  $\q=0$ are deduced from this theorem. Section 5 is devoted to the proof of Theorem 4.6.
 The similar study of  problem \eqref{0.12} (the case where $\r=0$) is given in Sections 6 and 7.

\subsection{Notation} 
Let $\mathfrak{H}$ and $\mathfrak{H}_*$ be complex separable Hilbert spaces. The symbols $(\,\cdot\, ,\,\cdot\,)_\mathfrak{H}$ and $\Vert \,\cdot\,\Vert _\mathfrak{H}$ 
stand for the inner product and the norm in  $\mathfrak{H}$; the symbol $\Vert \,\cdot\,\Vert _{\mathfrak{H}\rightarrow\mathfrak{H}_*}$ denotes the norm of 
a linear continuous operator from $\mathfrak{H}$ to $\mathfrak{H}_*$.

The symbols $\langle \,\cdot\, ,\,\cdot\,\rangle$ and $\vert \,\cdot\,\vert$ stand for the inner product and the norm in  $\mathbb{C}^n$, 
$\mathbf{1}= \mathbf{1}_n$ is the identity $(n\times n)$-matrix. If $a$ is an $(n\times n)$-matrix, then the symbol $\vert a\vert$ means the norm of 
$a$ viewed as a linear operator in  $\mathbb{C}^n$. 
We denote $\mathbf{x}=(x_1,x_2, x_3)\in\mathbb{R}^3$, $iD_j=\partial _j =\partial /\partial x_j$, $j=1,2,3$, $\mathbf{D}=-i\nabla=(D_1,D_2,D_3)$. 
The $L_p$-classes of  $\mathbb{C}^n$-valued functions in the domain $\mathcal{O}\subset\mathbb{R}^3$ are denoted by 
$L_p(\mathcal{O};\mathbb{C}^n)$, $1\leqslant p\leqslant \infty$. The Sobolev classes of  $\mathbb{C}^n$-valued functions in the domain  $\mathcal{O}$ are denoted by
 $H^s(\mathcal{O};\mathbb{C}^n)$. The symbol $H^1_0(\mathcal{O};\mathbb{C}^n)$ stands for the closure of $C_0^\infty (\mathcal{O};\mathbb{C}^n)$ in $H^1(\mathcal{O};\mathbb{C}^n)$. If $n=1$, we write simply $L_p(\mathcal{O})$, $H^s(\mathcal{O})$, etc., but sometimes we use such simple notation also 
 for the spaces of vector-valued or matrix-valued functions. 
Various constants in estimates are denoted by  $c$, $\mathfrak c$, $C$, $\mathcal{C}$, $\mathfrak{C}$ 
(possibly, with indices and marks).

\section{Preliminaries\label{Sec1}}

\subsection{Lattice\label{Sec1.1}} Let $\Gamma\subset\mathbb{R}^3$ be a lattice generated by the basis $\a_1, \a_2, \a_3$, i.~e., 
$$
\Gamma = \big\{ \a \in \R^3:\ \a = z_1 \a_1 + z_2 \a_2 + z_3 \a_3,\ z_j \in \Z\big\}. 
$$
By $\Omega \subset \R^3$ we denote the elementary cell of the lattice $\Gamma$:
$$
\Omega = \big\{ \x \in \R^3:\ \x = t_1 \a_1 + t_2 \a_2 + t_3 \a_3,\ -1/2 < t_j < 1/2 \big\}. 
$$
Let $\b_1, \b_2, \b_3 \in \R^3$ be the basis dual to  $\a_1, \a_2, \a_3$, i.~e., 
$\langle \b_l, \a_j \rangle = 2\pi \delta_{lj}$.  Denote
$$
2r_0 = \min_{j=1,2,3} |\b_j|, \quad 2r_1 = \operatorname{diam}\Omega.
$$

For $\Gamma$-periodic functions $f(\x)$ in $\mathbb{R}^3$ we denote $f^\varepsilon (\mathbf{x}):=f(\mathbf{x}/\varepsilon)$,  $\varepsilon >0$.
For square periodic matrix-valued functions $f(\x)$, we put
$$
\overline{f}:=\vert \Omega\vert ^{-1}\int_\Omega f(\mathbf{x})\,d\mathbf{x}, \quad \underline{f}:=\Big(\vert \Omega\vert ^{-1}\int_\Omega f(\mathbf{x})^{-1}\,d\mathbf{x}
\Big)^{-1}.
$$
In the definition of $\overline{f}$, it is assumed that $f \in L_{1,\text{loc}}(\R^3)$; in the definition of $\underline{f}$, it is assumed that the matrix $f(\x)$ is non-degenerate
and $f^{-1} \in L_{1,\text{loc}}(\R^3)$.

Let $\widetilde{H}^1(\Omega;\C^n)$ be the subspace of $H^1(\Omega;\C^n)$ consisting of functions whose $\Gamma$-periodic extension to $\mathbb{R}^3$ belongs to 
$H^1_{\mathrm{loc}}(\mathbb{R}^3;\C^n)$.

\subsection{The Steklov smoothing\label{Sec1.2a}} 
 We define the operator $S_\varepsilon^{(n)}\!,$ ${\varepsilon \!>\!0}$, acting in $L_2(\mathbb{R}^3;\mathbb{C}^n)$ (where $n \in\mathbb{N}$) 
 and given by 
 \begin{equation}
\label{S_eps}
\begin{split}
(S_\varepsilon^{(n)} \mathbf{u})(\mathbf{x})=\vert \Omega \vert ^{-1}\int_\Omega \mathbf{u}(\mathbf{x}-\varepsilon \mathbf{y})\,d\mathbf{y},\quad \mathbf{u}\in L_2(\mathbb{R}^3;\mathbb{C}^n);
\end{split}
\end{equation}
$S_\eps$ is called the \textit{Steklov smoothing operator}.
We drop the index $n$ and write simply $S_\varepsilon$. Obviously,
$S_\varepsilon \mathbf{D}^\alpha \mathbf{u}=\mathbf{D}^\alpha S_\varepsilon \mathbf{u}$ for $\mathbf{u}\in H^\sigma(\mathbb{R}^3;\mathbb{C}^n)$ 
and any multiindex $\alpha$ such that $\vert \alpha\vert \leqslant \sigma$. Note that 
\begin{equation}
\label{S_eps <= 1}
\Vert S_\varepsilon \Vert _{L_2(\mathbb{R}^3)\rightarrow L_2(\mathbb{R}^3)}\leqslant 1.
\end{equation}
We need the following properties of the operator $S_\varepsilon$
(see \cite[Lemmas~1.1 and~1.2]{ZhPas1} or \cite[Propositions 3.1 and 3.2]{PSu}).

\begin{proposition}
\label{prop_Seps - I}
For any function $\mathbf{u}\in H^1(\mathbb{R}^3;\mathbb{C}^n)$, we have
\begin{equation*}
\Vert S_\varepsilon \mathbf{u}-\mathbf{u}\Vert _{L_2(\mathbb{R}^3)}\leqslant \varepsilon r_1\Vert \mathbf{D}\mathbf{u}\Vert _{L_2(\mathbb{R}^3)},
\end{equation*}
where $2r_1=\mathrm{diam}\,\Omega$.
\end{proposition}

\begin{proposition}
\label{prop f^eps S_eps}
Let $f$ be a $\Gamma$-periodic function in $\mathbb{R}^3$ such that $f \in L_2(\Omega)$. Let $[f ^\varepsilon ]$
be the operator of multiplication by the function $f^\eps(\x)$. Then the operator $[f ^\varepsilon ]S_\varepsilon $ is continuous in $L_2(\mathbb{R}^3)$ and
\begin{equation*}
\Vert [f^\varepsilon]S_\varepsilon \Vert _{L_2(\mathbb{R}^3)\rightarrow L_2(\mathbb{R}^3)}\leqslant \vert \Omega \vert ^{-1/2}\Vert f \Vert _{L_2(\Omega)}.
\end{equation*}
\end{proposition}

\subsection{Functional classes}
Let ${\O \subset \R^3}$ be a bounded domain. 
If the boundary $\partial \O$ and the vector-valued function $\u(\x)$ are sufficiently smooth, 
then the normal component $\u_n$ and the tangential component $\u_\tau$ of $\u$ on the boundary are correctly defined.
 In the nonsmooth situation, relations $\u_n\vert_{\partial \O}=0$ and $\u_\tau\vert_{\partial \O}=0$ can be understood 
 in the generalized sense. Recall the following definitions; see \cite{BS1, BS2}.

\begin{definition}\label{def1}
Let $\u \in L_2(\O;\C^3)$ and $\div \u \in L_2(\O)$. By definition, the relation  \hbox{$\u_n \vert_{\partial \O}=0$} means that 
$$
(\u, \nabla \omega)_{L_2(\O)} = - (\div \u, \omega)_{L_2(\O)},\quad \forall \omega \in H^1(\O). 
$$
\end{definition}

\begin{definition}\label{def2}
Let $\u \in L_2(\O;\C^3)$ and $\rot \u \in L_2(\O;\C^3)$. By definition, the relation $\u_\tau \vert_{\partial \O}=0$ means that
$$
(\u, \rot \z)_{L_2(\O)} = (\rot \u, \z)_{L_2(\O)},\quad \forall \z \in L_2(\O;\C^3): \ \rot \z \in L_2(\O;\C^3). 
$$
\end{definition}

Let $s(\x)$ be a symmetic $(3 \times 3)$-matrix-valued function in $\O$ with real entries and such that $s, s^{-1} \in L_\infty$ and $s(\x)>0$. 
Besides the ordinary space $L_2(\O;\C^3)$, we need to define the weighted space 
$
L_2(\O;s) = L_2(\O;\C^3;s)
$
with the inner product
$$
({\mathbf f}_1, {\mathbf f}_2)_{L_2(\O;s)} = \int_\O \langle s(\x) {\mathbf f}_1(\x), {\mathbf f}_2(\x) \rangle \,d\x.
$$

We introduce two divergence free subspaces in $L_2(\O;\C^3)$:
\begin{align}
\label{5.1}
J(\O) :=& \bigl\{ {\mathbf u} \in L_2(\O;\C^3): \ \int_\O \langle \u, \nabla \omega \rangle \, d\x =0, \ \forall \omega \in H^1_0(\O) \bigr\},
\\
\label{5.2}
J_0(\O) :=& \bigl\{ {\mathbf u} \in L_2(\O;\C^3): \ \int_\O \langle \u, \nabla \omega \rangle \, d\x =0, \ \forall \omega \in H^1(\O) \bigr\}.
\end{align}
The subspace \eqref{5.1} consists of all functions ${\mathbf u} \in L_2(\O;\C^3)$ such that $\div \u =0$ in the sense of distributions. 
The subspace  \eqref{5.2} consists of all functions  ${\mathbf u} \in L_2(\O;\C^3)$ such that  $\div \u =0$ and  
$\u_n\vert_{\partial \O}=0$ (in the sense of Definition \ref{def1}). 
The sets  \eqref{5.1} and \eqref{5.2} can be viewed as subspaces of the weighted space $L_2(\O;s)$.

\subsection{Estimates in the neighborhood of the boundary}
In this subsection,  we formulate two auxiliary statements that are valid for Lipschitz bounded domains  $\O \subset \R^3$; see \cite{ZhPas1} and \cite[Section 5]{PSu}. 
More precisely, we assume the following. 

\begin{condition}\label{cond1}
Let $\O \subset \R^3$ be a bounded domain. We put $(\partial \O)_\eps = \{ \x \in \R^3: \operatorname{dist} \{ \x; \partial \O\} < \eps\}$.
Suppose that there exists a number  $\eps_0 \in (0,1]$ such that the set
 $(\partial \O)_{2\eps_0}$ can be covered by a finite number of  open sets admitting diffeomorphisms of class $C^{0,1}$ rectifying the boundary $\partial \O$. Denote $\eps_1 = \eps_0(1+r_1)^{-1}$, where $2r_1 = \operatorname{diam} \Omega$.
\end{condition}

Condition \ref{cond1} is ensured by the fact that the boundary is Lipschitz. 
The number $\eps_0$ depends only on the domain  $\O$, while $\eps_1$ depends on the domain $\O$ and the parameters of the lattice $\Gamma$.

\begin{lemma}\label{lem01}
Suppose that Condition {\rm \ref{cond1}} is satisfied. Denote
 $B_{2\eps} = (\partial \O)_{2\eps} \cap \O$. Then the following statements hold. 
 
 $1^\circ$. For any function $u \in H^1(\O)$, we have 
 $$
\int_{B_{2\eps}} |u(\x)|^2 \, d\x \leqslant \beta_0 \eps \|u\|_{H^1(\O)} \|u\|_{L_2(\O)},
\quad 0< \eps \leqslant \eps_0.
$$

$2^\circ$. For any function $u \in H^1(\R^3)$, we have
$$
\int_{(\partial \O)_{2\eps}} |u(\x)|^2 \, d\x \leqslant \beta_0 \eps \|u\|_{H^1(\R^3)} \|u\|_{L_2(\R^3)},
\quad 0< \eps \leqslant \eps_0.
$$
The constant $\beta_0$ depends only on the domain $\O$.
\end{lemma}

\begin{lemma}\label{lem02}
Suppose that Condition {\rm \ref{cond1}} is satisfied. Let $f(\x)$ be a $\Gamma$-periodic function in  $\R^3$ such that $f\in L_2(\Omega)$. 
Let $S_\eps$ be given by  \eqref{S_eps}. 
Then for $0< \eps \leqslant \eps_1$   and any function $\u \in H^1(\R^3;\C^n)$ we have
$$
\int_{(\partial \O)_{2\eps}} |f^\eps(\x)|^2 | (S_\eps \u)(\x)|^2 \, d\x \leqslant \beta_* \eps |\Omega|^{-1} \|f\|^2_{L_2(\Omega)} \|\u\|_{H^1(\R^3)} \|\u\|_{L_2(\R^3)}.
$$
Here $\beta_* = \beta_0 (1+r_1)$.
\end{lemma}

\section{Statement of the problem. Main results\label{Sec2}}

\subsection{Statement of the problem}

Suppose that $\eta(\x)$ and $\mu(\x)$ are 
 symmetric $(3\times 3)$-matrix-valued functions in $\R^3$ with  real entries, periodic with respect to the lattice $\Gamma$, and such that 
 \begin{equation*}
\eta, \eta^{-1} \in L_\infty,\quad \eta(\x)>0;\quad 
\mu, \mu^{-1} \in L_\infty,\quad \mu(\x)>0.
\end{equation*}

Let  ${\O \subset \R^3}$ be a bounded domain of class  $C^{1,1}$. 
We study the electromagnetic resonator filling the domain $\O$.
Suppose that the dielectric permittivity and the magnetic permeability are given by the matrix-valued functions 
$\eta^\eps(\x) = \eta(\eps^{-1}\x)$ and $\mu^\eps(\x) = \mu(\eps^{-1}\x)$.

The intensities of the electric and magnetic fields are denoted by  $\u_\eps(\x)$ and $\v_\eps(\x)$, respectively.
The electric and magnetic displacement vectors $\w_\eps$ and $\z_\eps$ are given by  
$\w_\eps(\x) = \eta^\eps(\x) \u_\eps(\x)$, $\z_\eps(\x)= \mu^\eps(\x) \v_\eps(\x)$. 

We write the Maxwell operator  $\M_\eps$ in terms of the displacement vectors.
This operator acts in the space $J(\O) \oplus J_0(\O)$ viewed as a subspace of
 the weighted space
$$
L_2\big(\O;\C^3; (\eta^\eps)^{-1}\big) \oplus L_2\big(\O;\C^3; (\mu^\eps)^{-1}\big),
$$
and is given by 
\begin{equation*}
\M_\eps = \begin{pmatrix} 0 & i \rot (\mu^\eps)^{-1} \cr -i \rot (\eta^\eps)^{-1} & 0 \end{pmatrix}
\end{equation*}
on the domain
\begin{equation}\label{Dom_M_eps}
\begin{aligned}
\Dom \M_\eps = &\big\{ (\w,\z) \in J(\O) \oplus J_0(\O): \ \rot (\eta^\eps)^{-1} \w \in L_2(\O;\C^3),
\\ 
&\rot (\mu^\eps)^{-1} \z \in L_2(\O;\C^3),
\ ((\eta^\eps)^{-1} \w)_\tau \vert_{\partial \O}=0 \big\}. 
\end{aligned}
\end{equation}
Here the boundary condition for  $\w$ is understood in the sense of Definition~\ref{def2}. 
Note that in general $\Dom \M_\eps$ is not contained in $H^1(\O;\C^6)$, since the coefficients are not assumed to be smooth.

The operator $\M_\eps$ is selfadjoint; see \cite{BS1,BS2}. Therefore, $\lambda =i$ is a regular point of the operator $\M_\eps$. 
\textit{Our goal} is to study the behavior of the resolvent $(\M_\eps - i I)^{-1}$.
In other words, we are interested in the behavior of the solutions $(\w_\eps,\z_\eps)$ of the equation  
\begin{equation}\label{M1}
( \M_\eps - iI) \begin{pmatrix} \w_\eps \cr \z_\eps \end{pmatrix} = \begin{pmatrix} \q \cr \r \end{pmatrix}, \quad \q \in J(\O),\ \r \in J_0(\O),
\end{equation}
and  the fields $\u_\eps = (\eta^\eps)^{-1}\w_\eps$, $\v_\eps = (\mu^\eps)^{-1}\z_\eps$.  
In details, the Maxwell system  \eqref{M1} takes the form 
\begin{equation*}
\left\{\begin{matrix}
& i \, \rot (\mu^\eps)^{-1} \z_\eps - i \w_\eps = \q,
\cr
& -i \, \rot (\eta^\eps)^{-1} \w_\eps - i \z_\eps = \r,
\cr
& \div \w_\eps=0,\ \div \z_\eps =0,
\cr
& ((\eta^\eps)^{-1} \w_\eps)_\tau\vert_{\partial \O} =0,\  (\z_\eps)_n \vert_{\partial \O} =0.
\end{matrix}
\right.
\end{equation*}

\begin{remark}
Instead of $\lambda =i$, one could take any other regular point for the operator ${\mathcal M}_\eps$.
\end{remark}

\subsection{The effective matrices $\eta^0$ and  $\mu^0$\label{eff_mat}}
To define the effective matrix $\eta^0$, we consider the auxiliary problem on the cell $\Omega$. Let $\e_1,\e_2,\e_3$ be the standard orthonormal basis in $\R^3$. Let $\Phi_j\in \widetilde{H}^1(\Omega)$ be the periodic solution of the problem 
\begin{equation}
\label{eff1}
\div \eta(\x) (\nabla \Phi_j(\x) + \e_j)=0, \quad \int_\Omega \Phi_j(\x)\,d\x=0.
\end{equation}
(The solution is understood in the weak sense.) 
Let $Y_\eta(\x)$ be the $(3\times 3)$-matrix with the columns $\nabla \Phi_j(\x)$, $j=1,2,3$. We put
\begin{equation}
\label{eff2}
\widetilde{\eta}(\x) := \eta(\x) (Y_\eta(\x)+\1).
\end{equation}
The effective matrix $\eta^0$ is defined by 
\begin{equation*}
\eta^0 := |\Omega|^{-1} \int_{\Omega} \wt{\eta}(\x)\,d\x. 
\end{equation*}
It turns out that the matrix $\eta^0$ is positive. We also need to define the matrix
\begin{equation*}
G_\eta(\x) := \wt{\eta}(\x) (\eta^0)^{-1} - \1. 
\end{equation*}

The positive effective matrix $\mu^0$ is defined in a similar way.
Let $\Psi_j\in \widetilde{H}^1(\Omega)$ be the periodic solution of the problem
\begin{equation}
\label{eff4}
\div \mu(\x) (\nabla \Psi_j(\x) + \e_j)=0, \quad \int_\Omega \Psi_j(\x)\,d\x=0.
\end{equation}
Let $Y_\mu(\x)$ be the $(3\times 3)$-matrix with the columns $\nabla \Psi_j(\x)$, $j=1,2,3$. Denote
\begin{equation}
\label{eff5}
\widetilde{\mu}(\x) := \mu(\x) (Y_\mu(\x)+\1).
\end{equation}
The effective matrix $\mu^0$ is given by 
\begin{equation}
\label{eff6}
\mu^0 := |\Omega|^{-1} \int_{\Omega} \wt{\mu}(\x)\,d\x. 
\end{equation}
We also define the matrix
\begin{equation*}
G_\mu(\x) := \wt{\mu}(\x) (\mu^0)^{-1} - \1. 
\end{equation*}

Let us mention some properties of the effective matrices and the properties of the solutions of problems \eqref{eff1} and \eqref{eff4}.

\begin{remark}
\label{rem2.1}
{\rm 1)} The following inequalities for the effective matrices are known as the Voigt--Reuss bracketing{\rm :}
$$
\underline{\eta} \leqslant \eta^0 \leqslant \overline{\eta}, \quad 
\underline{\mu} \leqslant \mu^0 \leqslant \overline{\mu}.
$$
This implies 
$$
|\eta^0|  \leqslant \| \eta \|_{L_\infty}, \  | (\eta^0)^{-1}|  \leqslant \| \eta^{-1} \|_{L_\infty};
\quad |\mu^0|  \leqslant \| \mu \|_{L_\infty},\ | (\mu^0)^{-1} |  \leqslant \| \mu^{-1} \|_{L_\infty}.
$$
{\rm 2)} The matrix-valued functions $Y_\eta$, $G_\eta$, $Y_\mu$, and $G_\mu$ are periodic and have zero mean values.

\noindent{\rm 3)} It is easy to check that
\begin{equation}
\label{eff6b}
\begin{aligned}
\| Y_\eta \|_{L_2(\Omega)}  \leqslant \| \eta \|_{L_\infty}^{1/2} \|\eta^{-1} \|_{L_\infty}^{1/2} |\Omega|^{1/2},
\\
\| Y_\mu\|_{L_2(\Omega)}  \leqslant \|\mu \|_{L_\infty}^{1/2} \|\mu^{-1} \|_{L_\infty}^{1/2} |\Omega|^{1/2}, 
\end{aligned}
\end{equation}
\begin{equation}
\label{efff}
\begin{aligned}
\| \Phi_j \|_{L_2(\Omega)}  \leqslant (2r_0)^{-1} \| \eta \|_{L_\infty}^{1/2} \|\eta^{-1} \|_{L_\infty}^{1/2} |\Omega|^{1/2},
\\
\| \Psi_j\|_{L_2(\Omega)}  \leqslant (2r_0)^{-1} \|\mu \|_{L_\infty}^{1/2} \|\mu^{-1} \|_{L_\infty}^{1/2} |\Omega|^{1/2}.
\end{aligned}
\end{equation}

\noindent{\rm 4)} According to \cite[Chapter~3, Theorem 3.1]{LaUr}, the periodic solution $\Phi_j$ of problem 
\eqref{eff1} and the periodic solution $\Psi_j$ of  problem \eqref{eff4} are bounded and satyisfy estimates
$$
\| \Phi_j\|_{L_\infty} \leqslant \wh{C}_\eta, \quad \| \Psi_j\|_{L_\infty} \leqslant \wh{C}_\mu, \quad j=1,2,3.
$$
The constant $\wh{C}_\eta$ depends only on $\|\eta\|_{L_\infty}$, $\| \eta^{-1}\|_{L_\infty}$, and $\Omega${\rm ;}
the constant $\wh{C}_\mu$ depends only on $\|\mu \|_{L_\infty}$, $\| \mu^{-1}\|_{L_\infty}$, and $\Omega$.
\end{remark}

We need the following  property of the matrix-valued functions $Y_\eta^\eps$ and $Y_\mu^\eps$;
see \cite[Corollary 2.4]{PSu}.

\begin{lemma}\label{lem_PSu}
For any function $u \in H^1(\R^3)$ and $\eps>0$ we have
\begin{align*}
\int_{\R^3} |Y_\eta^\eps(\x)|^2 |u(\x)|^2\,d\x \leqslant \beta_{1,\eta} \int_{\R^3} |u(\x)|^2 \,d\x +
\beta_{2,\eta} \eps^2 \wh{C}_\eta^2 \int_{\R^3} |\nabla u(\x)|^2 \,d\x,
\\
\int_{\R^3} |Y_\mu^\eps(\x)|^2 |u(\x)|^2\,d\x \leqslant \beta_{1,\mu} \int_{\R^3} |u(\x)|^2 \,d\x +
\beta_{2,\mu} \eps^2 \wh{C}_\mu^2 \int_{\R^3} |\nabla u(\x)|^2 \,d\x.
\end{align*}
The constants $\beta_{1,\eta}$ and $\beta_{2,\eta}$ depend only on $\|\eta\|_{L_\infty}$ and $\|\eta^{-1}\|_{L_\infty}${\rm ;} the constants
$\beta_{1,\mu}$ and $\beta_{2,\mu}$ depend only on $\|\mu\|_{L_\infty}$ and $\|\mu^{-1}\|_{L_\infty}$.
\end{lemma}

Lemma \ref{lem_PSu} implies that the matrix-valued functions $Y_\eta^\eps$ and $Y_\mu^\eps$ are multipliers from the Sobolev space
$H^1(\R^3;\C^3)$ to $L_2(\R^3;\C^3)$.

\subsection{The effective Maxwell operator\label{sec_eff}}
\textit{The effective Maxwell operator} $\M^0$ acts in the space $J(\O) \oplus J_0(\O)$ viewed as a subspace of the weighted space
$$
L_2\big(\O;\C^3; (\eta^0)^{-1}\big) \oplus L_2\big(\O;\C^3; (\mu^0)^{-1}\big),
$$
and is given by
\begin{equation*}
\M^0 = \begin{pmatrix} 0 & i \rot (\mu^0)^{-1} \cr -i \rot (\eta^0)^{-1} & 0 \end{pmatrix}
\end{equation*}
on the domain
\begin{equation*}
\begin{aligned}
\Dom \M^0 = &\big\{ (\w,\z) \in J(\O) \oplus J_0(\O): \ \rot (\eta^0)^{-1} \w \in L_2(\O;\C^3),
\\ 
&\rot (\mu^0)^{-1} \z \in L_2(\O;\C^3),
\ ((\eta^0)^{-1} \w)_\tau \vert_{\partial \O}=0 \big\}. 
\end{aligned}
\end{equation*}
Since $\partial \O \in C^{1,1}$, the domain $\Dom \M^0$ is contained in $H^1(\O;\C^6)$ and can be represented as
\begin{equation}\label{Dom_M^0_2}
\begin{aligned}
\Dom \M^0 = &\big\{ (\w,\z) \in H^1(\O;\C^6): \ \div \w =0, \ \div \z =0,
\\ 
&
((\eta^0)^{-1} \w)_\tau \vert_{\partial \O}=0,\  \z_n \vert_{\partial \O}=0 \big\}. 
\end{aligned}
\end{equation}
Here the boundary conditions on $\w$ and $\z$ are understood in the sense of trace theorem.
This property was proved in  \cite[Theorem 2.3]{BS1} under the assumption that $\partial \O \in C^2$ and in  
\cite[Theorem 2.6]{F} under the assumption that $\partial \O \in C^{3/2+\delta}$, $\delta >0$.

We consider the effective Maxwell system
\begin{equation}\label{M1eff}
( \M^0 - iI) \begin{pmatrix} \w_0 \cr \z_0 \end{pmatrix} = \begin{pmatrix} \q \cr \r \end{pmatrix},
\end{equation}
and put $\u_0 = (\eta^0)^{-1}\w_0$, $\v_0 = (\mu^0)^{-1}\z_0$.  
In details, the Maxwell system \eqref{M1eff} is given by
\begin{equation*}
\left\{\begin{matrix}
& i \, \rot (\mu^0)^{-1} \z_0 - i \w_0 = \q,
\cr
& -i \, \rot (\eta^0)^{-1} \w_0 - i \z_0 = \r,
\cr
& \div \w_0 =0,\ \div \z_0 =0,
\cr
& ((\eta^0)^{-1} \w_0)_\tau\vert_{\partial \O} =0,\  (\z_0)_n \vert_{\partial \O} =0.
\end{matrix}
\right.
\end{equation*}

The classical results (see \cite{BeLPap, Sa, ZhKO}) show that, as $\eps \to 0$,
the vector-valued functions $\u_\eps, \w_\eps, \v_\eps, \z_\eps$ 
\textit{weakly converge} in $L_2(\O;\C^3)$ to the corresponding homogenized fields $\u_0, \w_0, \v_0, \z_0$.

\subsection{Main results}
We find approximations for the filelds 
$\u_\eps, \w_\eps, \v_\eps, \z_\eps$ in the $L_2(\O;\C^3)$-norm. 
To formulate the results, we need one more  Maxwell system
\begin{equation}\label{M_corr}
( \M^0 - iI) \begin{pmatrix} \wh{\w}_\eps \cr \wh{\z}_\eps \end{pmatrix} = \begin{pmatrix} \q_\eps \cr \r_\eps \end{pmatrix},
\end{equation}
which is called the ``correction'' Maxwell system.
This system has effective coefficients, but the vector-valued functions $\q_\eps$ and $\r_\eps$ in the right-hand side
depend on $\eps$. They are defined as follows.
We extend the functions $\q$ and $\r$ by zero to $\R^3 \setminus \O$:
$$
\wt{\q}(\x) = \begin{cases} \q(\x), & \x \in \O,
\\ 0, &\x \in \R^3 \setminus \O,
\end{cases}
\quad
\wt{\r}(\x) = \begin{cases} \r(\x), & \x \in \O,
\\ 0, &\x \in \R^3 \setminus \O.
\end{cases}
$$
Next, consider the vector-valued functions $S_\eps(Y_\eta^\eps)^* \wt{\q}$ and $S_\eps(Y_\mu^\eps)^* \wt{\r}$. These functions belong to 
 $L_2(\R^3;\C^3)$, since the operators $S_\eps[(Y_\eta^\eps)^*] = ([Y_\eta^\eps] S_\eps)^*$ 
and $S_\eps[(Y_\mu^\eps)^*] = ([Y_\mu^\eps] S_\eps)^*$ are continuous in $L_2(\R^3;\C^3)$ due to Proposition 
\ref{prop f^eps S_eps}.
Let  ${\mathcal P}_{\eta^0}$ be the orthogonal projection of $L_2(\O;(\eta^0)^{-1})$ onto $J(\O)$
and let ${\mathcal P}^0_{\mu^0}$ be the orthogonal projection of  $L_2(\O;(\mu^0)^{-1})$ onto $J_0(\O)$.
Restricting the functions  $S_\eps(Y_\eta^\eps)^* \wt{\q}$ and $S_\eps(Y_\mu^\eps)^* \wt{\r}$ to the domain $\O$
and applying the projections ${\mathcal P}_{\eta^0}$ and ${\mathcal P}^0_{\mu^0}$, respectively, we define the functions
\begin{equation}
\label{right}
\q_\eps := {\mathcal P}_{\eta^0} S_\eps(Y_\eta^\eps)^* \wt{\q}, 
\quad 
\r_\eps := {\mathcal P}^0_{\mu^0} S_\eps(Y_\mu^\eps)^* \wt{\r}.
\end{equation}
Thus, $\q_\eps \in J(\O)$ and $\r_\eps \in J_0(\O)$.
Using Proposition  \ref{prop f^eps S_eps} and inequalities \eqref{eff6b},
we have
\begin{equation}
\label{right_est}
\begin{aligned}
\| \q_\eps \|_{L_2(\O)} \leqslant \|\eta\|_{L_\infty(\Omega)}  \|\eta^{-1}\|_{L_\infty(\Omega)}
\|\q\|_{L_2(\O)}, 
\\
\| \r_\eps \|_{L_2(\O)} \leqslant \|\mu \|_{L_\infty(\Omega)}  \| \mu^{-1} \|_{L_\infty(\Omega)}
\|\r \|_{L_2(\O)}.
\end{aligned}
\end{equation}

In terms of the solutions of system \eqref{M_corr}, we define the ``correction'' fields
\begin{equation}
\label{corr_u_v}
\wh{\u}_\eps = (\eta^0)^{-1} \wh{\w}_\eps, 
\quad 
\wh{\v}_\eps = (\mu^0)^{-1} \wh{\z}_\eps.
\end{equation}
By \eqref{Dom_M^0_2}, we have $\wh{\u}_\eps, \wh{\w}_\eps, \wh{\v}_\eps, \wh{\z}_\eps \in H^1(\O;\C^3)$.

 \begin{remark}\label{rem2.2}
 The ``correction'' fields $\wh{\u}_\eps, \wh{\w}_\eps, \wh{\v}_\eps, \wh{\z}_\eps$ weakly converge to zero in  $L_2(\O;\C^3)$, as $\eps \to 0$. 
 This can be easily checked by using the ``mean value property'' and the fact that the right-hand sides $\q_\eps, \r_\eps$ 
 of system \eqref{M_corr} contain rapidly oscillating factors with zero mean values{\rm ;} see \eqref{right}. 
 \end{remark}

Our main result is the following theorem. 

\begin{theorem}\label{main_th}
Let $(\w_\eps, \z_\eps)$ be the solution of system \eqref{M1}. Let $\u_\eps = (\eta^\eps)^{-1}\w_\eps$ and  $\v_\eps = (\mu^\eps)^{-1}\z_\eps$. 
Let $(\w_0, \z_0)$ be the solution of the effective system \eqref{M1eff} and let 
$\u_0 = (\eta^0)^{-1}\w_0$,  $\v_0 = (\mu^0)^{-1}\z_0$. Let  
$(\wh{\w}_\eps, \wh{\z}_\eps)$ be the solution of the correction system \eqref{M_corr}, and let $\wh{\u}_\eps$,
$\wh{\v}_\eps$ be given by  \eqref{corr_u_v}. Let $Y_\eta$, $G_\eta$, $Y_\mu$, and $G_\mu$ be the periodic matrix-valued functions
defined in Subsection  {\rm \ref{eff_mat}}. 
Suppose that $\eps_1$ is subject to Condition~{\rm \ref{cond1}}.
Then for $0 < \eps \le \eps_1$ we have 
\begin{align}
\label{main_1}
\| \u_\eps - (\1 + Y_\eta^\eps)(\u_0 + \wh{\u}_\eps) \|_{L_2(\O)} &\le C_1 \eps^{1/2} \left( \| \q \|_{L_2(\O)} + 
\| \r \|_{L_2(\O)}\right),
\\
\label{main_2}
\| \w_\eps - (\1 + G_\eta^\eps)(\w_0 + \wh{\w}_\eps) \|_{L_2(\O)} &\le C_2 \eps^{1/2} 
\left( \| \q \|_{L_2(\O)} + \| \r \|_{L_2(\O)}\right),
\\
\label{main_3}
\| \v_\eps - (\1 + Y_\mu^\eps)(\v_0 + \wh{\v}_\eps) \|_{L_2(\O)} &\le C_3 \eps^{1/2} \left( \| \q \|_{L_2(\O)} + 
\| \r \|_{L_2(\O)}\right),
\\
\label{main_4}
\| \z_\eps - (\1 + G_\mu^\eps)(\z_0 + \wh{\z}_\eps) \|_{L_2(\O)} &\le C_4 \eps^{1/2} \left( \| \q \|_{L_2(\O)} + 
\| \r \|_{L_2(\O)}\right).
\end{align}
The constants $C_1, C_2, C_3, C_4$ depend on the norms $\|\eta\|_{L_\infty}$, $\|\eta^{-1}\|_{L_\infty}$,
$\|\mu\|_{L_\infty}$, $\|\mu^{-1}\|_{L_\infty}$, the parameters of the lattice $\Gamma$, and the domain $\O$. 
\end{theorem}

\begin{remark}
{\rm 1)} Note that $(\1 + Y_\eta^\eps)(\u_0 + \wh{\u}_\eps) \in L_2(\O;\C^3)$, since  
$\u_0 + \wh{\u}_\eps \in H^1(\O;\C^3)$, and the matrix-valued function
$Y_\eta^\eps$ is a multiplier from  $H^1(\O;\C^3)$ to $L_2(\O;\C^3)${\rm ;} 
see Lemma {\rm \ref{lem_PSu}}. Similarly, approximations for the fields $\w_\eps, \v_\eps, \z_\eps$ also belong to $L_2(\O;\C^3)$.

\noindent {\rm 2)} Approximations for $\u_\eps$, $\w_\eps$, $\v_\eps$, and $\z_\eps$ are similar to each other.
For instance, the field $\u_\eps$ is approximated by the sum of four terms{\rm :}
$$
\u_\eps \sim \u_0  + Y_\eta^\eps \u_0 +\wh{\u}_\eps + Y_\eta^\eps \wh{\u}_\eps.
$$
Here the first term is the homogenized field, and other three terms weakly tend to zero and can be interpreted as correctors of zero order. 

\noindent {\rm 3)} The order of estimates from Theorem~{\rm \ref{main_th}} is worse that the order of similar estimates in $\R^3${\rm ;} 
this is explained by the boundary influence.  

\noindent {\rm 4)} The result of Theorem~{\rm \ref{main_th}} can be formulated in the operator terms:
$$
\| (\M_\eps - i I)^{-1} - (I + {\mathcal G}^\eps) ( \M^0 - i I)^{-1} (I + {\mathcal Z}_\eps)\|
\le C \eps^{1/2},
$$
where
$$
{\mathcal G}^\eps = \begin{pmatrix} G_\eta^\eps & 0 \\  0 & G_\mu^\eps \end{pmatrix},
\quad 
{\mathcal Z}_\eps = \begin{pmatrix} {\mathcal P}_{\eta^0} S_\eps (Y_\eta^\eps)^* \Pi & 0 \\  0 & 
{\mathcal P}^0_{\mu^0} S_\eps (Y_\mu^\eps)^* \Pi  \end{pmatrix},
$$
and $\Pi$ is the operator of extension by zero. 
\end{remark}

\section{Reduction of the problem to the second order equations}

We represent the solution of  system \eqref{M1} as 
\begin{equation*}
\w_\eps = \w_\eps^{(\q)} + \w_\eps^{(\r)}, \quad \z_\eps = \z_\eps^{(\q)} + \z_\eps^{(\r)},
\end{equation*}
where $(\w_\eps^{(\q)},  \z_\eps^{(\q)})$ is the solution of system \eqref{M1} with $\r=0$, and
$(\w_\eps^{(\r)},  \z_\eps^{(\r)})$ is the solution of  system \eqref{M1} with $\q =0$.
We put 
\begin{equation*}
\begin{aligned}
\u_\eps^{(\q)} = (\eta^\eps)^{-1}\w_\eps^{(\q)},
\quad  
\v_\eps^{(\q)} = (\mu^\eps)^{-1}\z_\eps^{(\q)},
\\
\u_\eps^{(\r)} = (\eta^\eps)^{-1}\w_\eps^{(\r)},
\quad   
\v_\eps^{(\r)} = (\mu^\eps)^{-1}\z_\eps^{(\r)}.
\end{aligned}
\end{equation*}

Then  $(\w_\eps^{(\r)},  \z_\eps^{(\r)})$ is the solution of the problem 
\begin{equation}\label{Mr}
\left\{\begin{matrix}
& \w_\eps^{(\r)} = \rot (\mu^\eps)^{-1} \z_\eps^{(\r)},
\cr
& \rot (\eta^\eps)^{-1} \w_\eps^{(\r)} + \z_\eps^{(\r)} = i \r,
\cr
& \div \w_\eps^{(\r)} =0,\ \div \z_\eps^{(\r)} =0,
\cr
& ((\eta^\eps)^{-1} \w_\eps^{(\r)})_\tau\vert_{\partial \O} =0,\  (\z_\eps^{(\r)})_n \vert_{\partial \O} =0.
\end{matrix}
\right.
\end{equation}
The pair  $(\w_\eps^{(\q)},  \z_\eps^{(\q)})$ is the solution of the problem 
\begin{equation}\label{Mq}
\left\{\begin{matrix}
& \z_\eps^{(\q)} = - \rot (\eta^\eps)^{-1} \w_\eps^{(\q)},
\cr
& \rot (\mu^\eps)^{-1} \z_\eps^{(\q)} - \w_\eps^{(\q)} = - i \q,
\cr
& \div \w_\eps^{(\q)} =0,\ \div \z_\eps^{(\q)} =0,
\cr
& ((\eta^\eps)^{-1} \w_\eps^{(\q)})_\tau\vert_{\partial \O} =0,\  (\z_\eps^{(\q)})_n \vert_{\partial \O} =0.
\end{matrix}
\right.
\end{equation}
From \eqref{Mr} it follows that $\z_\eps^{(\r)}$ is the solution of the following boundary value problem\begin{equation}\label{Mrr}
\left\{\begin{matrix}
& \rot (\eta^\eps)^{-1} \rot (\mu^\eps)^{-1} \z_\eps^{(\r)}  + \z_\eps^{(\r)} = i \r, \quad \div \z_\eps^{(\r)} =0,
\cr
& (\z_\eps^{(\r)})_n \vert_{\partial \O} =0, 
\  ((\eta^\eps)^{-1} \rot (\mu^\eps)^{-1} \z_\eps^{(\r)})_\tau\vert_{\partial \O} =0.
\end{matrix}
\right.
\end{equation}
By \eqref{Mq}, the function $\w_\eps^{(\q)}$ is the solution of the problem
\begin{equation}\label{Mqq}
\left\{\begin{matrix}
& \rot (\mu^\eps)^{-1} \rot (\eta^\eps)^{-1} \w_\eps^{(\q)}  + \w_\eps^{(\q)} = i \q,
\quad \div \w_\eps^{(\q)} =0,
\cr
&  ( (\eta^\eps)^{-1} \w_\eps^{(\q)})_\tau \vert_{\partial \O} =0, 
\  (\rot (\eta^\eps)^{-1} \w_\eps^{(\q)})_n\vert_{\partial \O} =0.
\end{matrix}
\right.
\end{equation}
Problems \eqref{Mrr} and \eqref{Mqq} are similar to each other, however,  the role of the coefficients $\eta^\eps$ and $\mu^\eps$,
the boundary conditions, and the conditions on the right-hand sides ($\r \in J_0(\O)$, $\q \in J(\O)$) are different. 
We study these problems separately and next combine the results. 

The effective fields $\u_0, \w_0, \v_0, \z_0$ and the correction fields $\wh{\u}_\eps, \wh{\w}_\eps, \wh{\v}_\eps, 
\wh{\z}_\eps$ are also represented as the sum of two terms, with indices $\q$ and $\r$ respectively.
The terms with index $\q$ correspond to the case $\r=0$, 
and the terms with index $\r$ correspond to the case $\q=0$.

\section{The study of the problem with  $\q=0$}

\subsection{Symmetrization\label{sim}}
Putting $\bvarphi_\eps = (\mu^\eps)^{-1/2}\z_\eps^{(\r)}$, we reduce  problem \eqref{Mrr} to  
\begin{equation}\label{Mrr1}
\left\{\begin{matrix}
& (\mu^\eps)^{-1/2} \rot (\eta^\eps)^{-1} \rot (\mu^\eps)^{-1/2} \bvarphi_\eps  + \bvarphi_\eps = 
i (\mu^\eps)^{-1/2} \r,   \ \div (\mu^\eps)^{1/2} \bvarphi_\eps =0,
\\ 
&((\mu^\eps)^{1/2} \bvarphi_\eps)_n \vert_{\partial \O} =0, 
\  ((\eta^\eps)^{-1} \rot (\mu^\eps)^{-1/2} \bvarphi_\eps)_\tau\vert_{\partial \O} =0.
\end{matrix}
\right.
\end{equation}
Here $\r \in J_0(\O)$. 
Automatically, $\bvarphi_\eps$ is also the solution of the elliptic equation
\begin{equation}\label{Mrr2}
({\mathcal L}_\eps +I) \bvarphi_\eps = i (\mu^\eps)^{-1/2} \r,
\end{equation}
where the operator ${\mathcal L}_\eps$ is formally given by the differential expression 
\begin{equation*}
{\mathcal L}_\eps =  (\mu^\eps)^{-1/2} \rot (\eta^\eps)^{-1} \rot (\mu^\eps)^{-1/2} - 
(\mu^\eps)^{1/2} \nabla \div (\mu^\eps)^{1/2}
\end{equation*}
with the boundary conditions from \eqref{Mrr1}. Strictly speaking, ${\mathcal L}_\eps$ is the selfadjoint operator in $L_2(\O;\C^3)$ 
corresponding to the closed nonnegative quadratic form
\begin{equation}\label{forma}
\begin{aligned}
{\mathfrak l}_\eps[\bvarphi, \bvarphi] = & \int_\O \left( \langle (\eta^\eps)^{-1} \rot (\mu^\eps)^{-1/2} \bvarphi, 
\rot (\mu^{\eps})^{-1/2} \bvarphi \rangle + | \div (\mu^{\eps})^{1/2} \bvarphi |^2 \right)\, d\x,
\\
\Dom {\mathfrak l}_\eps =\{& \bvarphi \in L_2(\O;\C^3): \quad  \div (\mu^{\eps})^{1/2} \bvarphi \in L_2(\O),
\\ 
&\rot (\mu^{\eps})^{-1/2} \bvarphi \in L_2(\O;\C^3), \quad 
 ((\mu^\eps)^{1/2} \bvarphi)_n \vert_{\partial \O} =0\}.
\end{aligned}
\end{equation}
From the results of \cite{BS1, BS2} it follows that the form \eqref{forma} is closed.

\begin{remark}
{\rm 1)} In general, $\operatorname{Dom} {\mathfrak l}_\eps \not\subset H^1(\O;\C^3)$.

\noindent{\rm 2)} The second boundary condition in \eqref{Mrr1} is natural, it is not reflected in the domain of the quadratic form ${\mathfrak l}_\eps$. 

\noindent {\rm 3)} 
The form ${\mathfrak l}_\eps$ and the operator ${\mathcal L}_\eps$ are reduced by the orthogonal decomposition
$$
L_2(\O;\C^3) = {\mathcal J}_0(\O;\mu^\eps) \oplus {\mathcal G}(\O;\mu^\eps),
$$
where
$$
\begin{aligned}
{\mathcal J}_0(\O;\mu^\eps) &= \{ \f: \ (\mu^\eps)^{1/2} \f \in J_0(\O) \}, 
\\
 {\mathcal G}(\O;\mu^\eps) &= \{ (\mu^\eps)^{1/2} \nabla \omega: \ \omega \in H^1(\O) \}.
\end{aligned}
$$
\end{remark}

\subsection{The effective problem\label{sec4.2}}
Let $\eta^0$ and $\mu^0$ be the effective matrices defined in Subsection~\ref{eff_mat}.
We put $\bvarphi_0 = (\mu^0)^{-1/2} \z_0^{(\r)}$. Then $\bvarphi_0$ is the solution of the problem 
\begin{equation}\label{Mrr4}
\left\{\begin{matrix}
& (\mu^0)^{-1/2} \rot (\eta^0)^{-1} \rot (\mu^0)^{-1/2} \bvarphi_0  + \bvarphi_0 = 
i (\mu^0)^{-1/2} \r, \ \div (\mu^0)^{1/2} \bvarphi_0 =0,
\\ 
&((\mu^0)^{1/2} \bvarphi_0)_n \vert_{\partial \O} =0, 
\  ((\eta^0)^{-1} \rot (\mu^0)^{-1/2} \bvarphi_0)_\tau\vert_{\partial \O} =0.
\end{matrix}
\right.
\end{equation}
Automatically,  $\bvarphi_0$ is also the solution of the elliptic equation 
\begin{equation}\label{Mrr5}
({\mathcal L}^0 +I) \bvarphi_0 = i (\mu^0)^{-1/2} \r,
\end{equation}
where ${\mathcal L}^0$ is the selfadjoint operator in  $L_2(\O;\C^3)$ corresponding to the closed nonnegative quadratic form
\begin{equation}
\label{forma_eff}
\begin{aligned}
{\mathfrak l}_0 [\bvarphi, \bvarphi] =& \int_\O \left( \langle (\eta^0)^{-1} \rot (\mu^0)^{-1/2} \bvarphi, 
\rot (\mu^{0})^{-1/2} \bvarphi \rangle + | \div (\mu^{0})^{1/2} \bvarphi |^2 \right)\, d\x,
\\
 \Dom {\mathfrak l}_0 =& \{ \bvarphi \in L_2(\O;\C^3):\  
\div (\mu^{0})^{1/2} \bvarphi \in L_2(\O),
\\ 
&\rot (\mu^{0})^{-1/2} \bvarphi \in L_2(\O;\C^3), \ 
 ((\mu^0)^{1/2} \bvarphi)_n \vert_{\partial \O} =0\}.
\end{aligned}
\end{equation}
Since $\partial \O \in C^{1,1}$, the set $\Dom {\mathfrak l}_0$ coincides with
$$
\Dom {\mathfrak l}_0 = \{ \bvarphi \in H^1(\O;\C^3):   \  ((\mu^0)^{1/2} \bvarphi)_n \vert_{\partial \O} =0\}.
$$
The form \eqref{forma_eff} is coercive: the following two-sided estimates hold:
\begin{equation}\label{coercive}
{\mathfrak c}_1  \| \bvarphi \|^2_{H^1(\O)} \leqslant
{\mathfrak l}_0 [\bvarphi, \bvarphi] + \| \bvarphi \|^2_{L_2(\O)} \leqslant 
{\mathfrak c}_2  \| \bvarphi \|^2_{H^1(\O)},\quad \bvarphi \in \Dom {\mathfrak l}_0. 
\end{equation}
The constant  ${\mathfrak c}_1$ depends on  $\| \mu \|_{L_\infty}$, $\| \mu^{-1} \|_{L_\infty}$,
$\| \eta \|_{L_\infty}$, and the domain $\O$, and ${\mathfrak c}_2$ depends on 
$\| \mu \|_{L_\infty}$, $\| \mu^{-1} \|_{L_\infty}$, $\| \eta^{-1} \|_{L_\infty}$, and the domain $\O$.
These properties were proved in  \cite[Theorem 2.3]{BS1} under the assumption $\partial \O \in C^2$
and in  \cite[Theorem 2.6]{F} under the assumption $\partial \O \in C^{3/2+\delta}$, $\delta>0$.

The operator ${\mathcal L}^0$ is a strongly elliptic operator with constant coefficients.
The smoothness of the boundary ($\partial \O \in C^{1,1}$) ensures the following regularity:
the resolvent $({\mathcal L}^0+I)^{-1}$ is continuous from $L_2(\O;\C^3)$ to $H^2(\O;\C^3)$. We have
\begin{equation}\label{Mrr6}
\| ({\mathcal L}^0+I)^{-1} \|_{L_2(\O) \to H^2(\O)} \leqslant c_*,
\end{equation}
where the constant  $c_*$ depends only on  $\|\eta\|_{L_\infty}$, $\|\eta^{-1}\|_{L_\infty}$,
$\|\mu\|_{L_\infty}$, $\|\mu^{-1}\|_{L_\infty}$, and the domain $\O$.
Thus, the operator ${\mathcal L}^0$ can be given by the differential expression
\begin{equation*}
{\mathcal L}^0 =  (\mu^0)^{-1/2} \rot (\eta^0)^{-1} \rot (\mu^0)^{-1/2} - 
(\mu^0)^{1/2} \nabla \div (\mu^0)^{1/2}
\end{equation*}
on the domain
\begin{equation*}
\Dom {\mathcal L}^0 = \{ \bvarphi \in H^2(\O;\C^3): \ ((\mu^0)^{1/2} \bvarphi)_n \vert_{\partial \O} =0,\ 
((\eta^0)^{-1} \rot (\mu^0)^{-1/2} \bvarphi)_\tau\vert_{\partial \O} =0 \}.
\end{equation*}

\begin{remark}\label{rem4.2}
{\rm 1)} Under the assumption that $\partial \O \in C^{1,1}$ {\rm (}and for sufficiently smooth coefficients{\rm )},
the $H^2$-regularity of the solutions of the Dirichlet or Neumann problems 
 for strongly elliptic second order equations can be found, e.~g., in the book \cite[Chapter~4]{McL}. 
 The proof is based on the method of difference quotients  and relies on the coercivity of the quadratic form. 
 In our case, the coefficients of the operator ${\mathcal L}^0$ are constant and the coercivity condition \eqref{coercive}
 holds. However, the boundary conditions are of mixed type. It is easy to prove the required regularity 
 for the operator  ${\mathcal L}^0$ by the same method. 

\noindent{\rm 2)} The form ${\mathfrak l}_0$ and the operator  
${\mathcal L}^0$ are reduced by the orthogonal decomposition 
$$
L_2(\O;\C^3) = {\mathcal J}_0(\O;\mu^0) \oplus {\mathcal G}(\O;\mu^0),
$$
where 
$$
\begin{aligned}
{\mathcal J}_0(\O;\mu^0) &= \{ \f: \ (\mu^0)^{1/2} \f \in J_0(\O) \}, 
\\
 {\mathcal G}(\O;\mu^0) &= \{ (\mu^0)^{1/2} \nabla \omega: \ \omega \in H^1(\O) \}.
\end{aligned}
$$
\end{remark}

By \eqref{Mrr5} and \eqref{Mrr6}, $\bvarphi_0 \in H^2(\O;\C^3)$ and 
\begin{equation}\label{Mrr9}
   \| \bvarphi_0 \|_{H^2(\O)} \leqslant c_* \| \mu^{-1}\|^{1/2}_{L_\infty} \| \r \|_{L_2(\O)}.
\end{equation}

Let $P_\O: H^2(\O;\C^3) \to H^2(\R^3;\C^3)$ be the linear continuous extension operator. Denote
\begin{equation}\label{Mrr10}
   \| P_\O \|_{H^2(\O) \to H^2(\R^3)} =: C_\O.
   \end{equation}
We put $\wt{\bvarphi}_0 := P_\O \bvarphi_0 \in H^2(\R^3;\C^3)$. According to  \eqref{Mrr9} and \eqref{Mrr10},
\begin{equation}\label{Mrr11}
   \| \wt{\bvarphi}_0 \|_{H^2(\R^3)} \leqslant {\mathcal C}_1 \| \r \|_{L_2(\O)},
\end{equation}
where ${\mathcal C}_1 = C_\O c_* \| \mu^{-1}\|^{1/2}_{L_\infty}$.

\subsection{The correction problem\label{sec4.3}} We put $\brho_\eps = (\mu^0)^{-1/2} \wh{\z}_\eps^{(\r)}$. 
Then $\brho_\eps$ is the solution of the problem 
\begin{equation}\label{Mrr11a}
\left\{\begin{matrix}
& (\mu^0)^{-1/2} \rot (\eta^0)^{-1} \rot (\mu^0)^{-1/2} \brho_\eps   + \brho_\eps  = 
i (\mu^0)^{-1/2} \r_\eps, \ \div (\mu^0)^{1/2} \brho_\eps  =0,
\\ 
&((\mu^0)^{1/2} \brho_\eps )_n \vert_{\partial \O} =0, 
\  ((\eta^0)^{-1} \rot (\mu^0)^{-1/2} \brho_\eps )_\tau\vert_{\partial \O} =0.
\end{matrix}
\right.
\end{equation}
Automatically,  $\brho_\eps$ is also the solution of the elliptic equation
\begin{equation*}
({\mathcal L}^0 +I) \brho_\eps = i (\mu^0)^{-1/2} \r_\eps. 
\end{equation*}
By   \eqref{right_est} and \eqref{Mrr6}, we have $\brho_\eps \in H^2(\O;\C^3)$ and 
\begin{equation}\label{Mrr13}
   \| \brho_\eps \|_{H^2(\O)} \leqslant c_* \| \mu\|_{L_\infty} \| \mu^{-1}\|^{3/2}_{L_\infty} \| \r \|_{L_2(\O)}.
\end{equation}
We put $\wt{\brho}_\eps := P_\O \brho_\eps \in H^2(\R^3; \C^3)$. According to \eqref{Mrr10} and \eqref{Mrr13},
\begin{equation}\label{Mrr14}
   \| \wt{\brho}_\eps \|_{H^2(\R^3)} \leqslant {\mathcal C}_2 \| \r \|_{L_2(\O)},
\end{equation}
where ${\mathcal C}_2 = C_\O c_* \| \mu \|_{L_\infty} \| \mu^{-1}\|^{3/2}_{L_\infty}$.

\subsection{The first order approximation for $\bvarphi_\eps$\label{sec4.4}}
Let $\bvarphi_\eps$ be the solution of equation \eqref{Mrr2}. 
We look for the first order approximation  $\bpsi_\eps$ of the solution $\bvarphi_\eps$ in the form similar to the 
case  of  $\R^3$ (studied in \cite{Su2}). First, we introduce the necessary objects. 
Let $W_\mu^*(\x)$ be the $\Gamma$-periodic $(3 \times 3)$-matrix-valued function given by 
\begin{equation}\label{Mrr15}
   W_\mu^*(\x) = \mu(\x)^{-1/2} \wt{\mu}(\x) (\mu^0)^{-1/2}
   = \mu(\x)^{1/2} (\1 + Y_\mu (\x)) (\mu^0)^{-1/2},
\end{equation}
where $\wt{\mu}(\x)$ is the matrix \eqref{eff5}. Denote $\c_j:= (\mu^0)^{-1/2} \e_j$, $j=1,2,3$.
Let $\check{\Psi}_j(\x)$ be the $\Gamma$-periodic solution of the problem
 \begin{equation*}
   \div \mu(\x) (\nabla \check{\Psi}_j(\x) + \c_j)=0,
   \quad \int_\Omega \check{\Psi}_j(\x)\, d\x =0.
\end{equation*}
Let ${\f}_{lj}(\x)$ (where $l,j=1,2,3$) be the $\Gamma$-periodic solution of the problem
 \begin{equation}\label{Mrr17}
\begin{aligned}
   & \mu(\x)^{-1/2} \rot \eta(\x)^{-1} \left( \rot \mu(\x)^{-1/2} \f_{lj}(\x) + 
   i \e_l \times (\nabla \check{\Psi}_j(\x) + \c_j) \right) 
   \\
   & - \mu(\x)^{1/2} \nabla \left( \div \mu(\x)^{1/2} \f_{lj}(\x) +
   i \e_l \cdot (\mu(\x) (\nabla \check{\Psi}_j(\x) + \c_j) \right) = 0,
   \\
   & \int_\Omega \f_{lj} (\x)\, d\x =0.
   \end{aligned}
\end{equation}
 Let $\Lambda_l(\x)$ (where $l=1,2,3$) be the  $\Gamma$-periodic  $(3\times 3)$-matrix-valued function 
 with the columns $\f_{lj}(\x)$, $j=1,2,3$. Note that
 \begin{equation}\label{Mrr17a}
\| \Lambda_l\|_{L_2(\Omega)} \leqslant C_\Lambda |\Omega|^{1/2},  
\end{equation}
where the constant $C_\Lambda$ depends on  $\|\eta\|_{L_\infty}$, $\|\eta^{-1}\|_{L_\infty}$, $\|\mu\|_{L_\infty}$, $\|\mu^{-1}\|_{L_\infty}$, and the parameters of the lattice $\Gamma$; see  \cite[Subsection 5.3]{Su2}.

\begin{remark}\label{rem4.3}
It is easily seen {\rm (}see \cite[Subsection 4.4]{Su1}{\rm )} that the solution of problem  \eqref{Mrr17} satisfies
\begin{align}
\label{div_f}
&\div \mu(\x)^{1/2} \f_{lj}(\x) = i \e_l \cdot ((\mu^0)^{1/2}\e_j) - i \e_l \cdot (\wt{\mu}(\x) (\mu^0)^{-1/2} \e_j),
\\
\label{rot_f}
\begin{split}
&\eta(\x)^{-1} \rot \mu(\x)^{-1/2} \f_{lj}(\x) = 
i (\1+Y_\eta(\x)) (\eta^0)^{-1} \left( \e_l \times ((\mu^0)^{-1/2}\e_j) \right) 
\\
&+ i \eta(\x)^{-1} \left( \left((\1+Y_\mu(\x)) (\mu^0)^{-1/2} \e_j \right) \times \e_l\right).
\end{split}
\end{align}
\end{remark}

Suppose that the functions  $\wt{\bvarphi}_0, \wt{\brho}_\eps \in H^2(\R^3;\C^3)$  are introduced in Subsections \ref{sec4.2} and \ref{sec4.3}, respectively. 
Let $S_\eps$ be the Steklov smoothing operator (see \eqref{S_eps}).
We look for the first order approximation $\bpsi_\eps$ of the solution $\bvarphi_\eps$ of equation \eqref{Mrr2} in the form
\begin{equation}\label{Mrr18}
\begin{aligned}
   & \wt{\bpsi}_\eps = (W_\mu^\eps)^{*} S_\eps (\wt{\bvarphi}_0 + \wt{\brho}_\eps)
   + \eps \sum_{l=1}^3 \Lambda_l^\eps S_\eps D_l  (\wt{\bvarphi}_0 + \wt{\brho}_\eps),  
   \\
   & \bpsi_\eps = \wt{\bpsi}_\eps \vert_{\O}.
   \end{aligned}
\end{equation}
Expression \eqref{Mrr18} is prompted by the form of the first order approximation for the solution of the 
similar equation in $\R^3$; see \cite{Su2}.

\begin{lemma}\label{lem4.3}
We have
$\bpsi_\eps \in L_2(\O;\C^3)$, $\rot (\mu^\eps)^{-1/2} \bpsi_\eps \in L_2(\O;\C^3)$,
$\div (\mu^\eps)^{1/2} \bpsi_\eps \in L_2(\O)$, and
\begin{align}
\label{lem1}
&\| {\bpsi}_\eps - (W_\mu^\eps)^{*} S_\eps (\wt{\bvarphi}_0 + \wt{\brho}_\eps) \|_{L_2(\O)}
\leqslant {\mathcal C}_3 \eps \| \r \|_{L_2(\O)},
\\
\label{lem2}
\begin{split}
&\| (\eta^\eps)^{-1}\rot (\mu^\eps)^{-1/2}{\bpsi}_\eps - 
(\1 + Y_\eta^\eps) (\eta^0)^{-1} \rot (\mu^0)^{-1/2}S_\eps 
 (\wt{\bvarphi}_0 + \wt{\brho}_\eps) \|_{L_2(\O)}
\leqslant {\mathcal C}_4 \eps \| \r \|_{L_2(\O)},
\end{split}
\\
\label{lem3a}
&\| \div (\mu^\eps)^{1/2}{\bpsi}_\eps  \|_{L_2(\O)}
\leqslant {\mathcal C}_5 \eps \| \r \|_{L_2(\O)}.
\end{align}
The constants ${\mathcal C}_3$, ${\mathcal C}_4$, and ${\mathcal C}_5$ depend only on the norms $\|\eta\|_{L_\infty}$, 
$\|\eta^{-1}\|_{L_\infty}$, $\|\mu\|_{L_\infty}$, $\|\mu^{-1}\|_{L_\infty}$, the parameters of the lattice $\Gamma$, and the domain $\O$.
\end{lemma}

\begin{proof}
By Proposition \ref{prop f^eps S_eps} and \eqref{Mrr17a},  \eqref{Mrr18}, 
$$
\| {\bpsi}_\eps - (W_\mu^\eps)^{*} S_\eps (\wt{\bvarphi}_0 + \wt{\brho}_\eps) \|_{L_2(\O)}
= \eps \bigl\| \sum_{l=1}^3 \Lambda_l^\eps S_\eps D_l  (\wt{\bvarphi}_0 + \wt{\brho}_\eps) \bigr\|_{L_2(\O)}
\leqslant \eps C_\Lambda \sqrt{3} \| \wt{\bvarphi}_0 + \wt{\brho}_\eps \|_{H^1(\O)}.
$$
Together with \eqref{Mrr11} and \eqref{Mrr14}, this implies estimate \eqref{lem1} with  ${\mathcal C}_3=
C_\Lambda \sqrt{3} ({\mathcal C}_1 + {\mathcal C}_2)$.

Now, we check \eqref{lem3a}. According to \eqref{Mrr15} and \eqref{Mrr18}, 
$$
(\mu^\eps)^{1/2} \wt{\bpsi}_\eps = \wt{\mu}^\eps (\mu^0 )^{-1/2} S_\eps  (\wt{\bvarphi}_0 + \wt{\brho}_\eps)
+ \eps \sum_{l=1}^3 (\mu^\eps)^{1/2} \Lambda_l^\eps S_\eps D_l  (\wt{\bvarphi}_0 + \wt{\brho}_\eps).
$$
From \eqref{eff4} it follows that $\div \wt{\mu}(\x)=0$ (i.~e., the divergence of the columns of $\wt{\mu}(\x)$ is equal to zero). Hence,
\begin{equation}\label{lem4}
\begin{aligned}
\div (\mu^\eps)^{1/2} \wt{\bpsi}_\eps &= \sum_{j=1}^3 (\wt{\mu}^\eps (\mu^0 )^{-1/2} \e_j) 
\cdot \nabla [S_\eps  (\wt{\bvarphi}_0 + \wt{\brho}_\eps)]_j 
\\
&+ \eps
\sum_{l=1}^3 [\div (\mu^\eps)^{1/2} \Lambda_l^\eps] S_\eps D_l  (\wt{\bvarphi}_0 + \wt{\brho}_\eps)
\\
&+ \eps \sum_{l,j=1}^3 ((\mu^\eps)^{1/2} \f_{lj}^\eps) \cdot \nabla D_l [S_\eps  (\wt{\bvarphi}_0 + \wt{\brho}_\eps)]_j.
   \end{aligned}
\end{equation}
Here $[S_\eps  (\wt{\bvarphi}_0 + \wt{\brho}_\eps)]_j$ is the $j$th coordinate of the vector-valued function 
$S_\eps  (\wt{\bvarphi}_0 + \wt{\brho}_\eps)$, and $\div (\mu^\eps)^{1/2} \Lambda_l^\eps$ is the row (with the entries  
$\div (\mu^\eps)^{1/2} \f_{lj}^\eps$, $j=1,2,3$). Denote the consecutive summands in \eqref{lem4} by $J_1(\eps)$, $J_2(\eps)$, and $J_3(\eps)$. 
The last term is estimated by Proposition  \ref{prop f^eps S_eps} and \eqref{Mrr17a}:
\begin{equation}\label{lem5}
\| J_3(\eps) \|_{L_2(\O)} \leqslant 3 \eps \|\mu\|^{1/2}_{L_\infty} C_\Lambda\| \wt{\bvarphi}_0 + \wt{\brho}_\eps \|_{H^2(\R^3)} \leqslant {\mathcal C}_5' \eps \| \r\|_{L_2(\O)},
\end{equation}
where ${\mathcal C}_5' = 3\|\mu\|^{1/2}_{L_\infty} C_\Lambda ({\mathcal C}_1 + {\mathcal C}_2)$.
We have taken \eqref{Mrr11} and \eqref{Mrr14} into account.

By \eqref{div_f}, $\eps\div (\mu^\eps)^{1/2} \Lambda_l^\eps$ is the row with the entries 
$$
i \e_l \cdot ((\mu^0)^{1/2} \e_j) - i \e_l \cdot (\wt{\mu}^\eps (\mu^0)^{-1/2} \e_j), \quad j=1,2,3.
$$
This implies that the expression for $J_1(\eps)+ J_2(\eps)$ simplifies:
$$
J_1(\eps)+ J_2(\eps) = \sum_{j=1}^3 \left( \nabla [S_\eps (\wt{\bvarphi}_0 + \wt{\brho}_\eps)]_j\right)
\cdot ((\mu^0)^{1/2} \e_j) = \div (\mu^0)^{1/2} S_\eps(\wt{\bvarphi}_0 + \wt{\brho}_\eps).
$$
By Proposition~\ref{prop_Seps - I} and \eqref{Mrr11},  \eqref{Mrr14}, 
$$
\| \div (\mu^0)^{1/2} (S_\eps-I) (\wt{\bvarphi}_0 + \wt{\brho}_\eps)\|_{L_2(\R^3)} 
\leqslant
\eps r_1 \|\mu\|^{1/2}_{L_\infty} \| \wt{\bvarphi}_0 + \wt{\brho}_\eps \|_{H^2(\R^3)} \leqslant 
{\mathcal C}_5'' \eps \|\r\|_{L_2(\O)},
$$
where ${\mathcal C}_5'' = r_1 \|\mu\|^{1/2}_{L_\infty} ({\mathcal C}_1 + {\mathcal C}_2)$.
Since  $\div (\mu^0)^{1/2}(\bvarphi_0 + \brho_\eps)=0$ in the domain  $\O$ (see \eqref{Mrr4}, \eqref{Mrr11a}),   
\begin{equation}\label{lem6}
\|  J_1(\eps)+ J_2(\eps) \|_{L_2(\O)} \leqslant {\mathcal C}_5'' \eps \|\r\|_{L_2(\O)}.
\end{equation}
Relations \eqref{lem5} and \eqref{lem6} imply estimate \eqref{lem3a} with the constant ${\mathcal C}_5 =
{\mathcal C}_5' + {\mathcal C}_5''$.

It remains to check  \eqref{lem2}. We have
$$
(\mu^\eps)^{-1/2} \wt{\bpsi}_\eps = (\1+Y_\mu^\eps )  (\mu^0 )^{-1/2} S_\eps  (\wt{\bvarphi}_0 + \wt{\brho}_\eps)
+ \eps \sum_{l=1}^3 (\mu^\eps)^{-1/2} \Lambda_l^\eps S_\eps D_l  (\wt{\bvarphi}_0 + \wt{\brho}_\eps).
$$
Since $Y_\mu(\x)$ is the matrix with the columns $\nabla \Psi_j(\x)$, then $\rot (Y_\mu^\eps + \1)=0$. Hence,
\begin{equation}\label{lem7}
\begin{aligned}
&(\eta^\eps)^{-1}\rot (\mu^\eps)^{-1/2} \wt{\bpsi}_\eps 
\\
&= 
(\eta^\eps)^{-1}\sum_{j=1}^3
\left(\nabla [S_\eps  (\wt{\bvarphi}_0 + \wt{\brho}_\eps)]_j\right) 
 \times \left( (Y_\mu^\eps + \1)(\mu^0 )^{-1/2} \e_j \right)  
\\
&+ \eps (\eta^\eps)^{-1}
\sum_{l=1}^3 [\rot (\mu^\eps)^{-1/2} \Lambda_l^\eps] S_\eps D_l  (\wt{\bvarphi}_0 + \wt{\brho}_\eps)
\\
&+ \eps (\eta^\eps)^{-1}\sum_{l,j=1}^3 b_j (\mu^\eps)^{-1/2} \Lambda_{l}^\eps S_\eps  D_l D_j (\wt{\bvarphi}_0 + \wt{\brho}_\eps).
   \end{aligned}
\end{equation}
Here we have used the representation  $\rot = \sum_{j=1}^3 b_j D_j$, where 
$$
b_1 = \begin{pmatrix} 0 & 0 & 0 \cr 0 & 0 & -i \cr 0 & i & 0\end{pmatrix},
\quad 
b_2 = \begin{pmatrix} 0 & 0 & i \cr 0 & 0 & 0 \cr -i & 0 & 0\end{pmatrix},
\quad
b_3 = \begin{pmatrix} 0 & -i & 0 \cr i & 0 & 0 \cr 0 & 0 & 0\end{pmatrix}.
$$
The expression $\rot (\mu^\eps)^{-1/2} \Lambda_l^\eps$ is the matrix with the columns
$\rot (\mu^\eps)^{-1/2} \f_{lj}^\eps$, $j=1,2,3$. 
The consecutive summands in  \eqref{lem7} are denoted by ${\mathbf F}_1(\eps)$, 
${\mathbf F}_2(\eps)$, ${\mathbf F}_3(\eps)$. By Proposition \ref{prop f^eps S_eps} and 
\eqref{Mrr11}, \eqref{Mrr14}, \eqref{Mrr17a},
\begin{equation}\label{lem8}
\| {\mathbf F}_3(\eps) \|_{L_2(\O)} \leqslant \eps \| \eta^{-1}\|_{L_\infty} \| \mu^{-1}\|^{1/2}_{L_\infty}
3 C_\Lambda \|\wt{\bvarphi}_0 + \wt{\brho}_\eps \|_{H^2(\R^3)} \leqslant
{\mathcal C}_4 \eps \|\r\|_{L_2(\O)},
\end{equation}
where ${\mathcal C}_4= \| \eta^{-1}\|_{L_\infty} \| \mu^{-1}\|^{1/2}_{L_\infty}
3 C_\Lambda ({\mathcal C}_1 + {\mathcal C}_2)$.

According to \eqref{rot_f}, $\eps (\eta^\eps)^{-1} \rot (\mu^\eps)^{-1/2}\Lambda_l^\eps$
is the matrix with the columns 
 $$
 i (\1+ Y_\eta^\eps) (\eta^0)^{-1} \left( \e_l \times \left((\mu^0)^{-1/2} \e_j \right) \right)
 + i (\eta^\eps)^{-1} \left( \left((\1+Y_\mu^\eps) (\mu^0)^{-1/2} \e_j \right) \times \e_l \right).
 $$
 This allows us to simplify the expression for the sum ${\mathbf F}_1(\eps)+{\mathbf F}_2(\eps)$:
 $$
 {\mathbf F}_1(\eps)+{\mathbf F}_2(\eps) = (\1+ Y_\eta^\eps) (\eta^0)^{-1} \rot (\mu^0)^{-1/2}
 S_\eps   (\wt{\bvarphi}_0 + \wt{\brho}_\eps).
 $$
Together with \eqref{lem8}, this yields the required estimate \eqref{lem2}.
\end{proof}

\subsection{Introduction of the boundary layer correction term}
Denote
\begin{equation}
\label{pogr0}
\begin{aligned}
 {\mathcal R}_\eps[\bzeta] &:=  ((\1+Y_\eta^\eps) (\eta^0)^{-1} \rot   (\mu^0)^{-1/2}
 S_\eps   (\wt{\bvarphi}_0 + \wt{\brho}_\eps), \rot (\mu^\eps)^{-1/2} \bzeta )_{L_2(\O)} 
\\
&+ ((W_\mu^\eps)^* S_\eps   (\wt{\bvarphi}_0 + \wt{\brho}_\eps), \bzeta)_{L_2(\O)}
- i ((\mu^\eps)^{-1/2} \r, \bzeta )_{L_2(\O)}, \quad \bzeta \in \Dom {\mathfrak l}_\eps.
\end{aligned} 
\end{equation}
Recall that  $\Dom {\mathfrak l}_\eps$ is defined by \eqref{forma}.
We introduce the boundary layer correction term  $\s_\eps$ as the vector-valued function in the domain $\O$
such that
\begin{equation}
\label{pogr1}
\s_\eps \in L_2(\O;\C^3),\ \div (\mu^\eps)^{1/2} \s_\eps \in L_2(\O), \ 
\rot (\mu^\eps)^{-1/2} \s_\eps \in L_2(\O;\C^3),
\end{equation}
satisfying the identity 
\begin{equation}
\label{pogr2}
\begin{aligned}
&((\eta^\eps)^{-1} \rot (\mu^\eps)^{-1/2} \s_\eps, \rot (\mu^\eps)^{-1/2} \bzeta )_{L_2(\O)}
\\
&+ ( \div (\mu^\eps)^{1/2} \s_\eps, \div (\mu^\eps)^{1/2} \bzeta )_{L_2(\O)}
+ (\s_\eps,  \bzeta )_{L_2(\O)} =  {\mathcal R}_\eps[\bzeta],
\ \forall \bzeta \in \Dom {\mathfrak l}_\eps,
\end{aligned} 
\end{equation}
and the boundary condition 
\begin{equation}
\label{pogr3}
((\mu^\eps)^{1/2} \s_\eps )_n \vert_{\partial \O} = ((\mu^\eps)^{1/2} \bpsi_\eps )_n \vert_{\partial \O}.
\end{equation}

\begin{lemma}\label{lem4.5}
Let $\bvarphi_\eps$ be the solution of  problem \eqref{Mrr1}. Let $\bpsi_\eps$ be the first order approximation for the solution defined by \eqref{Mrr18}. 
Let $\s_\eps$ be the correction term satisfying relations  
\eqref{pogr1}--\eqref{pogr3}. We put $\V_\eps := \bvarphi_\eps - \bpsi_\eps + \s_\eps$. Then  $\V_\eps \in \Dom {\mathfrak l}_\eps$ and 
\begin{equation}
\label{pogr4}
\| \V_\eps \|_{L_2(\O)} + \| \div (\mu^\eps)^{1/2} \V_\eps \|_{L_2(\O)}
+ \| \rot (\mu^\eps)^{-1/2} \V_\eps \|_{L_2(\O)} \leqslant {\mathcal C}_6 \eps \| \r \|_{L_2(\O)}.
\end{equation}
The constant ${\mathcal C}_6$ depends only on the norms $\|\eta\|_{L_\infty}$, 
$\|\eta^{-1}\|_{L_\infty}$, $\|\mu\|_{L_\infty}$, $\|\mu^{-1}\|_{L_\infty}$, the parameters of the lattice $\Gamma$, and the domain $\O$.
\end{lemma}

\begin{proof}
From the conditions on  $\bvarphi_\eps$, $\bpsi_\eps$, and $\s_\eps$ it follows that
$$
\V_\eps \in L_2(\O;\C^3), \ \div (\mu^\eps)^{1/2} \V_\eps \in L_2(\O), \ \rot (\mu^\eps)^{-1/2} \V_\eps \in L_2(\O;\C^3).
$$
By \eqref{Mrr1}, the solution $\bvarphi_\eps$ satisfies the boundary condition
 $( (\mu^\eps)^{1/2} \bvarphi_\eps )_n \vert_{\partial \O} =0$. Together with  \eqref{pogr3}, this implies
 $( (\mu^\eps)^{1/2} \V_\eps )_n \vert_{\partial \O} =0$.
Hence, $\V_\eps \in \Dom {\mathfrak l}_\eps$.

Next, the function $\bvarphi_\eps$ satisfies the identity
\begin{equation*}
((\eta^\eps)^{-1} \rot (\mu^\eps)^{-1/2} \bvarphi_\eps, \rot (\mu^\eps)^{-1/2} \bzeta )_{L_2(\O)}
+ (\bvarphi_\eps,  \bzeta )_{L_2(\O)} 
=   i ((\mu^\eps)^{-1/2} \r, \bzeta )_{L_2(\O)},
\quad \forall \bzeta \in \Dom {\mathfrak l}_\eps,
\end{equation*}
and the condition $\div (\mu^\eps)^{1/2} \bvarphi_\eps =0$.
Combining this with \eqref{pogr0} and \eqref{pogr2}, we arrive at the identity
\begin{equation*}
\begin{aligned}
&{\mathfrak l}_\eps[\V_\eps, \bzeta] + (\V_\eps, \bzeta)_{L_2(\O)}
= - ((\eta^\eps)^{-1} \rot (\mu^\eps)^{-1/2} \bpsi_\eps, \rot (\mu^\eps)^{-1/2} \bzeta)_{L_2(\O)}
\\
&- ( \div (\mu^\eps)^{1/2} \bpsi_\eps, \div (\mu^\eps)^{1/2} \bzeta )_{L_2(\O)}
- (\bpsi_\eps,  \bzeta )_{L_2(\O)}
\\
&+  ((\1+Y_\eta^\eps) (\eta^0)^{-1} \rot   (\mu^0)^{-1/2}
 S_\eps   (\wt{\bvarphi}_0 + \wt{\brho}_\eps), \rot (\mu^\eps)^{-1/2} \bzeta )_{L_2(\O)} 
\\
&+ ((W_\mu^\eps)^* S_\eps   (\wt{\bvarphi}_0 + \wt{\brho}_\eps), \bzeta)_{L_2(\O)},
\quad \forall \bzeta \in \Dom {\mathfrak l}_\eps.
\end{aligned} 
\end{equation*}
Together with Lemma \ref{lem4.3}, this implies
$$
\begin{aligned}
& \left|{\mathfrak l}_\eps[\V_\eps, \bzeta] + (\V_\eps, \bzeta)_{L_2(\O)}\right|
\\
&\leqslant {\mathcal C}_6' \eps \| \r \|_{L_2(\O)}
\left( \|\bzeta\|_{L_2(\O)} + \|\div (\mu^\eps)^{1/2} \bzeta\|_{L_2(\O)} + \|\rot (\mu^\eps)^{-1/2} \bzeta\|_{L_2(\O)} \right),
\end{aligned}
$$
where ${\mathcal C}_6' ={\mathcal C}_3 + {\mathcal C}_4 + {\mathcal C}_5$. Substituting $\bzeta = \V_\eps$,
we arrive at estimate \eqref{pogr4} with the constant ${\mathcal C}_6 = 3 {\mathcal C}_6' \max\{\|\eta\|_{L_\infty};1\}$.
\end{proof}

Lemma \ref{lem4.5} shows that the difference $\bpsi_\eps -\s_\eps$ gives  approximation for the solution 
$\bvarphi_\eps$ with the error of sharp order $O(\eps)$. 
However, it is difficult to control the correction term $\s_\eps$, because it is solution of the elliptic equation with rapidly oscillating coefficients. 
But it is possible to estimate $\s_\eps$ in the ``energy'' norm.

\begin{theorem}\label{lem4.6}
Suppose that $\s_\eps$  satisfies relations 
\eqref{pogr1}--\eqref{pogr3}. Suppose that $\eps_1$ satisfies Condition {\rm \ref{cond1}}.
Then for $0< \eps \leqslant \eps_1$ we have
\begin{equation}
\label{pogr6a}
\| \s_\eps \|_{L_2(\O)} + \| \div (\mu^\eps)^{1/2} \s_\eps \|_{L_2(\O)}
+ \| \rot (\mu^\eps)^{-1/2} \s_\eps \|_{L_2(\O)} \leqslant {\mathcal C}_7 \eps^{1/2} \| \r \|_{L_2(\O)}.
\end{equation}
The constant  ${\mathcal C}_7$ depends only on the norms $\|\eta\|_{L_\infty}$, 
$\|\eta^{-1}\|_{L_\infty}$, $\|\mu\|_{L_\infty}$, $\|\mu^{-1}\|_{L_\infty}$, the parameters of the lattice $\Gamma$, and the domain $\O$.
\end{theorem}

The proof of Theorem \ref{lem4.6} requires a big technical work; 
it is given in Section~\ref{sec5}.

\subsection{Approximation of the function $\bvarphi_\eps$}

Approximation for the function  $\bvarphi_\eps$ is deduced from Lemma \ref{lem4.5} and Theorem \ref{lem4.6}.

\begin{theorem}
\label{th_varphi}
Let $\bvarphi_\eps$ be the solution of problem  \eqref{Mrr1}. 
Suppose that $\eps_1$ is subject to Condition {\rm \ref{cond1}}.
Then for $0< \eps \leqslant \eps_1$ we have
\begin{align}
\label{res01}
&\| \bvarphi_\eps - (W^\eps_\mu)^* (\bvarphi_0 + \brho_\eps)\|_{L_2(\O)} \leqslant {\mathcal C}_8 \eps^{1/2} \|\r\|_{L_2(\O)},
\\
\label{res02}
&\| (\eta^\eps)^{-1} \rot (\mu^\eps)^{-1/2}\bvarphi_\eps - (\1 + Y_\eta^\eps) 
(\eta^0)^{-1} \rot (\mu^0)^{-1/2} (\bvarphi_0 + \brho_\eps)\|_{L_2(\O)} 
\leqslant {\mathcal C}_9 \eps^{1/2} \|\r\|_{L_2(\O)}.
\end{align}
The constants ${\mathcal C}_8$ and ${\mathcal C}_9$ depend only on the norms $\|\eta\|_{L_\infty}$, 
$\|\eta^{-1}\|_{L_\infty}$, $\|\mu\|_{L_\infty}$, $\|\mu^{-1}\|_{L_\infty}$, the parameters of the lattice $\Gamma$, and the domain $\O$.
\end{theorem}

\begin{proof}
Lemma \ref{lem4.5} and Theorem \ref{lem4.6} directly imply that
$$
\| \bvarphi_\eps - \bpsi_\eps \|_{L_2(\O)} 
+ \| \rot (\mu^\eps)^{-1/2} (\bvarphi_\eps - \bpsi_\eps) \|_{L_2(\O)} \leqslant
({\mathcal C}_6+ {\mathcal C}_7) \eps^{1/2} \| \r \|_{L_2(\O)}
$$
for $0 < \eps \leqslant \eps_1$. Together with Lemma \ref{lem4.3}, this yields the following inequalities:
\begin{align}
\label{res1}
&\| \bvarphi_\eps - (W^\eps_\mu)^* S_\eps (\wt{\bvarphi}_0 + \wt{\brho}_\eps)\|_{L_2(\O)} \leqslant {\mathcal C}'_8 \eps^{1/2} \|\r\|_{L_2(\O)}, \quad 0< \eps \leqslant \eps_1,
\\
\label{res2}
\begin{split}
&\| (\eta^\eps)^{-1} \rot (\mu^\eps)^{-1/2}\bvarphi_\eps - (\1 + Y_\eta^\eps) 
(\eta^0)^{-1} \rot (\mu^0)^{-1/2} S_\eps (\wt{\bvarphi}_0 + \wt{\brho}_\eps)\|_{L_2(\O)} 
\\
&
\leqslant {\mathcal C}'_9 \eps^{1/2} \|\r\|_{L_2(\O)}, \quad 0< \eps \leqslant \eps_1,
\end{split}
\end{align}
where ${\mathcal C}_8'= {\mathcal C}_3 + {\mathcal C}_6 + {\mathcal C}_7$ and 
${\mathcal C}_9'= {\mathcal C}_4 +( {\mathcal C}_6 + {\mathcal C}_7) \| \eta^{-1} \|_{L_\infty}$.

It remains to show that the smoothing operator $S_\eps$ in  \eqref{res1}, \eqref{res2}
can be replaced by the identity operator; this will only affect the constants in the estimates.
By \eqref{Mrr15} and Lemma~\ref{lem_PSu},
$$
\begin{aligned}
&\| (W^\eps_\mu)^* (S_\eps -I) (\wt{\bvarphi}_0 + \wt{\brho}_\eps) \|_{L_2(\O)}
\\
&=
\| (\mu^\eps)^{1/2}(\1+ Y^\eps_\mu) (\mu^0)^{-1/2} (S_\eps -I) (\wt{\bvarphi}_0 + \wt{\brho}_\eps) \|_{L_2(\O)}
\\
&\leqslant
\| \mu\|_{L_\infty}^{1/2} \| \mu^{-1}\|_{L_\infty}^{1/2} \| (S_\eps -I) (\wt{\bvarphi}_0 + \wt{\brho}_\eps) \|_{L_2(\R^3)}
\\
&
+ \| \mu\|_{L_\infty}^{1/2} 
\| Y^\eps_\mu (\mu^0)^{-1/2} (S_\eps -I) (\wt{\bvarphi}_0 + \wt{\brho}_\eps) \|_{L_2(\R^3)}
\\
& \leqslant \| \mu\|_{L_\infty}^{1/2} \| \mu^{-1}\|_{L_\infty}^{1/2} \left( (1+ \sqrt{\beta_{1,\mu}})
\| (S_\eps -I) (\wt{\bvarphi}_0 + \wt{\brho}_\eps) \|_{L_2(\R^3)} \right.
\\
&
\left.+ \eps \sqrt{\beta_{2,\mu}} \wh{C}_\mu 
\| (S_\eps -I) \nabla (\wt{\bvarphi}_0 + \wt{\brho}_\eps) \|_{L_2(\R^3)} \right). 
\end{aligned}
$$
The first term in the parentheses is estimated with the help of Proposition~\ref{prop_Seps - I},
and the second one is estimated by using \eqref{S_eps <= 1}. We also take \eqref{Mrr11} and \eqref{Mrr14} into account. 
This yields
\begin{equation}
\label{res3}
\| (W^\eps_\mu)^* (S_\eps -I) (\wt{\bvarphi}_0 + \wt{\brho}_\eps) \|_{L_2(\O)} 
\leqslant {\mathcal C}_8'' \eps \|\r\|_{L_2(\O)},
\end{equation}
where ${\mathcal C}_8'' =\| \mu\|_{L_\infty}^{1/2} \| \mu^{-1}\|_{L_\infty}^{1/2} \left( (1+ \sqrt{\beta_{1,\mu}}) r_1 + 
2 \sqrt{\beta_{2,\mu}} \wh{C}_\mu\right) ({\mathcal C}_1 + {\mathcal C}_2)$.
Relations \eqref{res1} and \eqref{res3} imply the required estimate  \eqref{res01} with the constant
${\mathcal C}_8={\mathcal C}_8' + {\mathcal C}_8''$.

Similarly, using Lemma \ref{lem_PSu}, Proposition \ref{prop_Seps - I}, estimate \eqref{S_eps <= 1}, and inequalities \eqref{Mrr11}, \eqref{Mrr14}, we deduce 
\begin{equation}
\label{res4}
\| (\1 + Y_\eta^\eps) 
(\eta^0)^{-1} \rot (\mu^0)^{-1/2} (S_\eps -I) (\wt{\bvarphi}_0 + \wt{\brho}_\eps)\|_{L_2(\O)} 
\leqslant 
{\mathcal C}_9'' \eps \|\r\|_{L_2(\O)},
\end{equation}
where
${\mathcal C}_9''=  \| \eta^{-1}\|_{L_\infty} \| \mu^{-1}\|^{1/2}_{L_\infty} \left( r_1 (1+ \sqrt{\beta_{1,\eta}})
+ 2 \sqrt{\beta_{2,\eta}} \wh{C}_\eta \right)({\mathcal C}_1 + {\mathcal C}_2)$.
Combining \eqref{res2} and \eqref{res4}, we arrive at estimate \eqref{res02} with the constant ${\mathcal C}_9=
{\mathcal C}_9' + {\mathcal C}_9''$.
\end{proof}

\subsection{The final result in the case where $\q=0$}

Let us express the fields with index $\r$ in terms of the function $\bvarphi_\eps$ introduced in Subsection~\ref{sim}:
$$
\begin{aligned}
&\v_\eps^{(\r)}= (\mu^\eps)^{-1/2} \bvarphi_\eps,
\quad
\z_\eps^{(\r)}= (\mu^\eps)^{1/2} \bvarphi_\eps,
\\
&\u_\eps^{(\r)}= (\eta^\eps)^{-1}\rot (\mu^\eps)^{-1/2} \bvarphi_\eps,
\quad
\w_\eps^{(\r)}= \rot (\mu^\eps)^{-1/2} \bvarphi_\eps.
\end{aligned}
$$
Similarly, the effective fields with  index $\r$ are related to the function $\bvarphi_0$ introduced in Subsection~\ref{sec4.2}:
$$
\begin{aligned}
&\v_0^{(\r)}= (\mu^0)^{-1/2} \bvarphi_0,
\quad
\z_0^{(\r)}= (\mu^0)^{1/2} \bvarphi_0,
\\
&\u_0^{(\r)}= (\eta^0)^{-1}\rot (\mu^0)^{-1/2} \bvarphi_0,
\quad
\w_0^{(\r)}= \rot (\mu^0)^{-1/2} \bvarphi_0.
\end{aligned}
$$
The correction fields with  index $\r$ are expressed in terms of  $\brho_\eps$ (see Subsection \ref{sec4.3}):
$$
\begin{aligned}
&\wh{\v}_\eps^{(\r)}= (\mu^0)^{-1/2} \brho_\eps,
\quad
\wh{\z}_\eps^{(\r)}= (\mu^0)^{1/2} \brho_\eps,
\\
&\wh{\u}_\eps^{(\r)}= (\eta^0)^{-1}\rot (\mu^0)^{-1/2} \brho_\eps,
\quad
\wh{\w}_\eps^{(\r)}= \rot (\mu^0)^{-1/2} \brho_\eps.
\end{aligned}
$$
Combining these relations with Theorem~\ref{th_varphi}, we arrive at the final result in the case where $\q=0$.

\begin{theorem}\label{main_th_q=0}
Let $(\w_\eps^{(\r)}, \z_\eps^{(\r)})$ be the solution of system \eqref{M1} with $\q=0$ and let 
$\u_\eps^{(\r)} = (\eta^\eps)^{-1}\w_\eps^{(\r)}$,  $\v_\eps^{(\r)} = (\mu^\eps)^{-1}\z_\eps^{(\r)}$. 
Let $(\w_0^{(\r)}, \z_0^{(\r)})$ be the solution of the effective system \eqref{M1eff} with $\q=0$ and let 
$\u_0^{(\r)} = (\eta^0)^{-1}\w_0^{(\r)}$,  $\v_0^{(\r)} = (\mu^0)^{-1}\z_0^{(\r)}$. Let  
$(\wh{\w}^{(\r)}_\eps, \wh{\z}^{(\r)}_\eps)$ be the solution of the correction system \eqref{M_corr} with $\q=0$ and let 
$\wh{\u}_\eps^{(\r)} =  (\eta^0)^{-1}\wh{\w}^{(\r)}_\eps$,
$\wh{\v}_\eps^{(\r)} = (\mu^0)^{-1}\wh{\z}^{(\r)}_\eps$. Suppose that $Y_\eta$, $G_\eta$, $Y_\mu$, and $G_\mu$
are the periodic matrix-valued functions introduced in Subsection  {\rm \ref{eff_mat}}. 
Suppose that $\eps_1$ is subject to Condition {\rm \ref{cond1}}.
Then for $0 < \eps \le \eps_1$ we have 
\begin{align*}
\| \u_\eps^{(\r)} - (\1 + Y_\eta^\eps)(\u_0^{(\r)} + \wh{\u}_\eps^{(\r)}) \|_{L_2(\O)} &\le {\mathcal C}_9 \eps^{1/2}  
\| \r \|_{L_2(\O)},
\\
\| \w_\eps^{(\r)} - (\1 + G_\eta^\eps)(\w_0^{(\r)} + \wh{\w}_\eps^{(\r)}) \|_{L_2(\O)} &\le {\mathcal C}_9 \|\eta\|_{L_\infty} \eps^{1/2} 
\| \r \|_{L_2(\O)},
\\
\| \v_\eps^{(\r)} - (\1 + Y_\mu^\eps)(\v_0^{(\r)} + \wh{\v}_\eps^{(\r)}) \|_{L_2(\O)} &\le {\mathcal C}_8 \|\mu^{-1}\|_{L_\infty}^{1/2} \eps^{1/2} \| \r \|_{L_2(\O)},
\\
\| \z_\eps^{(\r)} - (\1 + G_\mu^\eps)(\z_0^{(\r)} + \wh{\z}_\eps^{(\r)}) \|_{L_2(\O)} &\le {\mathcal C}_8 \|\mu\|_{L_\infty}^{1/2} \eps^{1/2}  
\| \r \|_{L_2(\O)}.
\end{align*}
The constants ${\mathcal C}_8$ and ${\mathcal C}_9$ depend only on the norms $\|\eta\|_{L_\infty}$, $\|\eta^{-1}\|_{L_\infty}$,
$\|\mu\|_{L_\infty}$, $\|\mu^{-1}\|_{L_\infty}$, the parameters of the lattice $\Gamma$, and the domain $\O$. 
\end{theorem}

\section{Estimation of the boundary layer correction term\label{sec5}}

This section contains the proof of Theorem~\ref{lem4.6}.

\subsection{Identification of  $\Dom {\mathfrak l}_\eps$ and $\Dom {\mathfrak l}_0$}
The following lemma plays the key role for what follows.

\begin{lemma}\label{lem5.1}
There exists a linear operator $T_\eps: \Dom {\mathfrak l}_\eps \to \Dom {\mathfrak l}_0$ such that 
the function ${\bzeta}^0_\eps = T_\eps \bzeta$, $\bzeta \in \Dom {\mathfrak l}_\eps$, satisfies relations
\begin{equation}\label{lemmm1}
\div (\mu^0)^{1/2} {\bzeta}^0_\eps = \div (\mu^\eps)^{1/2} \bzeta,
\quad \rot (\mu^0)^{-1/2} {\bzeta}^0_\eps = \rot (\mu^\eps)^{-1/2} \bzeta,
\end{equation}
and estimates
\begin{align}
\label{lemmm2}
&\| {\bzeta}^0_\eps \|_{L_2(\O)} \leqslant {\mathcal C}_{10} \| \bzeta \|_{L_2(\O)},
\\
\label{lemmm3}
&\| {\bzeta}^0_\eps \|_{H^1(\O)}  \leqslant   {\mathcal C}_{11} \left(\| \bzeta \|_{L_2(\O)} 
+ \|\div (\mu^\eps)^{1/2} \bzeta \|_{L_2(\O)}
+ \|\rot (\mu^\eps)^{-1/2} \bzeta \|_{L_2(\O)}\right).
\end{align}
The constant  $ {\mathcal C}_{10}$ depends only on  
$\|\mu\|_{L_\infty}$, $\|\mu^{-1}\|_{L_\infty}$, and ${\mathcal C}_{11}$ depends on the same parameters and the domain $\O$. 
\end{lemma}

\begin{proof}
We consider two auxiliary problems. 

\textbf{The first auxiliary problem}. Let $\bzeta \in \Dom {\mathfrak l}_\eps$.
Denote $f_\eps := \div (\mu^\eps)^{1/2} \bzeta \in L_2(\O)$.
Suppose that $\phi_{\eps,1} \in H^1(\O)$ is the  solution of the Neumann problem
\begin{equation}
\label{lemm1}
\div \mu^0 \nabla \phi_{\eps,1}(\x) = f_\eps(\x), \ \x \in \O; \quad (\mu^0 \nabla \phi_{\eps,1})_n\vert_{\partial \O}=0. 
\end{equation}
The solvability condition is fulfilled, since
$$
\int_\O f_\eps(\x)\, d\x = \int_\O \div (\mu^\eps)^{1/2} \bzeta \,d\x =0,
$$
by the condition  $((\mu^\eps)^{1/2} \bzeta)_n\vert_{\partial \O}=0$.
The solution of problem \eqref{lemm1} is defined up a constant summand; 
we need only the gradient of the solution. The solution satisfies the identity
\begin{equation}
\label{lemm2}
\int_\O \langle \mu^0 \nabla \phi_{\eps,1}, \nabla \omega  \rangle \, d\x = 
\int_O \langle (\mu^\eps)^{1/2} \bzeta, \nabla \omega  \rangle \, d\x,
\quad \forall \omega \in H^1(\O).  
\end{equation}

We put ${\bzeta}^0_{\eps,1} = (\mu^0)^{1/2} \nabla \phi_{\eps,1}$. Then 
\begin{equation}
\label{lemm3}
\div (\mu^0)^{1/2}  {\bzeta}^0_{\eps,1} = \div (\mu^\eps)^{1/2} \bzeta, 
 \quad 
\rot (\mu^0)^{-1/2}  {\bzeta}^0_{\eps,1} = 0,
\quad  
((\mu^0)^{1/2}  {\bzeta}^0_{\eps,1})_n \vert_{\partial \O}=0.
\end{equation}
Substituting $\omega = \phi_{\eps,1}$ in \eqref{lemm2}, we arrive at the estimate 
\begin{equation}
\label{lemm4}
\| {\bzeta}^0_{\eps,1} \|_{L_2(\O)} \leqslant \|\mu\|_{L_\infty}^{1/2} \|\mu^{-1}\|_{L_\infty}^{1/2} \| \bzeta \|_{L_2(\O)}.
\end{equation}
The smoothness of the boundary  ($\partial \O \in C^{1,1}$) ensures the regularity of the solution 
of problem \eqref{lemm1}: we have $\phi_{\eps,1} \in H^2(\O)$ and 
\begin{equation}
\label{lemm5}
\| {\bzeta}^0_{\eps,1} \|_{H^1(\O)} \leqslant  {c}_1 \| f_\eps \|_{L_2(\O)} = c_1 
\| \div (\mu^\eps)^{1/2} \bzeta \|_{L_2(\O)}.
\end{equation}
The constant $c_1$ depends only on the norms $\| \mu \|_{L_\infty}$, $\| \mu^{-1} \|_{L_\infty}$, and the domain $\O$.

\textbf{The second auxiliary problem}. Let $\bzeta \in \Dom {\mathfrak l}_\eps$.
Denote $g_\eps := \div \mu^0 (\mu^\eps)^{-1/2} \bzeta \in H^{-1}(\O)$.
Let $\phi_{\eps,2} \in H^1(\O)$ be the solution of the Neumann problem
\begin{equation*}
\div \mu^0 \nabla \phi_{\eps,2}(\x) = g_\eps(\x), \ \x \in \O; 
\quad (\mu^0 \nabla \phi_{\eps,2})_n\vert_{\partial \O}= (\mu^0 (\mu^\eps)^{-1/2} \bzeta)_n\vert_{\partial \O}. 
\end{equation*}
Strictly speaking, the solution $\phi_{\eps,2}$ is understood as an element of  $H^1(\O)$ satisfying the identity
\begin{equation}
\label{lemm7}
\int_\O \langle \mu^0 \nabla \phi_{\eps,2}, \nabla \omega  \rangle \, d\x = 
\int_O \langle \mu^0 (\mu^\eps)^{-1/2} \bzeta, \nabla \omega  \rangle \, d\x,
\quad \forall \omega \in H^1(\O).  
\end{equation}
The solution is defined up to a constant summand; we need only the gradient of the solution.
We put ${\bzeta}^0_{\eps,2} = (\mu^0)^{1/2} ((\mu^\eps)^{-1/2}\bzeta - \nabla \phi_{\eps,2})$. Then 
\begin{equation}
\label{lemm8}
\div (\mu^0)^{1/2}  {\bzeta}^0_{\eps,2} = 0, 
\quad 
\rot (\mu^0)^{-1/2}  {\bzeta}^0_{\eps,2} = \rot (\mu^\eps)^{-1/2}\bzeta,
\quad  
((\mu^0)^{1/2}  {\bzeta}^0_{\eps, 2})_n \vert_{\partial \O}=0.
\end{equation}
 Substituting $\omega = \phi_{\eps,2}$ in \eqref{lemm7}, we arrive at the estimate
 \begin{equation}
\label{lemm9}
\| (\mu^0)^{1/2} \nabla \phi_{\eps,2}  \|_{L_2(\O)} \leqslant \|\mu\|_{L_\infty}^{1/2} \|\mu^{-1}\|_{L_\infty}^{1/2} 
\| \bzeta \|_{L_2(\O)}.
\end{equation}
Hence,
\begin{equation}
\label{lemm10}
\| {\bzeta}^0_{\eps,2}   \|_{L_2(\O)} \leqslant 2 \|\mu\|_{L_\infty}^{1/2} \|\mu^{-1}\|_{L_\infty}^{1/2} \| \bzeta \|_{L_2(\O)}.
\end{equation}
From \eqref{lemm8} and \eqref{lemm10} it follows that  ${\bzeta}^0_{\eps,2} \in \Dom {\mathfrak l}_0 \subset H^1(\O;\C^3)$. 
We have  (cf. \eqref{coercive})
\begin{equation}
\label{lemm11}
\| {\bzeta}^0_{\eps,2}   \|_{H^1(\O)} \leqslant c_2 ( \|\rot (\mu^\eps)^{-1/2}\bzeta \|_{L_2(\O)} + \| \bzeta \|_{L_2(\O)}).
\end{equation}
The constant $c_2$ depends only on the norms $\| \mu \|_{L_\infty}$, $\| \mu^{-1} \|_{L_\infty}$, and the domain $\O$.

We put ${\bzeta}^0_\eps = {\bzeta}^0_{\eps,1} + {\bzeta}^0_{\eps,2}$. Then relations \eqref{lemm3} and \eqref{lemm8}
imply \eqref{lemmm1}. Combining \eqref{lemm4} and \eqref{lemm10}, we obtain 
estimate \eqref{lemmm2} with the constant ${\mathcal C}_{10}= 3 \|\mu\|^{1/2}_{L_\infty} \| \mu^{-1}\|^{1/2}_{L_\infty}$. 
Finally, \eqref{lemm5} and \eqref{lemm11} imply inequality \eqref{lemmm3} with the constant ${\mathcal C}_{11}=\max \{c_1, c_2\}$.
\end{proof}

\begin{remark}
Under the assumptions of Lemma~{\rm \ref{lem5.1}}, we have
\begin{equation}
\label{lemm12}
 (\mu^0)^{-1/2} {\bzeta}^0_\eps - (\mu^\eps)^{-1/2} \bzeta = \nabla (\phi_{\eps,1} - \phi_{\eps,2}).
 \end{equation}
\end{remark}

Let $0< \eps \leqslant \eps_0$ (where $\eps_0$ is chosen according to Condition \ref{cond1}). 
We fix a cut-off function $\theta_\eps(\x)$ onto the $(2\eps)$-neighborhood of the boundary $\partial \O$
satisfying the following conditions:
\begin{equation}
\label{srezka}
\begin{aligned}
\theta_\eps \in C_0^\infty(\R^3), \quad 0 \leqslant \theta_\eps(\x) \leqslant 1, 
\quad \operatorname{supp} \theta_\eps \subset (\partial  \O)_{2\eps},
\\
\theta_\eps(\x)=1 \ \text{for}\ \x \in (\partial \O)_\eps, \quad  \eps |\nabla \theta_\eps(\x)| \leqslant \kappa=
\operatorname{Const}.
\end{aligned}
 \end{equation}
The constant $\kappa$ depends only on the domain $\O$.

\subsection{Analysis of the first term in ${\mathcal R}_\eps[\bzeta]$}
Denote the first summand in  \eqref{pogr0} by ${\mathcal J}_\eps[\bzeta]$
and represent it as the sum of four terms:
\begin{equation}
\label{rhs1}
{\mathcal J}_\eps[\bzeta] = \sum_{l=1}^4 {\mathcal J}_\eps^{(l)}[\bzeta], \quad  \bzeta \in \Dom {\mathfrak l}_\eps,
 \end{equation}
where 
\begin{align}
\label{rhs2}
{\mathcal J}_\eps^{(1)}[\bzeta] &=
(Y_\eta^\eps (\eta^0)^{-1} \rot   (\mu^0)^{-1/2}
 S_\eps   (\wt{\bvarphi}_0 + \wt{\brho}_\eps), \rot (\mu^\eps)^{-1/2} \bzeta )_{L_2(\O)},
\\
\label{rhs3}
{\mathcal J}_\eps^{(2)}[\bzeta] &=
( (\eta^0)^{-1} \rot   (\mu^0)^{-1/2}
  {\bvarphi}_0, \rot (\mu^\eps)^{-1/2} \bzeta )_{L_2(\O)},
\\
\label{rhs4}
{\mathcal J}_\eps^{(3)}[\bzeta] &=
((\eta^0)^{-1} \rot   (\mu^0)^{-1/2}
 {\brho}_\eps, \rot (\mu^\eps)^{-1/2} \bzeta )_{L_2(\O)},
\\
\label{rhs4a}
{\mathcal J}_\eps^{(4)}[\bzeta] &=
( (\eta^0)^{-1} \rot   (\mu^0)^{-1/2}
 (S_\eps-I) (\wt{\bvarphi}_0+  \wt{\brho}_\eps), \rot (\mu^\eps)^{-1/2} \bzeta )_{L_2(\O)}.
\end{align}

\begin{lemma}\label{lem5.3}
For $0< \eps \leqslant \eps_1$ the term \eqref{rhs2} satisfies
\begin{equation}
\label{t0}
|{\mathcal J}_\eps^{(1)}[\bzeta]| \leqslant {\mathcal C}_{12} \eps^{1/2} \|\r\|_{L_2(\O)}
\|  \rot (\mu^\eps)^{-1/2} \bzeta \|_{L_2(\O)}, \quad  \bzeta \in \Dom {\mathfrak l}_\eps. 
\end{equation}
The constant ${\mathcal C}_{12}$ depends on the norms $\| \eta \|_{L_\infty}$,   $\| \eta^{-1} \|_{L_\infty}$,
$\| \mu \|_{L_\infty}$, $\| \mu^{-1} \|_{L_\infty}$, the domain $\O$, and the parameters of the lattice $\Gamma$.
\end{lemma}

\begin{proof}
Denote $h_{\eps,j}:=[(\eta^0)^{-1} \rot   (\mu^0)^{-1/2}
 S_\eps   (\wt{\bvarphi}_0 + \wt{\brho}_\eps)]_j$, $j=1,2,3$. 
The columns of the matrix $Y_\eta^\eps$ are given by  $(\nabla \Phi_j)^\eps=\eps \nabla \Phi_j^\eps$, $j=1,2,3$.
Obviously, $\eps (\nabla \Phi_j^\eps) h_{\eps,j} = \eps \nabla (\Phi_j^\eps h_{\eps,j}) - 
\eps \Phi_j^\eps \nabla h_{\eps,j}$. Hence,
\begin{equation}
\label{rhs5}
{\mathcal J}_\eps^{(1)}[\bzeta] = \wh{\mathcal J}_\eps^{(1)}[\bzeta]
- \wt{\mathcal J}_\eps^{(1)}[\bzeta],
\end{equation}
where
\begin{align}
\label{rhs6}
\wh{\mathcal J}_\eps^{(1)}[\bzeta]=
\eps \sum_{j=1}^3 \left(  \nabla (\Phi_j^\eps h_{\eps,j}), \rot (\mu^\eps)^{-1/2} \bzeta \right)_{L_2(\O)},
\\
\wt{\mathcal J}_\eps^{(1)}[\bzeta]=
\eps \sum_{j=1}^3 \left( \Phi_j^\eps \nabla h_{\eps,j}, \rot (\mu^\eps)^{-1/2} \bzeta \right)_{L_2(\O)}.
\label{rhs7}
\end{align}

Suppose that $\theta_\eps$ is the cut-off function satisfying \eqref{srezka}. 
The term \eqref{rhs6} can be written as
\begin{equation}
\label{rhs8}
\wh{\mathcal J}_\eps^{(1)}[\bzeta]=
\eps \sum_{j=1}^3 \left(  \nabla ( \theta_\eps \Phi_j^\eps h_{\eps,j}), \rot (\mu^\eps)^{-1/2} \bzeta \right)_{L_2(\O)},
\end{equation}
since
$$
\left(  \nabla ( (1-\theta_\eps) \Phi_j^\eps h_{\eps,j}), \rot (\mu^\eps)^{-1/2} \bzeta \right)_{L_2(\O)}=0,
$$
which can be checked by integration by parts and using the identity $\div \rot=0$.
(When checking, we can replace $\rot (\mu^\eps)^{-1/2} \bzeta$ by $\rot (\mu^0)^{-1/2} {\bzeta}^0_\eps$,
where ${\bzeta}^0_\eps \in \Dom {\mathfrak l}_0$. First we assume that  ${\bzeta}^0_\eps \in H^2(\O;\C^3)$, 
and next we  close the result by continuity.)

Next, we have
\begin{equation}
\label{rhs9}
\eps \nabla ( \theta_\eps \Phi_j^\eps h_{\eps,j}) = (\eps \nabla  \theta_\eps) \Phi_j^\eps h_{\eps,j}+
 \theta_\eps (\nabla \Phi_j)^\eps h_{\eps,j} + \eps \theta_\eps \Phi_j^\eps \nabla h_{\eps,j}. 
\end{equation}
The first term on the right is estimated with the help of \eqref{srezka} and Lemma \ref{lem02}:
$$
\begin{aligned}
&\| (\eps \nabla  \theta_\eps) \Phi_j^\eps h_{\eps,j} \|_{L_2(\O)}
\leqslant \kappa \| \Phi_j^\eps h_{\eps,j}\|_{L_2((\partial \O)_{2\eps} )} 
\\
&\leqslant
\eps^{1/2} \kappa \beta_*^{1/2} |\Omega|^{-1/2} \| \Phi_j\|_{L_2(\Omega)} \| \eta^{-1}\|_{L_\infty}
\| \mu^{-1}\|_{L_\infty}^{1/2} \| \wt{\bvarphi}_0 + \wt{\brho}_\eps\|_{H^2(\R^3)}
\end{aligned}
$$
for $0< \eps \leqslant \eps_1$.
Together with  \eqref{efff}, \eqref{Mrr11}, and \eqref{Mrr14}, this implies 
$$
\| (\eps \nabla  \theta_\eps) \Phi_j^\eps h_{\eps,j} \|_{L_2(\O)}
\leqslant {\mathcal C}_{12}' \eps^{1/2} \| \r \|_{L_2( \O )},\quad 0< \eps \leqslant \eps_1,
$$
where ${\mathcal C}_{12}' = \kappa \beta_*^{1/2}  (2 r_0)^{-1} \|\eta\|_{L_\infty}^{1/2} \| \eta^{-1}\|^{3/2}_{L_\infty}
\| \mu^{-1}\|_{L_\infty}^{1/2}({\mathcal C}_1 + {\mathcal C}_2)$.

To estimate the second term in the right-hand side of \eqref{rhs9}, we apply 
\eqref{srezka} and Lemma \ref{lem02}:
$$
\begin{aligned}
&\|  \theta_\eps (\nabla \Phi_j)^\eps h_{\eps,j} \|_{L_2(\O)}
\leqslant \| (\nabla \Phi_j)^\eps h_{\eps,j}\|_{L_2((\partial \O)_{2\eps} )} 
\\
&\leqslant
\eps^{1/2} \beta_*^{1/2} |\Omega|^{-1/2} \| \nabla \Phi_j\|_{L_2(\Omega)} \| \eta^{-1}\|_{L_\infty}
\| \mu^{-1}\|_{L_\infty}^{1/2} \| \wt{\bvarphi}_0 + \wt{\brho}_\eps\|_{H^2(\R^3)}
\end{aligned}
$$
for $0< \eps \leqslant \eps_1$.
 Taking  \eqref{eff6b}, \eqref{Mrr11}, and \eqref{Mrr14} into account, we arrive at
 $$
\| \theta_\eps ( \nabla \Phi_j)^\eps h_{\eps,j} \|_{L_2(\O)}
\leqslant {\mathcal C}_{12}'' \eps^{1/2} \| \r \|_{L_2( \O)},
\quad 0< \eps \leqslant \eps_1,
$$
where ${\mathcal C}_{12}'' =  \beta_*^{1/2}  \|\eta\|_{L_\infty}^{1/2} \| \eta^{-1}\|^{3/2}_{L_\infty}
\| \mu^{-1}\|_{L_\infty}^{1/2}({\mathcal C}_1 + {\mathcal C}_2)$.
 
Finally, the third term in the right-hand side of \eqref{rhs9} is estimated by using Proposition 
\ref{prop f^eps S_eps}:
\begin{equation}
\label{t1}
\begin{aligned}
&\eps \| \theta_\eps \Phi_j^\eps \nabla h_{\eps,j}\|_{L_2(\O)}
\leqslant \eps \| \Phi_j^\eps \nabla h_{\eps,j}\|_{L_2(\R^3)}
\\
&\leqslant \eps |\Omega|^{-1/2} \| \Phi_j \|_{L_2(\Omega)} \| \eta^{-1}\|_{L_\infty}
\| \mu^{-1}\|_{L_\infty}^{1/2} \| \wt{\bvarphi}_0 + \wt{\brho}_\eps\|_{H^2(\R^3)}.
\end{aligned}
\end{equation}
Together with \eqref{efff}, \eqref{Mrr11}, and \eqref{Mrr14}, this yields
\begin{equation}
\label{t2}
\eps \| \theta_\eps \Phi_j^\eps \nabla h_{\eps,j}\|_{L_2(\O)}
\leqslant {\mathcal C}_{12}''' \eps \| \r \|_{L_2(\O)},
\end{equation}
where ${\mathcal C}_{12}''' =  (2 r_0)^{-1} \|\eta\|_{L_\infty}^{1/2} \| \eta^{-1}\|^{3/2}_{L_\infty}
\| \mu^{-1}\|_{L_\infty}^{1/2}({\mathcal C}_1 + {\mathcal C}_2)$.

As a result, we obtain
\begin{equation}
\label{t2a}
\| \eps \nabla ( \theta_\eps \Phi_j^\eps h_{\eps,j}) \|_{L_2(\O)} \leqslant \wh{\mathcal C}_{12} \eps^{1/2}\|\r\|_{L_2(\O)},
\quad 0< \eps \leqslant \eps_1,
\end{equation}
where $\wh{\mathcal C}_{12} = {\mathcal C}_{12}' + {\mathcal C}_{12}''  + {\mathcal C}_{12}'''$. 
This imples  the following estimate for the term \eqref{rhs8}:
\begin{equation}
\label{t3}
|\wh{\mathcal J}_\eps^{(1)}[\bzeta]| \leqslant 3 \wh{\mathcal C}_{12}
\eps^{1/2} \| \r\|_{L_2(\O)} \| \rot (\mu^\eps)^{-1/2} \bzeta \|_{L_2(\O)},\quad 0< \eps \leqslant \eps_1.
\end{equation}

Now, we consider the term  \eqref{rhs7}. Similarly to  \eqref{t1}, \eqref{t2}, 
$$
\|  \Phi_j^\eps \nabla h_{\eps,j} \|_{L_2(\O)} \leqslant {\mathcal C}_{12}''' \|\r\|_{L_2(\O)}.
$$
Consequently,
\begin{equation}
\label{t4}
|\wt{\mathcal J}_\eps^{(1)}[\bzeta]| \leqslant 3 {\mathcal C}'''_{12}
\eps \| \r\|_{L_2(\O)} \| \rot (\mu^\eps)^{-1/2} \bzeta \|_{L_2(\O)}.
\end{equation}

As a result, relations \eqref{rhs5}, \eqref{t3}, and \eqref{t4} imply the required  estimate \eqref{t0} with the constant
${\mathcal C}_{12} = 3 (\wh{\mathcal C}_{12} + {\mathcal C}_{12}''')$. 
\end{proof}

The term \eqref{rhs4a} is estimated by Proposition~\ref{prop_Seps - I}:
$$
|{\mathcal J}_\eps^{(4)}[\bzeta]| \leqslant \eps r_1
\| \eta^{-1}\|_{L_\infty}
\| \mu^{-1}\|_{L_\infty}^{1/2} \| \wt{\bvarphi}_0 + \wt{\brho}_\eps\|_{H^2(\R^3)}   \| \rot (\mu^\eps)^{-1/2} \bzeta \|_{L_2(\O)}.
$$
Taking  \eqref{Mrr11} and \eqref{Mrr14} into account, we arrive at 
\begin{equation}
\label{t44}
|{\mathcal J}_\eps^{(4)}[\bzeta]| 
\leqslant {\mathcal C}_{13} \eps \| \r \|_{L_2( \O)}  \| \rot (\mu^\eps)^{-1/2} \bzeta \|_{L_2(\O)},
\end{equation}
where ${\mathcal C}_{13} =  r_1   \| \eta^{-1}\|_{L_\infty}
\| \mu^{-1}\|_{L_\infty}^{1/2}({\mathcal C}_1 + {\mathcal C}_2)$.

We transform the term \eqref{rhs3}, using Lemma \ref{lem5.1}. Let ${\bzeta}^0_\eps = T_\eps \bzeta$.
  Then ${\bzeta}^0_\eps \in \Dom {\mathfrak l}_0$ and 
  $\rot (\mu^\eps)^{-1/2} \bzeta = \rot (\mu^0)^{-1/2} {\bzeta}^0_\eps$.
Hence,
\begin{equation*}
{\mathcal J}_\eps^{(2)}[\bzeta] = 
((\eta^0)^{-1} \rot (\mu^0)^{-1/2} \bvarphi_0, \rot (\mu^0)^{-1/2} {\bzeta}^0_\eps)_{L_2(\O)}.
\end{equation*}
Since $\bvarphi_0$ is the  solution of  problem \eqref{Mrr4}, we have  
\begin{equation}
\label{t6}
{\mathcal J}_\eps^{(2)}[\bzeta] = 
- ( \bvarphi_0,  {\bzeta}^0_\eps)_{L_2(\O)} + i ((\mu^0)^{-1/2} \r, {\bzeta}^0_\eps)_{L_2(\O)}.
\end{equation}
  
 The term \eqref{rhs4} is transformed similarly:
\begin{equation*}
{\mathcal J}_\eps^{(3)}[\bzeta] = 
((\eta^0)^{-1} \rot (\mu^0)^{-1/2} \brho_\eps, \rot (\mu^0)^{-1/2} {\bzeta}^0_\eps)_{L_2(\O)}.
\end{equation*}
Since $\brho_\eps$ is the  solution of  problem \eqref{Mrr11a}, we have 
\begin{equation}
\label{t8}
{\mathcal J}_\eps^{(3)}[\bzeta] = 
- ( \brho_\eps, {\bzeta}^0_\eps)_{L_2(\O)} + i ((\mu^0)^{-1/2} \r_\eps, {\bzeta}^0_\eps)_{L_2(\O)}.
\end{equation}

\subsection{Further analysis of  ${\mathcal R}_\eps[\bzeta]$}
From \eqref{rhs1}, \eqref{t6}, and \eqref{t8} it follows that the functional \eqref{pogr0} can be represented as
\begin{equation}
\label{t9}
\begin{aligned}
{\mathcal R}_\eps[\bzeta]=&
{\mathcal J}_\eps^{(1)}[\bzeta] + {\mathcal J}_\eps^{(4)}[\bzeta] 
 +((W^\eps_\mu)^* S_\eps (\wt{\bvarphi}_0 + \wt{\brho}_\eps),\bzeta)_{L_2(\O)}
- ( \bvarphi_0, {\bzeta}^0_\eps)_{L_2(\O)}
\\  
&- ( \brho_\eps, {\bzeta}^0_\eps)_{L_2(\O)} + i ((\mu^0)^{-1/2} \r_\eps, {\bzeta}^0_\eps)_{L_2(\O)}
+i (\r, (\mu^0)^{-1/2} {\bzeta}^0_\eps - (\mu^\eps)^{-1/2} \bzeta)_{L_2(\O)}.
\end{aligned}
\end{equation}

By \eqref{lemm12}, the last summand in  \eqref{t9} is equal to zero: 
\begin{equation}
\label{t10}
(\r, (\mu^0)^{-1/2} {\bzeta}^0_\eps - (\mu^\eps)^{-1/2} \bzeta)_{L_2(\O)} =
(\r, \nabla (\phi_{\eps,1} - \phi_{\eps,2}))_{L_2(\O)}=0,
\end{equation}
since $\r \in J_0(\O)$ (see \eqref{5.2}).

Using \eqref{Mrr15}, we represent the third term in  \eqref{t9} as
\begin{equation}
\label{t11}
\begin{aligned}
&((W^\eps_\mu)^* S_\eps (\wt{\bvarphi}_0 + \wt{\brho}_\eps),\bzeta)_{L_2(\O)} =
( \wt{\mu}^\eps (\mu^0)^{-1/2} S_\eps (\wt{\bvarphi}_0 + \wt{\brho}_\eps), (\mu^\eps)^{-1/2}\bzeta)_{L_2(\O)}
\\
&= 
{\mathcal J}_\eps^{(5)}[\bzeta] + {\mathcal J}_\eps^{(6)}[\bzeta]
+ 
(  (\mu^0)^{1/2}  ({\bvarphi}_0 + {\brho}_\eps), (\mu^\eps)^{-1/2}\bzeta)_{L_2(\O)},
\end{aligned}
\end{equation}
where
\begin{align}
\label{t12}
{\mathcal J}_\eps^{(5)}[\bzeta] &=( (\wt{\mu}^\eps - \mu^0) (\mu^0)^{-1/2} S_\eps (\wt{\bvarphi}_0 + \wt{\brho}_\eps), (\mu^\eps)^{-1/2}\bzeta)_{L_2(\O)},
\\
\label{t13}
{\mathcal J}_\eps^{(6)}[\bzeta] &= ( (\mu^0)^{1/2} (S_\eps -I) (\wt{\bvarphi}_0 + \wt{\brho}_\eps), (\mu^\eps)^{-1/2}\bzeta)_{L_2(\O)}.
\end{align}
Now, relations \eqref{t9}--\eqref{t11} imply that
\begin{equation}
\label{ttt9}
\begin{aligned}
&{\mathcal R}_\eps[\bzeta]=
{\mathcal J}_\eps^{(1)}[\bzeta] + {\mathcal J}_\eps^{(4)}[\bzeta] 
 + {\mathcal J}_\eps^{(5)}[\bzeta] + {\mathcal J}_\eps^{(6)}[\bzeta]
 \\
 &+ i ((\mu^0)^{-1/2} \r_\eps, {\bzeta}^0_\eps)_{L_2(\O)}
+ ((\mu^0)^{1/2}  ({\bvarphi}_0 + {\brho}_\eps), (\mu^\eps)^{-1/2} \bzeta - (\mu^0)^{-1/2} {\bzeta}^0_\eps)_{L_2(\O)}.
\end{aligned}
\end{equation}
By \eqref{lemm12}, the last summand in  \eqref{ttt9} is equal to zero: 
\begin{equation}
\label{ttt10}
\begin{aligned}
&((\mu^0)^{1/2}  ({\bvarphi}_0 + {\brho}_\eps), (\mu^\eps)^{-1/2} \bzeta - (\mu^0)^{-1/2} {\bzeta}^0_\eps)_{L_2(\O)}
\\
&=((\mu^0)^{1/2}  ({\bvarphi}_0 + {\brho}_\eps), \nabla (\phi_{\eps,2} - \phi_{\eps,1}))_{L_2(\O)}=0,
\end{aligned}
\end{equation}
because $ (\mu^0)^{1/2}({\bvarphi}_0 + {\brho}_\eps)  \in J_0(\O)$.
According to  \eqref{ttt9} and \eqref{ttt10},
\begin{align}
\label{ttt11}
{\mathcal R}_\eps[\bzeta]=&
{\mathcal J}_\eps^{(1)}[\bzeta] + {\mathcal J}_\eps^{(4)}[\bzeta] 
 + {\mathcal J}_\eps^{(5)}[\bzeta] + {\mathcal J}_\eps^{(6)}[\bzeta]
 +{\mathcal J}_\eps^{(7)}[\bzeta],
 \\
 \label{ttt12}
 {\mathcal J}_\eps^{(7)}[\bzeta] :=&
 i ((\mu^0)^{-1/2} \r_\eps,  {\bzeta}^0_\eps)_{L_2(\O)}.
\end{align}

\subsection{Estimates for the terms ${\mathcal J}_\eps^{(5)}[\bzeta]$ and ${\mathcal J}_\eps^{(6)}[\bzeta]$}
The following lemma is traditional for homogenization theory. For completeness, we give  the proof.

\begin{lemma}
Suppose that $\wt{\mu}(\x)$ is given by  \eqref{eff5} and $\mu^0$ is the effective matrix \eqref{eff6}.
There exist functions $M_{lj}^{(i)}\in \wt{H}^1(\Omega)$, $i,l,j=1,2,3,$ such that
\begin{align}
\label{t14}
\wt{\mu}_{li}(\x) - \mu^0_{li} = \sum_{j=1}^3 \partial_j M^{(i)}_{lj}(\x), \quad l, i =1,2,3, 
\\
\label{t15}
M_{lj}^{(i)}(\x) = - M^{(i)}_{jl}(\x),\quad i,l,j=1,2,3.
\end{align}
We have
\begin{equation}
\label{t16}
\begin{aligned}
\| M_{lj}^{(i)}   \|_{L_2(\Omega)} \leqslant C_M' |\Omega|^{1/2},
\quad C_M' := r_0^{-1} \| \mu\|_{L_\infty},\quad i,l,j=1,2,3,
\\
\| \nabla M_{lj}^{(i)}   \|_{L_2(\Omega)} \leqslant C_M'' |\Omega|^{1/2},
\quad C_M'' :=  2 \| \mu\|_{L_\infty},\quad i,l,j=1,2,3.
\end{aligned}
\end{equation}
\end{lemma}

\begin{proof}
Let $U_{li}(\x)$ be the  $\Gamma$-periodic solution of the equation
\begin{equation}
\label{t17}
\Delta U_{li}(\x) = \wt{\mu}_{li} (\x) - \mu^0_{li}.
\end{equation}
The solvability condition is ensured by  \eqref{eff6}.
Since the right-hand side of  \eqref{t17} belongs to  $L_2(\Omega)$, then $U_{li} \in \wt{H}^2(\Omega)$. 
We put
\begin{equation}
\label{t18}
M_{lj}^{(i)}(\x) := \partial_j U_{li}(\x) - \partial_l U_{ji}(\x), \quad i,l,j =1,2,3.
\end{equation}
Obviously,  \eqref{t15} holds. 
By \eqref{eff4}, the divergence of the columns of the matrix $\wt{\mu}(\x) - \mu^0$ is equal to zero, i.~e., 
$$
\sum_{j=1}^3 \partial_j (\wt{\mu}_{ji}(\x) - \mu^0_{ji})=0, \quad i=1,2,3.
$$
Denote $f_i(\x) := \sum_{j=1}^3 \partial_j U_{ji}(\x)$. Then, according to \eqref{t17},
$$
\Delta f_i(\x) = \sum_{j=1}^3 \partial_j \Delta U_{ji}(\x) = \sum_{j=1}^3 \partial_j (\wt{\mu}_{ji}(\x) - \mu^0_{ji}) =0.
$$
Thus, $f_i(\x)$ is the periodic solution of the Laplace equation and it has zero mean value.
Hence, $f_i(\x)=0$. Taking \eqref{t18} into account, we obtain 
$$
\sum_{j=1}^3 \partial_j M^{(i)}_{lj}(\x) = \sum_{j=1}^3 \partial_j^2 U_{li}(\x) - \sum_{j=1}^3 
\partial_j \partial_l U_{ji}(\x) 
= \Delta U_{li}(\x) - \partial_l f_i(\x) = \Delta U_{li}(\x).
$$
Together with   \eqref{t17}, this implies the required identity \eqref{t14}.

By the Fourier series, it is easily seen that  the periodic solution of equation  \eqref{t17} satisfies
$$
\begin{aligned}
\| \nabla U_{li}\|_{L_2(\Omega)} &\leqslant (2r_0)^{-1} \| \wt{\mu}_{li} - \mu^0_{li}\|_{L_2(\Omega)},
\\
\| \nabla_2 U_{li}\|_{L_2(\Omega)} &\leqslant \| \wt{\mu}_{li} - \mu^0_{li}\|_{L_2(\Omega)}.
\end{aligned}
$$
In its turn, it follows from equation \eqref{eff4}  that 
$$
\| \wt{\mu}_{li} - \mu^0_{li}\|_{L_2(\Omega)} \leqslant \| \wt{\mu}_{li}\|_{L_2(\Omega)}
\leqslant \|\mu\|_{L_\infty} |\Omega|^{1/2}. 
$$
This leads to the required estimates \eqref{t16}.
\end{proof}

\begin{lemma}\label{lem5.5}
For $0< \eps \leqslant \eps_1$ the term \eqref{t12} satisfies 
\begin{equation}\label{t19}
| {\mathcal J}_\eps^{(5)}[\bzeta]| \leqslant {\mathcal C}_{14} \eps^{1/2} \| \r \|_{L_2(\O)}
\left( \|\bzeta\|_{L_2(\O)} + \|\div (\mu^\eps)^{1/2} \bzeta\|_{L_2(\O)}
+ \|\rot (\mu^\eps)^{-1/2} \bzeta\|_{L_2(\O)} \right).
\end{equation}
The constant ${\mathcal C}_{14}$ depends on the norms $\| \eta \|_{L_\infty}$, $\| \eta^{-1} \|_{L_\infty}$,
$\| \mu \|_{L_\infty}$, $\| \mu^{-1} \|_{L_\infty}$, the parameters of the lattice $\Gamma$, and the domain $\O$.
\end{lemma}

\begin{proof}
Denote $[(\mu^0)^{-1/2} S_\eps (\wt{\bvarphi}_0 + \wt{\brho}_\eps ) ]_i =: q_{\eps, i}$, $i=1,2,3$.
We represent the term \eqref{t12} as 
\begin{equation}\label{t19a}
{\mathcal J}_\eps^{(5)}[\bzeta] = \sum_{l,i=1}^3 \left( (\wt{\mu}_{li}^\eps - \mu^0_{li}) q_{\eps,i},
[ (\mu^\eps)^{-1/2} \bzeta]_l \right)_{L_2(\O)}.
\end{equation}
By \eqref{t14}, 
\begin{equation}\label{t19b}
\wt{\mu}_{li}^\eps - \mu^0_{li} = \eps \sum_{j=1}^3 \partial_j (M_{lj}^{(i)})^\eps,
\quad l, i =1,2,3.
\end{equation}
Consequently,
\begin{equation}\label{t20}
{\mathcal J}_\eps^{(5)}[\bzeta] = \wh{\mathcal J}_\eps^{(5)}[\bzeta] - \wt{\mathcal J}_\eps^{(5)}[\bzeta],
\end{equation}
where 
\begin{align}\label{t21}
\wh{\mathcal J}_\eps^{(5)}[\bzeta] &= \eps \sum_{l,i,j=1}^3 \left( \partial_j ((M_{lj}^{(i)})^\eps q_{\eps,i}),
[ (\mu^\eps)^{-1/2} \bzeta]_l \right)_{L_2(\O)},
\\
\label{t22}
\wt{\mathcal J}_\eps^{(5)}[\bzeta] &= \eps \sum_{l,i,j=1}^3 \left(  (M_{lj}^{(i)})^\eps  \partial_j q_{\eps,i},
[ (\mu^\eps)^{-1/2} \bzeta]_l \right)_{L_2(\O)}.
\end{align}

Let $\theta_\eps$ be the cut-off function satisfying \eqref{srezka}. 
We represent the term  \eqref{t21} as
\begin{align}
\label{t23}
&\wh{\mathcal J}_\eps^{(5)}[\bzeta] = {\mathcal J}_\eps^{(8)}[\bzeta]
+ {\mathcal J}_\eps^{(9)}[\bzeta],
\\
\label{t24}
&{\mathcal J}_\eps^{(8)}[\bzeta]=
\eps \sum_{l,i,j=1}^3 \left( \partial_j ( \theta_\eps (M_{lj}^{(i)})^\eps q_{\eps,i}),
[ (\mu^\eps)^{-1/2} \bzeta]_l \right)_{L_2(\O)}, 
\\
\label{t25}
&{\mathcal J}_\eps^{(9)}[\bzeta]=
\eps \sum_{l,i,j=1}^3 \left( \partial_j ( (1-\theta_\eps) (M_{lj}^{(i)})^\eps q_{\eps,i}),
[ (\mu^\eps)^{-1/2} \bzeta]_l \right)_{L_2(\O)}.
\end{align}

Consider the term \eqref{t24}. We have
\begin{equation}\label{t26}
\eps \partial_j ( \theta_\eps (M_{lj}^{(i)})^\eps q_{\eps,i})
= \eps (\partial_j  \theta_\eps) (M_{lj}^{(i)})^\eps q_{\eps,i} +
\theta_\eps (\partial_j M_{lj}^{(i)})^\eps q_{\eps,i} + 
\eps \theta_\eps (M_{lj}^{(i)})^\eps \partial_j q_{\eps,i}. 
\end{equation}
The first summand on the right is estimated by \eqref{srezka} and Lemma \ref{lem02}:
$$
\begin{aligned}
&\| \eps (\partial_j  \theta_\eps) (M_{lj}^{(i)})^\eps q_{\eps,i} \|_{L_2(\O)}
\leqslant \kappa \| (M_{lj}^{(i)})^\eps q_{\eps,i} \|_{L_2((\partial \O)_{2\eps} )} 
\\
&\leqslant
\eps^{1/2} \kappa \beta_*^{1/2} |\Omega|^{-1/2} \| M_{lj}^{(i)} \|_{L_2(\Omega)} 
\| \mu^{-1}\|_{L_\infty}^{1/2} \| \wt{\bvarphi}_0 + \wt{\brho}_\eps\|_{H^1(\R^3)},
\ 0 < \eps \le \eps_1.
\end{aligned}
$$
Together with \eqref{Mrr11}, \eqref{Mrr14}, and \eqref{t16}, this implies
$$
\|  \eps (\partial_j  \theta_\eps) (M_{lj}^{(i)})^\eps q_{\eps,j} \|_{L_2(\O)}
\leqslant {\mathcal C}_{14}' \eps^{1/2} \| \r \|_{L_2( \O )},
\quad 0 < \eps \le \eps_1,
$$
where ${\mathcal C}_{14}' = \kappa \beta_*^{1/2}  C_M' 
\| \mu^{-1}\|_{L_\infty}^{1/2}({\mathcal C}_1 + {\mathcal C}_2)$.
Similarly, to estimate the second term in  \eqref{t26}, we apply  \eqref{srezka} and Lemma \ref{lem02}:
$$
\begin{aligned}
&\| \theta_\eps (\partial_j M_{lj}^{(i)})^\eps q_{\eps,i} \|_{L_2(\O)}
\\
&\leqslant \eps^{1/2} \beta_*^{1/2} |\Omega|^{-1/2} \| \partial_j M_{lj}^{(i)}\|_{L_2(\Omega)} 
\| \mu^{-1}\|_{L_\infty}^{1/2} \| \wt{\bvarphi}_0 + \wt{\brho}_\eps\|_{H^1(\R^3)}
\end{aligned}
$$
for $0 < \eps \le \eps_1$. 
Taking \eqref{Mrr11}, \eqref{Mrr14}, and \eqref{t16} into account, we obtain
$$
\| \theta_\eps (\partial_j M_{lj}^{(i)})^\eps q_{\eps,i} \|_{L_2(\O)}
\leqslant {\mathcal C}_{14}'' \eps^{1/2} \| \r \|_{L_2( \O )},
\quad 0 < \eps \le \eps_1,
$$
where ${\mathcal C}_{14}'' = \beta_*^{1/2}  C_M'' 
\| \mu^{-1}\|_{L_\infty}^{1/2}({\mathcal C}_1 + {\mathcal C}_2)$.
The third term in the right-hand side of \eqref{t26} is estimated with the help of Proposition \ref{prop f^eps S_eps}:
\begin{equation}\label{tt1}
\eps \|\theta_\eps (M_{lj}^{(i)})^\eps \partial_j q_{\eps,i} \|_{L_2(\O)}
\leqslant 
 \eps |\Omega|^{-1/2} \| M_{lj}^{(i)}  \|_{L_2(\Omega)} 
 \| \mu^{-1}\|_{L_\infty}^{1/2} \| \wt{\bvarphi}_0 + \wt{\brho}_\eps\|_{H^1(\R^3)}.
\end{equation}
Together with  \eqref{Mrr11}, \eqref{Mrr14}, and \eqref{t16}, this yields
\begin{equation}\label{tt2}
\eps \|\theta_\eps (M_{lj}^{(i)})^\eps \partial_j q_{\eps,i} \|_{L_2(\O)}
\leqslant  {\mathcal C}_{14}''' \eps \| \r \|_{L_2(\O)},
\end{equation}
where ${\mathcal C}_{14}''' =  C_M' \| \mu^{-1}\|_{L_\infty}^{1/2}({\mathcal C}_1 + {\mathcal C}_2)$.
As a result, we arrive at the estimate
\begin{equation}\label{tt2a}
\eps \| \partial_j ( \theta_\eps (M_{lj}^{(i)})^\eps q_{\eps,i}) \|_{L_2(\O)}
\leqslant \check{\mathcal C}_{14} \eps^{1/2} \| \r \|_{L_2(\O)}, \quad 0< \eps \leqslant \eps_1,
\end{equation}
where $\check{\mathcal C}_{14} = {\mathcal C}_{14}'+ {\mathcal C}_{14}'' + {\mathcal C}_{14}'''$.
This implies the following estimate for the term \eqref{t24}: 
\begin{equation}\label{tt3}
| {\mathcal J}_\eps^{(8)}[\bzeta]| \leqslant 9 \sqrt{3} \check{\mathcal C}_{14} \|\mu^{-1}\|_{L_\infty}^{1/2} \eps^{1/2} \|\r\|_{L_2(\O)} \| \bzeta \|_{L_2(\O)},
\quad 0< \eps \leqslant \eps_1.
\end{equation}

Now, we consider the term \eqref{t25}. By  \eqref{lemm12},
$(\mu^\eps)^{-1/2} \bzeta = (\mu^0)^{-1/2} {\bzeta}^0_\eps + \nabla(\phi_{\eps,2} - \phi_{\eps,1})$.
This allows us to write the term \eqref{t25} in the form
\begin{equation}\label{t27}
{\mathcal J}_\eps^{(9)}[\bzeta]=
\eps \sum_{l,i,j=1}^3 \left( \partial_j ( (1-\theta_\eps) (M_{lj}^{(i)})^\eps q_{\eps,i}),
[ (\mu^0)^{-1/2} {\bzeta}^0_\eps]_l \right)_{L_2(\O)}.
\end{equation}
We have taken into account that
$$
 \begin{aligned}
 &\sum_{l,j=1}^3 \left( \partial_j ( (1-\theta_\eps) (M_{lj}^{(i)})^\eps q_{\eps,i}),
 \partial_l  (\phi_{\eps,2} - \phi_{\eps,1}) \right)_{L_2(\O)}
 \\
& = - \sum_{l,j=1}^3 \left(  (1-\theta_\eps) (M_{lj}^{(i)})^\eps q_{\eps,i},
 \partial_j \partial_l  (\phi_{\eps,2} - \phi_{\eps,1}) \right)_{L_2(\O)}=0,
\end{aligned}
$$
which can be checked by  integration by parts and using  \eqref{t15}.
(When checking, we can replace the function $\phi_{\eps,2} - \phi_{\eps,1}$ by $\phi \in H^2(\O)$,
and next we close the result by continuity.)

Relation ${\bzeta}^0_\eps \in H^1(\O;\C^3)$ allows us to integrate by parts in   \eqref{t27}:
\begin{equation*}
{\mathcal J}_\eps^{(9)}[\bzeta]=
- \eps \sum_{l,i,j=1}^3 \left(  (1-\theta_\eps) (M_{lj}^{(i)})^\eps q_{\eps,i},
\partial_j [ (\mu^0)^{-1/2} {\bzeta}^0_\eps]_l \right)_{L_2(\O)}.
\end{equation*}
Similarly to \eqref{tt1} and \eqref{tt2}, we have
$$
 \| (1-\theta_\eps) (M_{lj}^{(i)})^\eps q_{\eps,i}\|_{L_2(\O)} \leqslant {\mathcal C}_{14}''' \| \r \|_{L_2(\O)}.
$$
Together with  \eqref{lemmm3}, this implies
\begin{equation}\label{t28}
| {\mathcal J}_\eps^{(9)}[\bzeta] | \leqslant 9 {\mathcal C}'''_{14} {\mathcal C}_{11} \|\mu^{-1}\|_{L_\infty}^{1/2} \eps 
\| \r \|_{L_2(\O)} 
\left( \|\bzeta\|_{L_2(\O)} + \|\div (\mu^\eps)^{1/2} \bzeta\|_{L_2(\O)}
+ \|\rot (\mu^\eps)^{-1/2} \bzeta\|_{L_2(\O)} \right).
\end{equation}

Finally, the term \eqref{t22} is estimated similarly to \eqref{tt1}, \eqref{tt2}:
\begin{equation}\label{t29}
| \wt{\mathcal J}_\eps^{(5)}[\bzeta] | \leqslant \wt{\mathcal C}_{14} \eps 
\| \r \|_{L_2(\O)} \| \bzeta\|_{L_2(\O)},
\end{equation}
where $\wt{\mathcal C}_{14} = 9 \sqrt{3} {\mathcal C}_{14}''' \|\mu^{-1}\|_{L_\infty}^{1/2}$.

As a result, relations \eqref{t20}, \eqref{t23}, \eqref{tt3}, \eqref{t28}, and \eqref{t29} imply estimate  \eqref{t19} with the constant 
${\mathcal C}_{14}= 9 \|\mu^{-1}\|_{L_\infty}^{1/2} (\sqrt{3} \check{\mathcal C}_{14} + {\mathcal C}'''_{14} {\mathcal C}_{11})+ \wt{\mathcal C}_{14}$.
\end{proof}

The term \eqref{t13} is estimated by Proposition \ref{prop_Seps - I}:
$$
| {\mathcal J}_\eps^{(6)}[\bzeta] | \leqslant \eps r_1 \|\mu\|^{1/2}_{L_\infty} \|\mu^{-1}\|^{1/2}_{L_\infty}
\| \wt{\bvarphi}_0 + \wt{\brho}_\eps \|_{H^1(\R^3)} \| \bzeta \|_{L_2(\O)}.
$$
Together with \eqref{Mrr11} and \eqref{Mrr14}, this yields
\begin{equation}\label{t30}
| {\mathcal J}_\eps^{(6)}[\bzeta] | \leqslant {\mathcal C}_{15} \eps 
\| \r \|_{L_2(\O)} \| \bzeta\|_{L_2(\O)},
\end{equation}
where ${\mathcal C}_{15} =  r_1 ({\mathcal C}_{1}+ {\mathcal C}_2) \|\mu\|_{L_\infty}^{1/2}\|\mu^{-1}\|_{L_\infty}^{1/2}$.

\subsection{Estimate for the term ${\mathcal J}_\eps^{(7)}[\bzeta]$\label{sec5.5}}
Using \eqref{right}, we represent the term  \eqref{ttt12} in the form
\begin{equation}\label{t31}
 {\mathcal J}_\eps^{(7)}[\bzeta] =
 i ((\mu^0)^{-1/2} {\mathcal P}^0_{\mu^0} S_\eps (Y_\mu^\eps)^* \wt{\r}, {\bzeta}^0_\eps)_{L_2(\O)}.
\end{equation}

\begin{remark}\label{rem5.6}
Let  ${\mathcal P}^0_{\mu^0}$ be the orthogonal projection of  $L_2(\O;(\mu^0)^{-1})$ onto $J_0(\O)$.
It is easily seen that   ${\mathcal P}^0_{\mu^0} \f \in H^1(\O;\C^3)$ if $\f \in H^1(\O;\C^3)$, and 
\begin{equation}\label{t32}
 \| {\mathcal P}^0_{\mu^0} \f \|_{H^1(\O)} \leqslant {\mathfrak c} \| \f \|_{H^1(\O)}.
\end{equation}
The constant $\mathfrak c$ depends only on the norms $\| \mu \|_{L_\infty}$, $\| \mu^{-1} \|_{L_\infty}$, and the domain $\O$.
\end{remark}

  Since the operator  ${\mathcal P}^0_{\mu^0}$ is selfadjoint in the weighted space $L_2(\O;(\mu^0)^{-1})$,
  the functional \eqref{t31} can be written as
 \begin{equation}\label{t33}
 {\mathcal J}_\eps^{(7)}[\bzeta] =
 i ((\mu^0)^{-1}  S_\eps (Y_\mu^\eps)^* \wt{\r},
 {\mathcal P}^0_{\mu^0} (\mu^0)^{1/2} {\bzeta}^0_\eps)_{L_2(\O)}.
\end{equation}

 Let $\theta_\eps$ be the cut-off function satisfying \eqref{srezka}.
We write the term \eqref{t33} as the sum of two terms:
\begin{align}\label{t34}
 &{\mathcal J}_\eps^{(7)}[\bzeta] = \wh{\mathcal J}_\eps^{(7)}[\bzeta] + \wt{\mathcal J}_\eps^{(7)}[\bzeta],
 \\
\label{t35} 
&\wh{\mathcal J}_\eps^{(7)}[\bzeta]: =
 i ((\mu^0)^{-1}  S_\eps (Y_\mu^\eps)^* \wt{\r},
 {\theta}_\eps {\mathcal P}^0_{\mu^0} (\mu^0)^{1/2} {\bzeta}^0_\eps)_{L_2(\O)},
\\
\label{t36}
&\wt{\mathcal J}_\eps^{(7)}[\bzeta] :=
 i ((\mu^0)^{-1}  S_\eps (Y_\mu^\eps)^* \wt{\r}, (1- {\theta}_\eps)
 {\mathcal P}^0_{\mu^0} (\mu^0)^{1/2} {\bzeta}^0_\eps)_{L_2(\O)}.
 \end{align}

 The term \eqref{t35} is estimated with the help of Proposition \ref{prop f^eps S_eps}, \eqref{eff6b}, and Lemma \ref{lem01}:
 $$
 \begin{aligned}
 &| \wh{\mathcal J}_\eps^{(7)}[\bzeta] | \leqslant \|\mu^{-1}\|_{L_\infty} \| S_\eps (Y_\mu^\eps)^* \wt{\r} \|_{L_2(\R^3)}
 \| {\mathcal P}^0_{\mu^0} (\mu^0)^{1/2} {\bzeta}^0_\eps \|_{L_2(B_{2\eps})}
  \\
  &\leqslant \eps^{1/2} 
  \|\mu \|^{1/2}_{L_\infty} \|\mu^{-1}\|^{3/2}_{L_\infty}  \beta_0^{1/2} \|\r\|_{L_2(\O)}
  \|{\mathcal P}^0_{\mu^0} (\mu^0)^{1/2} {\bzeta}^0_\eps \|_{H^1(\O)}, \quad 0< \eps \le \eps_0.
 \end{aligned}
 $$
Taking  \eqref{lemmm3} and \eqref{t32} into account, we arrive at the estimate
\begin{equation}
\label{q00}
 | \wh{\mathcal J}_\eps^{(7)}[\bzeta] | \leqslant 
 {\mathcal C}_{16} \eps^{1/2} \| \r \|_{L_2(\O)}
 \left( \| \bzeta\|_{L_2(\O)} + \| \div (\mu^\eps)^{1/2}\bzeta\|_{L_2(\O)}
 + \| \rot (\mu^\eps)^{-1/2} \bzeta\|_{L_2(\O)} \right)
 \end{equation}
 for $0< \eps \le \eps_0$,
where ${\mathcal C}_{16} =  {\mathfrak c} \|\mu \|_{L_\infty} \|\mu^{-1}\|^{3/2}_{L_\infty}  \beta_0^{1/2} {\mathcal C}_{11}$.

Now, we consider  the term \eqref{t36}. 
 The function $(1- {\theta}_\eps){\mathcal P}^0_{\mu^0} (\mu^0)^{1/2} {\bzeta}^0_\eps$ belongs to $H^1(\O;\C^3)$ 
 and is equal to zero on  $\partial \O$. We extend this function by zero to $\R^3 \setminus \O$;
the extended function is denoted by  $\p_\eps$. Note that $\p_\eps \in H^1(\R^3;\C^3)$.
We rewrite the term \eqref{t36} in the form
\begin{equation*}
\wt{\mathcal J}_\eps^{(7)}[\bzeta] :=
 i ( {\r}, Y_\mu^\eps S_\eps (\mu^0)^{-1} \p_\eps )_{L_2(\O)}.
\end{equation*}
 The columns of the matrix $Y_\mu^\eps$ are given by $\eps \nabla \Psi_j^\eps$, $j=1,2,3$.
Hence,
\begin{equation*}
\wt{\mathcal J}_\eps^{(7)}[\bzeta] =
  i \eps \sum_{j=1}^3 ( {\r}, \nabla \left( \Psi_j^\eps [S_\eps (\mu^0)^{-1} \p_\eps]_j\right) )_{L_2(\O)}
- i \eps \sum_{j=1}^3 ( {\r},  \Psi_j^\eps \nabla [S_\eps (\mu^0)^{-1} \p_\eps]_j )_{L_2(\O)}.
\end{equation*}
The first summand on the right is equal to zero, since $\r \in J_0(\O)$.
Then, by Proposition \ref{prop f^eps S_eps} and \eqref{efff}, we obtain
\begin{equation}
\label{q1}
| \wt{\mathcal J}_\eps^{(7)}[\bzeta] | \leqslant  3\eps  (2r_0)^{-1} \|\mu\|_{L_\infty}^{1/2} 
\|\mu^{-1} \|_{L_\infty}^{3/2}\| \r \|_{L_2(\O)} \| \nabla \p_\eps\|_{L_2(\O)}. 
\end{equation}
Consider the derivatives
$$
\begin{aligned}
\partial_l \p_\eps =
-  (\partial_l \theta_\eps) {\mathcal P}^0_{\mu^0} (\mu^0)^{1/2} {\bzeta}^0_\eps 
+  (1-\theta_\eps)  \partial_l   {\mathcal P}^0_{\mu^0} (\mu^0)^{1/2} {\bzeta}^0_\eps.
\end{aligned}
$$
 Consequently, from \eqref{srezka},  Lemma \ref{lem01}, and \eqref{t32} it follows that
 $$
 \begin{aligned}
&\eps \| \nabla \p_\eps\|_{L_2(\O)} \leqslant \kappa \| {\mathcal P}^0_{\mu^0} (\mu^0)^{1/2} {\bzeta}^0_\eps
\|_{L_2(B_{2\eps})} + \eps \| {\mathcal P}^0_{\mu^0} (\mu^0)^{1/2} {\bzeta}^0_\eps\|_{H^1(\O)}
\\
&\leqslant (\eps^{1/2} \kappa \beta_0^{1/2} + \eps) {\mathfrak c} \| \mu\|_{L_\infty}^{1/2} 
\| {\bzeta}^0_\eps\|_{H^1(\O)}, \quad 0< \eps \le \eps_0.
\end{aligned}
 $$
Together with \eqref{lemmm3}, this implies
 \begin{equation}
 \label{q2}
 \eps \| \nabla \p_\eps\|_{L_2(\O)} \leqslant 
 {\mathcal C}_{17} \eps^{1/2} 
 \left( \| \bzeta\|_{L_2(\O)} + \| \div (\mu^\eps)^{1/2}\bzeta \|_{L_2(\O)} 
 + \| \rot (\mu^\eps)^{-1/2}\bzeta \|_{L_2(\O)}\right)
\end{equation}
for $0< \eps \le \eps_0$,
where ${\mathcal C}_{17}= (\kappa \beta_0^{1/2} +1)  {\mathfrak c} \| \mu\|_{L_\infty}^{1/2} {\mathcal C}_{11}$.
Now, relations \eqref{q1} and \eqref{q2} yield
\begin{equation}
\label{q3}
| \wt{\mathcal J}_\eps^{(7)}[\bzeta] | \leqslant {\mathcal C}_{18} \eps^{1/2}
\| \r \|_{L_2(\O)}
 \left( \| \bzeta\|_{L_2(\O)} + \| \div (\mu^\eps)^{1/2}\bzeta \|_{L_2(\O)} 
 + \| \rot (\mu^\eps)^{-1/2}\bzeta \|_{L_2(\O)}\right)
\end{equation}
for $0< \eps \le \eps_0$, where
${\mathcal C}_{18}= 3 (2r_0)^{-1} \|\mu\|_{L_\infty}^{1/2} 
\|\mu^{-1} \|_{L_\infty}^{3/2} {\mathcal C}_{17}$.

Combining \eqref{t34}, \eqref{q00}, and \eqref{q3}, we obtain
\begin{equation}
\label{q4}
\begin{aligned}
&| {\mathcal J}_\eps^{(7)}[\bzeta] | \leqslant ({\mathcal C}_{16} + {\mathcal C}_{18}) \eps^{1/2}
\| \r \|_{L_2(\O)}
 \\
 &\times
  \left( \| \bzeta\|_{L_2(\O)} + \| \div (\mu^\eps)^{1/2}\bzeta \|_{L_2(\O)} 
 + \| \rot (\mu^\eps)^{-1/2}\bzeta \|_{L_2(\O)}\right),\quad 0< \eps \le \eps_0.
 \end{aligned}
 \end{equation}

\subsection{Taking the boundary condition into account. Completion of the proof of  Theorem \ref{lem4.6}}
Finally, relations \eqref{t0}, \eqref{t44}, \eqref{ttt11}, \eqref{t19}, \eqref{t30}, and \eqref{q4} imply the following estimate
for the functional \eqref{pogr0}:
\begin{equation}
\label{q5}
| {\mathcal R}_\eps [\bzeta] | \leqslant {\mathcal C}^\circ \eps^{1/2}
\| \r \|_{L_2(\O)}
  \left( \| \bzeta\|_{L_2(\O)} + \| \div (\mu^\eps)^{1/2}\bzeta \|_{L_2(\O)} 
 + \| \rot (\mu^\eps)^{-1/2}\bzeta \|_{L_2(\O)}\right),\ 0< \eps \le \eps_1,
 \end{equation}
where ${\mathcal C}^\circ={\mathcal C}_{12} + {\mathcal C}_{13} + {\mathcal C}_{14} + {\mathcal C}_{15}
+ {\mathcal C}_{16} + {\mathcal C}_{18}$.

To take into account the boundary condition \eqref{pogr3}, we consider the  solution $\xi_\eps$ of the Neumann problem
\begin{equation}
\label{q6}
 \div \mu^\eps \nabla \xi_\eps = 
\div (\mu^\eps)^{1/2}   \bpsi_\eps, \quad
(\mu^\eps \nabla \xi_\eps)_n\vert_{\partial \O}=
((\mu^\eps)^{1/2} \bpsi_\eps)_n\vert_{\partial \O}.
 \end{equation}
The solution $\xi_\eps \in H^1(\O)$ satisfies the identity
\begin{equation}
\label{q7}
 \int_\O \langle \mu^\eps \nabla \xi_\eps, \nabla \omega \rangle \, d\x  = 
 \int_\O \langle (\mu^\eps)^{1/2}  \bpsi_\eps, \nabla \omega \rangle \, d\x,
 \quad \forall \omega \in H^1(\O). 
 \end{equation}

\begin{lemma}\label{lem5.7}
The solution $\xi_\eps$ of problem \eqref{q6} satisfies the estimate
\begin{equation}
\label{q8}
 \| (\mu^\eps)^{1/2} \nabla  \xi_\eps \|_{L_2(\O)} \le {\mathcal C}_{19} \eps^{1/2} \|\r\|_{L_2(\O)},
 \quad 0 < \eps \le \eps_1.
 \end{equation}
The constant ${\mathcal C}_{19}$ depends on the norms $\| \eta \|_{L_\infty}$,  $\| \eta^{-1} \|_{L_\infty}$,
$\| \mu \|_{L_\infty}$, $\| \mu^{-1} \|_{L_\infty}$, the parameters of the lattice $\Gamma$, and the domain $\O$.
\end{lemma}

\begin{proof}
We estimate the right-hand side of  \eqref{q7}. By \eqref{lem1}, 
\begin{equation}
\label{q9}
 \left| \int_\O \langle (\mu^\eps)^{1/2}  \bpsi_\eps, \nabla \omega \rangle \, d\x - {\mathcal I}_\eps[\omega]\right|
 \le {\mathcal C}_3 \eps \|\r\|_{L_2(\O)} \| (\mu^\eps)^{1/2} \nabla \omega \|_{L_2(\O)}, 
 \end{equation}
where
 \begin{equation*}
{\mathcal I}_\eps[\omega] =
 \int_\O \langle (\mu^\eps)^{1/2} (W_\mu^\eps)^*   S_\eps (\wt{\bvarphi}_0 + \wt{\brho}_\eps), \nabla \omega \rangle \, d\x.
 \end{equation*}
 According to \eqref{Mrr15},
 $$
 \begin{aligned}
 &(\mu^\eps)^{1/2} (W_\mu^\eps)^*   S_\eps (\wt{\bvarphi}_0 + \wt{\brho}_\eps)=
 \wt{\mu}^\eps (\mu^0)^{-1/2}   S_\eps (\wt{\bvarphi}_0 + \wt{\brho}_\eps)
 \\
 &= (\wt{\mu}^\eps - \mu^0) (\mu^0)^{-1/2}   S_\eps (\wt{\bvarphi}_0 + \wt{\brho}_\eps)
 +  (\mu^0)^{1/2}  ( S_\eps -I) (\wt{\bvarphi}_0 + \wt{\brho}_\eps)
  +  (\mu^0)^{1/2}  (\wt{\bvarphi}_0 + \wt{\brho}_\eps).
\end{aligned}
$$
Since $(\mu^0)^{1/2}  ({\bvarphi}_0 + {\brho}_\eps)\in J_0(\O)$, then  
$((\mu^0)^{1/2}  ({\bvarphi}_0 + {\brho}_\eps), \nabla \omega)_{L_2(\O)}=0$.
Hence,
 \begin{equation}
\label{q11}
{\mathcal I}_\eps[\omega] =
{\mathcal I}_\eps^{(1)}[\omega] + {\mathcal I}_\eps^{(2)}[\omega],
 \end{equation}
where
\begin{align}
\label{q12}
 &{\mathcal I}_\eps^{(1)}[\omega]=
 \int_\O \langle   (\wt{\mu}^\eps - \mu^0) (\mu^0)^{-1/2}  S_\eps (\wt{\bvarphi}_0 + \wt{\brho}_\eps), \nabla \omega \rangle \, d\x,
 \\
 \label{q13}
 &{\mathcal I}_\eps^{(2)}[\omega]=
 \int_\O \langle  (\mu^0)^{1/2}  (S_\eps -I) (\wt{\bvarphi}_0 + \wt{\brho}_\eps), \nabla \omega \rangle \, d\x.
 \end{align}

The term \eqref{q13} is estimated by Proposition \ref{prop_Seps - I} and  
\eqref{Mrr11},  \eqref{Mrr14}:
 \begin{equation}
\label{q14}
| {\mathcal I}_\eps^{(2)}[\omega] | \le {\mathcal C}'_{19} \eps \| \r \|_{L_2(\O)}   \| (\mu^\eps)^{1/2} \nabla \omega \|_{L_2(\O)}, 
 \end{equation}
where ${\mathcal C}_{19}' = \|\mu\|_{L_\infty}^{1/2} \|\mu^{-1}\|_{L_\infty}^{1/2} r_1 ({\mathcal C}_1 + {\mathcal C}_2)$.

The term \eqref{q12} is considered similarly to  \eqref{t19a}--\eqref{tt3}. We have
\begin{equation*}
{\mathcal I}_\eps^{(1)}[\omega]=
\sum_{l,i=1}^3 \left( (\wt{\mu}_{li}^\eps - \mu^0_{li}) q_{\eps,i}, \partial_l \omega \right)_{L_2(\O)}.
\end{equation*}
Together with \eqref{t19b}, this implies
\begin{align}
\label{q16}
&{\mathcal I}_\eps^{(1)}[\omega]=
\wh{\mathcal I}_\eps^{(1)}[\omega] - \wt{\mathcal I}_\eps^{(1)}[\omega],
\\
\nonumber
&\wh{\mathcal I}_\eps^{(1)}[\omega] = 
\eps \sum_{l,i, j=1}^3 \left( \partial_j ( (M_{lj}^{(i)})^\eps q_{\eps,i}), \partial_l \omega \right)_{L_2(\O)},
\\
\label{q18}
&\wt{\mathcal I}_\eps^{(1)}[\omega] = 
\eps \sum_{l,i, j=1}^3 \left( (M_{lj}^{(i)})^\eps \partial_j q_{\eps,i}, \partial_l \omega \right)_{L_2(\O)}.
\end{align}
Let $\theta_\eps$ be the cut-off function satisfying \eqref{srezka}. Then
\begin{equation*}
\wh{\mathcal I}_\eps^{(1)}[\omega] = 
\eps \sum_{l,i, j=1}^3 \left( \partial_j ( \theta_\eps (M_{lj}^{(i)})^\eps q_{\eps,i}), \partial_l \omega \right)_{L_2(\O)}.
\end{equation*}
We have taken into account that
$$
 \sum_{l, j=1}^3 \left( \partial_j ( (1- \theta_\eps) (M_{lj}^{(i)})^\eps q_{\eps,i}), \partial_l \omega \right)_{L_2(\O)}=0,
$$
which can be checked by integration by parts and using  \eqref{t15}.
By \eqref{tt2a},  
 \begin{equation}
\label{q19}
| \wh{\mathcal I}_\eps^{(1)}[\omega] | \le 9 \sqrt{3}\check{\mathcal C}_{14} \| \mu^{-1}\|_{L_\infty}^{1/2} 
 \eps^{1/2} \| \r \|_{L_2(\O)}   \| (\mu^\eps)^{1/2} \nabla \omega \|_{L_2(\O)}, \quad 0 < \eps \le \eps_1,
 \end{equation}
 cf. \eqref{tt3}.
 The term \eqref{q18} is estimated similarly to the term \eqref{t22}:
 \begin{equation}
\label{q20}
| \wt{\mathcal I}_\eps^{(1)}[\omega] | \le \wt{\mathcal C}_{14} \eps \| \r \|_{L_2(\O)}   \| (\mu^\eps)^{1/2} \nabla \omega \|_{L_2(\O)},
 \end{equation}
 cf. \eqref{t29}.
 
 As a result, relations \eqref{q11},  \eqref{q14},  \eqref{q16}, \eqref{q19}, and \eqref{q20} imply that
 $$
 | {\mathcal I}_\eps [\omega] | \le \check{\mathcal C}_{19} \eps^{1/2} 
 \| \r \|_{L_2(\O)}   \| (\mu^\eps)^{1/2} \nabla \omega \|_{L_2(\O)}, \quad 
 0< \eps \le \eps_1,
 $$
where $\check{\mathcal C}_{19} = {\mathcal C}_{19}' + 9 \sqrt{3}\check{\mathcal C}_{14} \| \mu^{-1}\|_{L_\infty}^{1/2}  + \wt{\mathcal C}_{14}$.
Together with  \eqref{q9}, this yields
 \begin{equation}
\label{q21}
\left| \int_\O \langle (\mu^\eps)^{1/2} \bpsi_\eps,\nabla \omega \rangle \, d\x \right|
\le {\mathcal C}_{19} \eps^{1/2} \| \r \|_{L_2(\O)}   \| (\mu^\eps)^{1/2} \nabla \omega \|_{L_2(\O)},
\quad 0< \eps \le \eps_1,
 \end{equation}
where ${\mathcal C}_{19} ={\mathcal C}_{3} + \check{\mathcal C}_{19}$.
Substituting $\omega = \xi_\eps$ in \eqref{q7} and using \eqref{q21}, we arrive at the required inequality
\eqref{q8}.
\end{proof}

We put  $\f_\eps := (\mu^\eps)^{1/2} \nabla \xi_\eps$. By \eqref{q6}, the function $\f_\eps$ satisfies
 \begin{equation}
\label{q22}
\begin{aligned}
&\rot (\mu^\eps)^{-1/2} \f_\eps=0, \quad \div (\mu^\eps)^{1/2} \f_\eps= \div (\mu^\eps)^{1/2} \bpsi_\eps,
\\
 &((\mu^\eps)^{1/2} \f_\eps)_n \vert_{\partial \O} = ((\mu^\eps)^{1/2} \bpsi_\eps)_n \vert_{\partial \O}.
\end{aligned}
\end{equation}
Combining this with  \eqref{pogr1}--\eqref{pogr3}, we see that
 $\s_\eps - \f_\eps \in \Dom {\mathfrak l}_\eps$ and
 \begin{equation}
\label{q23}
{\mathfrak l}_\eps[\s_\eps - \f_\eps, \bzeta] +  (\s_\eps - \f_\eps, \bzeta)_{L_2(\O)}
= \wt{\mathcal R}_\eps[\bzeta], \quad \forall \bzeta \in \Dom {\mathfrak l}_\eps,
\end{equation}
where
\begin{equation*}
\wt{\mathcal R}_\eps[\bzeta] = {\mathcal R}_\eps[\bzeta] - ( \div (\mu^\eps)^{1/2} \bpsi_\eps, 
\div (\mu^\eps)^{1/2}\bzeta)_{L_2(\O)} - (\f_\eps,\bzeta)_{L_2(\O)}.
\end{equation*}
By \eqref{lem3a}, \eqref{q5}, and \eqref{q8},
\begin{equation}
\label{q25}
|\wt{\mathcal R}_\eps[\bzeta] | \le  \wt{\mathcal C}^\circ
\eps^{1/2} \|\r\|_{L_2(\O)}
\left( \|\bzeta\|_{L_2(\O)} + \| \div (\mu^\eps)^{1/2} \bzeta \|_{L_2(\O)} +  
\| \rot (\mu^\eps)^{-1/2} \bzeta \|_{L_2(\O)} \right), \ 0< \eps \le \eps_1,
\end{equation}
where $\wt{\mathcal C}^\circ = {\mathcal C}^\circ + {\mathcal C}_5 + {\mathcal C}_{19}$.
Substituting $\bzeta = \s_\eps - \f_\eps$ in \eqref{q23} and using \eqref{q25}, we obtain
\begin{equation}
\label{q26}
\begin{aligned}
 &\|\s_\eps - \f_\eps \|_{L_2(\O)} + \| \div (\mu^\eps)^{1/2} (\s_\eps - \f_\eps) \|_{L_2(\O)} 
 \\
 &+  
\| \rot (\mu^\eps)^{-1/2} (\s_\eps - \f_\eps) \|_{L_2(\O)} \le {\mathcal C}_7' \eps^{1/2} \| \r \|_{L_2(\O)}, 
\quad 0< \eps \le \eps_1,
\end{aligned}
\end{equation}
where ${\mathcal C}_7'= 3 \max \{ 1, \|\eta\|_{L_\infty}\}\wt{\mathcal C}^\circ$.

Combining \eqref{lem3a}, \eqref{q22}, \eqref{q26}, and Lemma \ref{lem5.7}, we arrive at the required inequality
\eqref{pogr6a} with the constant ${\mathcal C}_7= {\mathcal C}_7'+ {\mathcal C}_5+ {\mathcal C}_{19}$.
This \textit{completes the proof of Theorem} \ref{lem4.6}.

\section{The study of the problem in the case where $\r=0$}

\subsection{Symmetrization\label{sim6}}
Putting $\bphi_\eps = (\eta^\eps)^{-1/2}\w_\eps^{(\q)}$, we reduce problem \eqref{Mqq} to 
\begin{equation}\label{Mqq1}
\left\{\begin{matrix}
& (\eta^\eps)^{-1/2} \rot (\mu^\eps)^{-1} \rot (\eta^\eps)^{-1/2} \bphi_\eps  + \bphi_\eps = 
i (\eta^\eps)^{-1/2} \q, \quad \div (\eta^\eps)^{1/2} \bphi_\eps =0,
\\ 
&((\eta^\eps)^{-1/2} \bphi_\eps)_\tau \vert_{\partial \O} =0, 
\  (\rot (\eta^\eps)^{-1/2} \bphi_\eps)_n\vert_{\partial \O} =0.
\end{matrix}
\right.
\end{equation}
Here $\q \in J(\O)$. Automatically,  $\bphi_\eps$ is also the solution of the elliptic equation
\begin{equation}\label{Mqq2}
(\wh{\mathcal L}_\eps +I) \bphi_\eps = i (\eta^\eps)^{-1/2} \q,
\end{equation}
where the operator $\wh{\mathcal L}_\eps$ is formally given by the differential expression
\begin{equation*}
\wh{\mathcal L}_\eps =  (\eta^\eps)^{-1/2} \rot (\mu^\eps)^{-1} \rot (\eta^\eps)^{-1/2} - 
(\eta^\eps)^{1/2} \nabla \div (\eta^\eps)^{1/2}
\end{equation*}
with the boundary conditions from \eqref{Mqq1}. Strictly speaking, $\wh{\mathcal L}_\eps$ is the  selfadjoint operator in $L_2(\O;\C^3)$ corresponding to 
the closed nonnegative quadratic form
\begin{equation}\label{forma_q}
\begin{aligned}
\wh{\mathfrak l}_\eps[\bphi, \bphi] &= \int_\O \left( \langle (\mu^\eps)^{-1} \rot (\eta^\eps)^{-1/2} \bphi, 
\rot (\eta^{\eps})^{-1/2} \bphi \rangle + | \div (\eta^{\eps})^{1/2} \bphi |^2 \right)\, d\x,
\\
\Dom \wh{\mathfrak l}_\eps = &\{ \bphi \in L_2(\O;\C^3): \quad  \div (\eta^{\eps})^{1/2} \bphi \in L_2(\O),
\\ 
& \rot (\eta^{\eps})^{-1/2} \bphi \in L_2(\O;\C^3), \quad 
 ((\eta^\eps)^{-1/2} \bphi)_\tau \vert_{\partial \O} =0\}.
\end{aligned}
\end{equation}
By the results of  \cite{BS1, BS2}, the form \eqref{forma_q} is closed.

\begin{remark}
{\rm 1)} In general, $\operatorname{Dom} \wh{\mathfrak l}_\eps \not\subset H^1(\O;\C^3)$.

\noindent {\rm 2)} The second boundary condition in  \eqref{Mqq1} is natural, it is not reflected in the domain of the quadratic form  
$\wh{\mathfrak l}_\eps$. 

\noindent {\rm 3)} The form $\wh{\mathfrak l}_\eps$ and the operator 
$\wh{\mathcal L}_\eps$ are reduced by the orthogonal decomposition 
$$
L_2(\O;\C^3) = {\mathcal J}(\O;\eta^\eps) \oplus {\mathcal G}_0(\O;\eta^\eps),
$$
where
$$
\begin{aligned}
{\mathcal J}(\O;\eta^\eps) &= \{ \f: \ (\eta^\eps)^{1/2} \f \in J(\O) \}, 
\\
 {\mathcal G}_0(\O;\eta^\eps) &= \{ (\eta^\eps)^{1/2} \nabla \omega: \ \omega \in H^1_0(\O) \}.
\end{aligned}
$$
\end{remark}

\subsection{The effective problem\label{sec6.2}}
Let $\eta^0$ and $\mu^0$ be the effective matrices defined in Subsection \ref{eff_mat}.
We put $\bphi_0 = (\eta^0)^{-1/2} \w_0^{(\q)}$. Then $\bphi_0$ is the solution of the problem
\begin{equation}\label{Mqq4}
\left\{\begin{matrix}
& (\eta^0)^{-1/2} \rot (\mu^0)^{-1} \rot (\eta^0)^{-1/2} \bphi_0  + \bphi_0 = 
i (\eta^0)^{-1/2} \q,  \ \div (\eta^0)^{1/2} \bphi_0 =0,
\\ 
&((\eta^0)^{-1/2} \bphi_0)_\tau \vert_{\partial \O} =0, 
\  (\rot (\eta^0)^{-1/2} \bphi_0)_n \vert_{\partial \O} =0.
\end{matrix}
\right.
\end{equation}
Automatically, $\bphi_0$ is also the solution of the elliptic equation
\begin{equation}\label{Mqq5}
( \wh{\mathcal L}^0 +I) \bphi_0 = i (\eta^0)^{-1/2} \q,
\end{equation}
where $\wh{\mathcal L}^0$ is the selfadjoint operator in $L_2(\O;\C^3)$ corresponding to the closed nonnegative quadratic form
\begin{equation}
\label{forma_eff_q}
\begin{aligned}
\wh{\mathfrak l}_0 [\bphi, \bphi] =& \int_\O \left( \langle (\mu^0)^{-1} \rot (\eta^0)^{-1/2} \bphi, 
\rot (\eta^{0})^{-1/2} \bphi \rangle + | \div (\eta^{0})^{1/2} \bphi |^2 \right)\, d\x,
\\
 \Dom \wh{\mathfrak l}_0 =& \{ \bphi \in L_2(\O;\C^3):\  
\div (\eta^{0})^{1/2} \bphi \in L_2(\O),
\\ 
&\rot (\eta^{0})^{-1/2} \bphi \in L_2(\O;\C^3), \ 
 ((\eta^0)^{-1/2} \bphi)_\tau \vert_{\partial \O} =0\}.
\end{aligned}
\end{equation}
Due to smoothness of the boundary  ($\partial \O \in C^{1,1}$), the set $\Dom \wh{\mathfrak l}_0$ coincides with 
$$
\Dom \wh{\mathfrak l}_0 = \{ \bphi \in H^1(\O;\C^3):   \  ((\eta^0)^{-1/2} \bphi)_\tau \vert_{\partial \O} =0\}.
$$
The form \eqref{forma_eff_q} is coercive: the following two-sided estimates hold:
\begin{equation}\label{coercive_q}
\wh{\mathfrak c}_1  \| \bphi \|^2_{H^1(\O)} \leqslant
\wh{\mathfrak l}_0 [\bphi, \bphi] + \| \bphi \|^2_{L_2(\O)} \leqslant 
\wh{\mathfrak c}_2  \| \bphi \|^2_{H^1(\O)},\quad \bphi \in \Dom \wh{\mathfrak l}_0. 
\end{equation}
The constant $\wh{\mathfrak c}_1$ depends on $\| \eta \|_{L_\infty}$, $\| \eta^{-1} \|_{L_\infty}$,
$\| \mu \|_{L_\infty}$, and the domain $\O$, and the constant $\wh{\mathfrak c}_2$ depends on  
$\| \eta \|_{L_\infty}$, $\| \eta^{-1} \|_{L_\infty}$, $\| \mu^{-1} \|_{L_\infty}$, and the domain $\O$.
 These properties were proved in  \cite[Theorem 2.3]{BS1} under the assumption that $\partial \O \in C^2$ and in \cite[Theorem 2.3]{F} 
 under the assumption that $\partial \O \in C^{3/2+\delta}$, $\delta>0$.

The operator $\wh{\mathcal L}^0$ is a strongly elliptic operator with constant coefficients.
The smoothness of the boundary ($\partial \O \in C^{1,1}$) ensures the regularity:
the resolvent $(\wh{\mathcal L}^0+I)^{-1}$ is continuous from $L_2(\O;\C^3)$ to $H^2(\O;\C^3)$, and  
\begin{equation}\label{Mqq6}
\| (\wh{\mathcal L}^0+I)^{-1} \|_{L_2(\O) \to H^2(\O)} \leqslant \widehat{c}_*,
\end{equation}
where the constant $\widehat{c}_*$ depends only on $\|\eta\|_{L_\infty}$, $\|\eta^{-1}\|_{L_\infty}$,
$\|\mu\|_{L_\infty}$, $\|\mu^{-1}\|_{L_\infty}$, and the domain $\O$.
Thus, the operator $\wh{\mathcal L}^0$ can be given by the differential expression
\begin{equation*}
\wh{\mathcal L}^0 =  (\eta^0)^{-1/2} \rot (\mu^0)^{-1} \rot (\eta^0)^{-1/2} - 
(\eta^0)^{1/2} \nabla \div (\eta^0)^{1/2}
\end{equation*}
 on the domain
\begin{equation*}
\Dom \wh{\mathcal L}^0 = \{ \bphi \in H^2(\O;\C^3): \ ((\eta^0)^{-1/2} \bphi)_\tau \vert_{\partial \O} =0,
\ (\rot (\eta^0)^{-1/2} \bphi)_n \vert_{\partial \O} =0 \}.
\end{equation*}
This property is checked similarly to Remark \ref{rem4.2}(1).

The form $\wh{\mathfrak l}_0$ and the operator  
$\wh{\mathcal L}^0$ are reduced by the orthogonal decomposition
$$
L_2(\O;\C^3) = {\mathcal J}(\O;\eta^0) \oplus {\mathcal G}_0(\O;\eta^0),
$$
where
$$
\begin{aligned}
{\mathcal J}(\O;\eta^0) &= \{ \f: \ (\eta^0)^{1/2} \f \in J(\O) \}, 
\\
 {\mathcal G}_0 (\O;\eta^0) &= \{ (\eta^0)^{1/2} \nabla \omega: \ \omega \in H^1_0(\O) \}.
\end{aligned}
$$

By \eqref{Mqq5} and \eqref{Mqq6}, we have $\bphi_0 \in H^2(\O;\C^3)$ and 
\begin{equation}\label{Mqq9}
   \| \bphi_0 \|_{H^2(\O)} \leqslant \wh{c}_* \| \eta^{-1}\|^{1/2}_{L_\infty} \| \q \|_{L_2(\O)}.
\end{equation}

Let $P_\O: H^2(\O;\C^3) \to H^2(\R^3;\C^3)$ be the linear continuous extension operator; see Subsection 
\ref{sec4.2}.  We put $\wt{\bphi}_0 := P_\O \bphi_0 \in H^2(\R^3;\C^3)$. According to \eqref{Mrr10} and \eqref{Mqq9},
\begin{equation}\label{Mqq11}
   \| \wt{\bphi}_0 \|_{H^2(\R^3)} \leqslant {\mathfrak C}_1 \| \q \|_{L_2(\O)},
\end{equation}
where ${\mathfrak C}_1 = C_\O \wh{c}_* \| \eta^{-1}\|^{1/2}_{L_\infty}$.

\subsection{The correction problem\label{sec6.3}} We put ${\boldsymbol{\upsilon}}_\eps = (\eta^0)^{-1/2} \wh{\w}_\eps^{(\q)}$. 
Then $\bups_\eps$ is the solution of the problem
\begin{equation}\label{Mqq11a}
\left\{\begin{matrix}
& (\eta^0)^{-1/2} \rot (\mu^0)^{-1} \rot (\eta^0)^{-1/2} \bups_\eps   + \bups_\eps  = 
i (\eta^0)^{-1/2} \q_\eps, \ \div (\eta^0)^{1/2} \bups_\eps  =0,
\\ 
&((\eta^0)^{-1/2} \bups_\eps )_\tau \vert_{\partial \O} =0, 
\  (\rot (\eta^0)^{-1/2} \bups_\eps )_n\vert_{\partial \O} =0.
\end{matrix}
\right.
\end{equation}
Automatically,  $\bups_\eps$ is also the solution of the following elliptic equation
\begin{equation*}
(\wh{\mathcal L}^0 +I) \bups_\eps = i (\eta^0)^{-1/2} \q_\eps. 
\end{equation*}
By   \eqref{right_est} and \eqref{Mqq6}, we have $\bups_\eps \in H^2(\O;\C^3)$ and
\begin{equation}\label{Mqq13}
   \| \bups_\eps \|_{H^2(\O)} \leqslant \wh{c}_* \| \eta\|_{L_\infty} \| \eta^{-1}\|^{3/2}_{L_\infty} \| \q \|_{L_2(\O)}.
\end{equation}
We put $\wt{\bups}_\eps := P_\O \bups_\eps \in H^2(\R^3; \C^3)$. According to  \eqref{Mrr10} and \eqref{Mqq13},
\begin{equation}\label{Mqq14}
   \| \wt{\bups}_\eps \|_{H^2(\R^3)} \leqslant {\mathfrak C}_2 \| \q \|_{L_2(\O)},
\end{equation}
where ${\mathfrak C}_2 = C_\O \wh{c}_* \| \eta \|_{L_\infty} \| \eta^{-1}\|^{3/2}_{L_\infty}$.

\subsection{The first order approximation for $\bphi_\eps$}
Let $\bphi_\eps$ be the solution of equation \eqref{Mqq2}. 
By analogy with Subsection \ref{sec4.4}, we look for the first order approximation $\boldsymbol{\vartheta}_\eps$ of the solution $\bphi_\eps$ 
in the form similar to the case of $\R^3$.

Let us introduce the necessary objects. Let $W_\eta^*(\x)$ be the $\Gamma$-periodic  $(3 \times 3)$-matrix-valued function given by
\begin{equation}\label{Mqq15}
   W_\eta^*(\x) = \eta(\x)^{-1/2} \wt{\eta}(\x) (\eta^0)^{-1/2} = \eta(\x)^{1/2} (\1 +Y_{\eta}(\x)) (\eta^0)^{-1/2},
\end{equation}
where $\wt{\eta}(\x)$ is the matrix \eqref{eff2}. We put $\wh{\c}_j:= (\eta^0)^{-1/2} \e_j$, $j=1,2,3$.
 Let $\wh{\Phi}_j(\x)$ be the  $\Gamma$-periodic solution of the problem
 \begin{equation*}
   \div \eta(\x) (\nabla \wh{\Phi}_j(\x) + \wh{\c}_j)=0,
   \quad \int_\Omega \wh{\Phi}_j(\x)\, d\x =0.
\end{equation*}
Let $\wh{\f}_{lj}(\x)$ (where $l,j=1,2,3$) be the $\Gamma$-periodic solution of the problem
\begin{equation*}
\begin{aligned}
   & \eta(\x)^{-1/2} \rot \mu(\x)^{-1} \left( \rot \eta(\x)^{-1/2}  \wh{\f}_{lj}(\x) + 
   i \e_l \times (\nabla \wh{\Phi}_j(\x) + \wh{\c}_j) \right) 
   \\
   & - \eta(\x)^{1/2} \nabla \left( \div \eta(\x)^{1/2} \wh{\f}_{lj}(\x) +
   i \e_l \cdot (\eta(\x) (\nabla \wh{\Phi}_j(\x) + \wh{\c}_j) \right) = 0,
   \\
   & \int_\Omega \wh{\f}_{lj} (\x)\, d\x =0.
   \end{aligned}
\end{equation*}
Let $\wh{\Lambda}_l(\x)$ (where $l=1,2,3$) be the $\Gamma$-periodic $(3\times 3)$-matrix-valued function
with the columns $\wh{\f}_{lj}(\x)$, $j=1,2,3$. Similarly to \eqref{Mrr17a}, we have 
\begin{equation}\label{Mqq17a}
\| \wh{\Lambda}_l\|_{L_2(\Omega)} \leqslant C_{\wh{\Lambda}} |\Omega|^{1/2}.  
\end{equation}
The analog of Remark \ref{rem4.3} holds for the functions $\wh{\f}_{lj}$.
In particular, together with  \eqref{eff6b}, this implies 
\begin{equation}\label{6.21a}
\| \rot \eta^{-1/2} \wh{\Lambda}_l\|_{L_2(\Omega)} \leqslant C'_{\wh{\Lambda}} |\Omega|^{1/2}.
\end{equation}
The constants $C_{\wh{\Lambda}}$ and $C'_{\wh{\Lambda}}$ depend on  $\|\eta\|_{L_\infty}$, $\|\eta^{-1}\|_{L_\infty}$, $\|\mu\|_{L_\infty}$, $\|\mu^{-1}\|_{L_\infty}$, and the
parameters of the lattice $\Gamma$.

 Let $\wt{\bphi}_0, \wt{\bups}_\eps \in H^2(\R^3;\C^3)$ be the functions introduced in Subsections \ref{sec6.2} and \ref{sec6.3},
respectively. Let $S_\eps$ be the Steklov smoothing operator (see \eqref{S_eps}).
We look for the first order approximation $\btheta_\eps$ of the solution $\bphi_\eps$ of equation \eqref{Mqq2} in the form
\begin{equation}\label{Mqq18}
\begin{aligned}
   & \wt{\btheta}_\eps = (W_\eta^\eps)^{*} S_\eps (\wt{\bphi}_0 + \wt{\bups}_\eps)
   + \eps \sum_{l=1}^3 \wh{\Lambda}_l^\eps S_\eps D_l  (\wt{\bphi}_0 + \wt{\bups}_\eps),  
   \\
   & \btheta_\eps = \wt{\btheta}_\eps \vert_{\O}.
   \end{aligned}
\end{equation}

The following statement is completely analogous to  Lemma \ref{lem4.3}.

\begin{lemma}\label{lem6.3}
We have
$\btheta_\eps \in L_2(\O;\C^3)$, $\rot (\eta^\eps)^{-1/2} \btheta_\eps \in L_2(\O;\C^3)$,
$\div (\eta^\eps)^{1/2} \btheta_\eps \in L_2(\O)$, and 
\begin{align*}
&\| {\btheta}_\eps - (W_\eta^\eps)^{*} S_\eps (\wt{\bphi}_0 + \wt{\bups}_\eps) \|_{L_2(\O)}
\leqslant {\mathfrak C}_3 \eps \| \q \|_{L_2(\O)},
\\
&\| (\mu^\eps)^{-1}\rot (\eta^\eps)^{-1/2}{\btheta}_\eps - 
(\1 + Y_\mu^\eps) (\mu^0)^{-1} \rot (\eta^0)^{-1/2}S_\eps 
 (\wt{\bphi}_0 + \wt{\bups}_\eps) \|_{L_2(\O)}
\leqslant {\mathfrak C}_4 \eps \| \q \|_{L_2(\O)},
\\
&\| \div (\eta^\eps)^{1/2}{\btheta}_\eps  \|_{L_2(\O)}
\leqslant {\mathfrak C}_5 \eps \| \q \|_{L_2(\O)}.
\end{align*}
The constants ${\mathfrak C}_3$, ${\mathfrak C}_4$, and ${\mathfrak C}_5$ depend only on the norms $\|\eta\|_{L_\infty}$, 
$\|\eta^{-1}\|_{L_\infty}$, $\|\mu\|_{L_\infty}$, $\|\mu^{-1}\|_{L_\infty}$, the parameters of the lattice $\Gamma$, and the domain $\O$.
\end{lemma}

\subsection{Introduction of the boundary layer correction term}
Denote
\begin{equation}
\label{pogr0_q}
\begin{aligned}
 {\mathcal Q}_\eps[\bzeta] & :=  ((\1+Y_\mu^\eps) (\mu^0)^{-1} \rot   (\eta^0)^{-1/2}
 S_\eps   (\wt{\bphi}_0 + \wt{\bups}_\eps), \rot (\eta^\eps)^{-1/2} \bzeta )_{L_2(\O)} 
\\
&+ ((W_\eta^\eps)^* S_\eps   (\wt{\bphi}_0 + \wt{\bups}_\eps), \bzeta)_{L_2(\O)}
- i ((\eta^\eps)^{-1/2} \q, \bzeta )_{L_2(\O)}, \quad \bzeta \in \Dom \wh{\mathfrak l}_\eps.
\end{aligned} 
\end{equation}
Recall that $\Dom \wh{\mathfrak l}_\eps$ is defined by \eqref{forma_q}.
We introduce the boundary layer correction term  $\wh{\s}_\eps$ as a vector-valued function in the domain $\O$
such that
\begin{equation}
\label{pogr1_q}
\wh{\s}_\eps \in L_2(\O;\C^3),\ \div (\eta^\eps)^{1/2} \wh{\s}_\eps \in L_2(\O), \ 
\rot (\eta^\eps)^{-1/2} \wh{\s}_\eps \in L_2(\O;\C^3),
\end{equation}
satisfying the identity
\begin{equation}
\label{pogr2_q}
\begin{aligned}
&((\mu^\eps)^{-1} \rot (\eta^\eps)^{-1/2} \wh{\s}_\eps, \rot (\eta^\eps)^{-1/2} \bzeta )_{L_2(\O)}
\\
&+ ( \div (\eta^\eps)^{1/2} \wh{\s}_\eps, \div (\eta^\eps)^{1/2} \bzeta )_{L_2(\O)}
+ (\wh{\s}_\eps,  \bzeta )_{L_2(\O)} =  {\mathcal Q}_\eps[\bzeta],
\ \forall \bzeta \in \Dom \wh{\mathfrak l}_\eps,
\end{aligned} 
\end{equation}
and the boundary condition
\begin{equation}
\label{pogr3_q}
((\eta^\eps)^{-1/2} \wh{\s}_\eps )_\tau \vert_{\partial \O} = ((\eta^\eps)^{-1/2} \btheta_\eps )_\tau \vert_{\partial \O}.
\end{equation}

\begin{lemma}\label{lem6.5}
Let $\bphi_\eps$ be the solution of problem \eqref{Mqq1}. Let $\btheta_\eps$ be the first order approximation of the solution defined by  \eqref{Mqq18}. 
Let $\wh{\s}_\eps$ be the correction term satisfying 
\eqref{pogr1_q}--\eqref{pogr3_q}. We put $\wh{\V}_\eps := \bphi_\eps - \btheta_\eps + \wh{\s}_\eps$. Then  
$\wh{\V}_\eps \in \Dom \wh{\mathfrak l}_\eps$ and 
\begin{equation*}
\| \wh{\V}_\eps \|_{L_2(\O)} + \| \div (\eta^\eps)^{1/2} \wh{\V}_\eps \|_{L_2(\O)}
+ \| \rot (\eta^\eps)^{-1/2} \wh{\V}_\eps \|_{L_2(\O)} \leqslant {\mathfrak C}_6 \eps \| \q \|_{L_2(\O)}.
\end{equation*}
The constant ${\mathfrak C}_6$ depends only on the norms $\|\eta\|_{L_\infty}$, 
$\|\eta^{-1}\|_{L_\infty}$, $\|\mu\|_{L_\infty}$, $\|\mu^{-1}\|_{L_\infty}$, the parameters of the lattice $\Gamma$, and the domain $\O$.
\end{lemma}

Lemma \ref{lem6.5} is deduced from Lemma \ref{lem6.3} similarly to the proof of Lemma \ref{lem4.5}.

Lemma \ref{lem6.5} shows that the difference $\btheta_\eps - \wh{\s}_\eps$ gives approximation of the solution  
$\bphi_\eps$ in the ``energy'' norm with an error of sharp order $O(\eps)$. 
However, it is difficult to control the correction term $\wh{\s}_\eps$. 
We estimate $\wh{\s}_\eps$ in the ``energy'' norm.

\begin{theorem}\label{lem6.6}
Suppose that  $\wh{\s}_\eps$ satisfies relations 
\eqref{pogr1_q}--\eqref{pogr3_q}. Suppose that $\eps_1$ is subject to Condition {\rm \ref{cond1}}.
Then for $0< \eps \leqslant \eps_1$ we have
\begin{equation}
\label{pogr6a_q}
\| \wh{\s}_\eps \|_{L_2(\O)} + \| \div (\eta^\eps)^{1/2} \wh{\s}_\eps \|_{L_2(\O)}
+ \| \rot (\eta^\eps)^{-1/2} \wh{\s}_\eps \|_{L_2(\O)} \leqslant {\mathfrak C}_7 \eps^{1/2} \| \q \|_{L_2(\O)}.
\end{equation}
The constant ${\mathfrak C}_7$ depends only on the norms  $\|\eta\|_{L_\infty}$, 
$\|\eta^{-1}\|_{L_\infty}$, $\|\mu\|_{L_\infty}$, $\|\mu^{-1}\|_{L_\infty}$, the parameters of the lattice $\Gamma$, and the domain $\O$.
\end{theorem}

 Section \ref{sec7} is devoted to the proof of Theorem \ref{lem6.6}.

\subsection{Approximation of the function $\bphi_\eps$} From Lemma \ref{lem6.5} and Theorem \ref{lem6.6} we deduce approximation for the function $\bphi_\eps$.  
The proof is completely analogous to the proof of Theorem~\ref{th_varphi}.
 
\begin{theorem}
\label{th_phi}
Let $\bphi_\eps$ be the solution of problem  \eqref{Mqq1}. 
Suppose that $\eps_1$ is subject to Condition {\rm \ref{cond1}}.
Then for  $0<  \eps \leqslant \eps_1$ we have
\begin{align*}
&\| \bphi_\eps - (W^\eps_\eta)^* (\bphi_0 + \bups_\eps)\|_{L_2(\O)} \leqslant {\mathfrak C}_8 \eps^{1/2} 
\|\q\|_{L_2(\O)},
\\
&\| (\mu^\eps)^{-1} \rot (\eta^\eps)^{-1/2}\bphi_\eps - (\1 + Y_\mu^\eps) 
(\mu^0)^{-1} \rot (\eta^0)^{-1/2} (\bphi_0 + \bups_\eps)\|_{L_2(\O)} 
\leqslant {\mathfrak C}_9 \eps^{1/2} \|\q\|_{L_2(\O)}.
\end{align*}
The constants  ${\mathfrak C}_8$ and ${\mathfrak C}_9$ depend only on the norms  $\|\eta\|_{L_\infty}$, 
$\|\eta^{-1}\|_{L_\infty}$, $\|\mu\|_{L_\infty}$, $\|\mu^{-1}\|_{L_\infty}$, the parameters of the lattice $\Gamma$, and the domain $\O$.
\end{theorem}

\subsection{The final result in the case where $\r=0$}

Let us express the fields with index $\q$ in terms of the function $\bphi_\eps$ introduced in Subsection \ref{sim6}:
$$
\begin{aligned}
&\u_\eps^{(\q)}= (\eta^\eps)^{-1/2} \bphi_\eps,
\quad
\w_\eps^{(\q)}= (\eta^\eps)^{1/2} \bphi_\eps,
\\
&\v_\eps^{(\q)}= -(\mu^\eps)^{-1}\rot (\eta^\eps)^{-1/2} \bphi_\eps,
\quad
\z_\eps^{(\q)}= -\rot (\eta^\eps)^{-1/2} \bphi_\eps.
\end{aligned}
$$
Similarly, the effective fields with index $\q$ are related to the function  $\bphi_0$ defined in Subsection \ref{sec6.2}:
$$
\begin{aligned}
&\u_0^{(\q)}= (\eta^0)^{-1/2} \bphi_0,
\quad
\w_0^{(\q)}= (\eta^0)^{1/2} \bphi_0,
\\
&\v_0^{(\q)}= - (\mu^0)^{-1}\rot (\eta^0)^{-1/2} \bphi_0,
\quad
\z_0^{(\q)}= - \rot (\eta^0)^{-1/2} \bphi_0.
\end{aligned}
$$
The correction fields with index $\q$ are expressed in terms of  $\bups_\eps$ (see Subsection \ref{sec6.3}):
$$
\begin{aligned}
&\wh{\u}_\eps^{(\q)}= (\eta^0)^{-1/2} \bups_\eps,
\quad
\wh{\w}_\eps^{(\q)}= (\eta^0)^{1/2} \bups_\eps,
\\
&\wh{\v}_\eps^{(\q)}= - (\mu^0)^{-1}\rot (\eta^0)^{-1/2} \bups_\eps,
\quad
\wh{\z}_\eps^{(\q)}= - \rot (\eta^0)^{-1/2} \bups_\eps.
\end{aligned}
$$
Combining these relations with Theorem \ref{th_phi}, we arrive at the final result in the case where $\r =0$.

\begin{theorem}\label{main_th_r=0}
Let $(\w_\eps^{(\q)}, \z_\eps^{(\q)})$ be the solution of system  \eqref{M1} with $\r=0$ and let  
$\u_\eps^{(\q)} = (\eta^\eps)^{-1}\w_\eps^{(\q)}$,  $\v_\eps^{(\q)} = (\mu^\eps)^{-1}\z_\eps^{(\q)}$. 
Let $(\w_0^{(\q)}, \z_0^{(\q)})$ be the solution of the effective system \eqref{M1eff} with $\r=0$ and let  
$\u_0^{(\q)} = (\eta^0)^{-1}\w_0^{(\q)}$,  $\v_0^{(\q)} = (\mu^0)^{-1}\z_0^{(\q)}$. Let
$(\wh{\w}^{(\q)}_\eps, \wh{\z}^{(\q)}_\eps)$ be the solution of the correction system  \eqref{M_corr} with $\r =0$
and let $\wh{\u}_\eps^{(\q)} =  (\eta^0)^{-1}\wh{\w}^{(\q)}_\eps$,
$\wh{\v}_\eps^{(\q)} = (\mu^0)^{-1}\wh{\z}^{(\q)}_\eps$. Suppose that $Y_\eta$, $G_\eta$, $Y_\mu$, and $G_\mu$
are periodic matrix-valued functions introduced in Subsection {\rm \ref{eff_mat}}. 
Suppose that $\eps_1$ is subject to Condition  {\rm \ref{cond1}}.
Then for $0 < \eps \le \eps_1$ we have  
\begin{align*}
\| \u_\eps^{(\q)} - (\1 + Y_\eta^\eps)(\u_0^{(\q)} + \wh{\u}_\eps^{(\q)}) \|_{L_2(\O)} &\le
 {\mathfrak C}_8\|\eta^{-1}\|^{1/2}_{L_\infty} \eps^{1/2}  \| \q \|_{L_2(\O)},
\\
\| \w_\eps^{(\q)} - (\1 + G_\eta^\eps)(\w_0^{(\q)} + \wh{\w}_\eps^{(\q)}) \|_{L_2(\O)} &\le {\mathfrak C}_8 
\|\eta\|_{L_\infty}^{1/2} \eps^{1/2} 
\| \q \|_{L_2(\O)},
\\
\| \v_\eps^{(\q)} - (\1 + Y_\mu^\eps)(\v_0^{(\q)} + \wh{\v}_\eps^{(\q)}) \|_{L_2(\O)} &\le {\mathfrak C}_9
 \eps^{1/2} \| \q \|_{L_2(\O)},
\\
\| \z_\eps^{(\q)} - (\1 + G_\mu^\eps)(\z_0^{(\q)} + \wh{\z}_\eps^{(\q)}) \|_{L_2(\O)} &\le {\mathfrak C}_9 
\|\mu\|_{L_\infty} \eps^{1/2}  \| \q \|_{L_2(\O)}.
\end{align*}
The constants ${\mathfrak C}_8$ and ${\mathfrak C}_9$ depend on the norms  $\|\eta\|_{L_\infty}$, $\|\eta^{-1}\|_{L_\infty}$,
$\|\mu\|_{L_\infty}$, $\|\mu^{-1}\|_{L_\infty}$, the parameters of the lattice $\Gamma$, and the domain $\O$. 
\end{theorem}

\subsection{Completion of the proof of the main theorem}
Combining Theorems \ref{main_th_q=0} and \ref{main_th_r=0}, we directly obtain the statements of Theorem \ref{main_th}. 
The constants in estimates \eqref{main_1}--\eqref{main_4} are given by 
$$
\begin{aligned}
C_1 &= \max \{{\mathcal C}_9; {\mathfrak C}_8\|\eta^{-1}\|^{1/2}_{L_\infty} \};
\quad 
C_2 =  \max \{{\mathcal C}_9 \|\eta\|_{L_\infty}; 
{\mathfrak C}_8\|\eta\|^{1/2}_{L_\infty} \};
\\
 C_3 &= \max \{ {\mathcal C}_8 \|\mu^{-1}\|_{L_\infty}^{1/2}, {\mathfrak C}_9\};
\quad C_4 = \max \{ {\mathcal C}_8 \| \mu \|_{L_\infty}^{1/2}; {\mathfrak C}_9 \| \mu \|_{L_\infty} \}.
\end{aligned}
$$

\section{Estimation of  the correction term $\wh{\s}_\eps$\label{sec7}}

This section is devoted to the proof of  Theorem \ref{lem6.6}.

\subsection{Identification of  $\Dom \wh{\mathfrak l}_\eps$ and $\Dom \wh{\mathfrak l}_0$}
The following lemma plays the key role; it is similar to Lemma  \ref{lem5.1}.

\begin{lemma}\label{lem7.1}
There exists a linear operator $\wh{T}_\eps: \Dom \wh{\mathfrak l}_\eps \to \Dom \wh{\mathfrak l}_0$ such that the function
${\wh{\bzeta}}_\eps = \wh{T}_\eps \bzeta$, $\bzeta \in \Dom \wh{\mathfrak l}_\eps$, satisfies
identities
\begin{equation}\label{lemmm1_q}
\div (\eta^0)^{1/2} \wh{\bzeta}_\eps = \div (\eta^\eps)^{1/2} \bzeta,
\quad \rot (\eta^0)^{-1/2} \wh{\bzeta}_\eps = \rot (\eta^\eps)^{-1/2} \bzeta,
\end{equation}
and estimates
\begin{align}
\label{lemmm2_q}
&\| \wh{\bzeta}_\eps \|_{L_2(\O)} \leqslant {\mathfrak C}_{10} \| \bzeta \|_{L_2(\O)},
\\
\label{lemmm3_q}
&\| \wh{\bzeta}_\eps \|_{H^1(\O)}  \leqslant   {\mathfrak C}_{11} \left(\| \bzeta \|_{L_2(\O)} 
+ \|\div (\eta^\eps)^{1/2} \bzeta \|_{L_2(\O)}
+ \|\rot (\eta^\eps)^{-1/2} \bzeta \|_{L_2(\O)}\right).
\end{align}
The constant $ {\mathfrak C}_{10}$ depends only on $\|\eta \|_{L_\infty}$,
$\|\eta^{-1}\|_{L_\infty}$, and ${\mathfrak C}_{11}$ depends on the same parameters and the domain $\O$. 
\end{lemma}

\begin{proof}
We consider two auxiliary problems. 

\textbf{The first auxiliary problem}. Let $\bzeta \in \Dom \wh{\mathfrak l}_\eps$.
Denote $\wh{f}_\eps := \div (\eta^\eps)^{1/2} \bzeta \in L_2(\O)$.
Let $\wh{\phi}_{\eps,1} \in H^1_0(\O)$ be the solution of the Dirichlet problem 
\begin{equation}
\label{lemm1_q}
\div \eta^0 \nabla \wh{\phi}_{\eps,1}(\x) 
= \wh{f}_\eps(\x), \ \x \in \O; \quad  \wh{\phi}_{\eps,1}\vert_{\partial \O}=0. 
\end{equation}
The solution satisfies the identity
\begin{equation}
\label{lemm2_q}
\int_\O \langle \eta^0 \nabla \wh{\phi}_{\eps,1}, \nabla \omega  \rangle \, d\x = 
\int_\O \langle (\eta^\eps)^{1/2} \bzeta, \nabla \omega  \rangle \, d\x,
\quad \forall \omega \in H^1_0(\O).  
\end{equation}
We put  $\wh{\bzeta}_{\eps,1} = (\eta^0)^{1/2} \nabla \wh{\phi}_{\eps,1}$. Then we have
\begin{equation}
\label{lemm3_q}
\div (\eta^0)^{1/2}  \wh{\bzeta}_{\eps,1} = \div (\eta^\eps)^{1/2} \bzeta, 
 \ 
\rot (\eta^0)^{-1/2}  \wh{\bzeta}_{\eps,1} = 0,
\  
((\eta^0)^{-1/2}  \wh{\bzeta}_{\eps,1})_\tau \vert_{\partial \O}=0.
\end{equation}
The boundary condition is fulfilled, because the function  $\wh{\phi}_{\eps,1}$ is equal to zero on the boundary $\partial \O$, whence 
the tangential component of the gradient of this function is also equal to zero.
Substituting $\omega = \wh{\phi}_{\eps,1}$ in  \eqref{lemm2_q}, we arrive at the estimate
\begin{equation}
\label{lemm4_q}
\| \wh{\bzeta}_{\eps,1} \|_{L_2(\O)} \leqslant \|\eta\|_{L_\infty}^{1/2} \|\eta^{-1}\|_{L_\infty}^{1/2} \| \bzeta \|_{L_2(\O)}.
\end{equation}
The smoothness of the boundary ($\partial \O \in C^{1,1}$) ensures the regularity of the solution 
of problem  \eqref{lemm1_q}: we have  $\wh{\phi}_{\eps,1} \in H^2(\O)$ and 
\begin{equation}
\label{lemm5_q}
\| \wh{\bzeta}_{\eps,1} \|_{H^1(\O)} \leqslant  \wh{c}_1 \| \wh{f}_\eps \|_{L_2(\O)} = \wh{c}_1 
\| \div (\eta^\eps)^{1/2} \bzeta \|_{L_2(\O)}.
\end{equation}
The constant $\wh{c}_1$ depends only on the norms $\| \eta \|_{L_\infty}$, $\| \eta^{-1} \|_{L_\infty}$, and the domain $\O$.

\textbf{The second auxiliary problem}. Let $\bzeta \in \Dom \wh{\mathfrak l}_\eps$.
Denote  $\wh{g}_\eps := \div \eta^0 (\eta^\eps)^{-1/2} \bzeta \in H^{-1}(\O)$.
Let $\wh{\phi}_{\eps,2} \in H^1_0(\O)$ be the  solution of the Dirichlet problem
\begin{equation*}
\div \eta^0 \nabla \wh{\phi}_{\eps,2}(\x) = \wh{g}_\eps(\x), \ \x \in \O; 
\quad \wh{\phi}_{\eps,2} \vert_{\partial \O}= 0. 
\end{equation*}
The solution  $\wh{\phi}_{\eps,2}$ satisfies the integral identity
\begin{equation}
\label{lemm7_q}
\int_\O \langle \eta^0 \nabla \wh{\phi}_{\eps,2}, \nabla \omega  \rangle \, d\x = 
\int_\O \langle \eta^0 (\eta^\eps)^{-1/2} \bzeta, \nabla \omega  \rangle \, d\x,
\quad \forall \omega \in H^1_0(\O).  
\end{equation}
We put $\wh{\bzeta}_{\eps,2} = (\eta^0)^{1/2} ((\eta^\eps)^{-1/2}\bzeta - \nabla \wh{\phi}_{\eps,2})$. Then 
\begin{equation}
\label{lemm8_q}
\div (\eta^0)^{1/2}  \wh{\bzeta}_{\eps,2} = 0, 
\ 
\rot (\eta^0)^{-1/2}  \wh{\bzeta}_{\eps,2} = \rot (\eta^\eps)^{-1/2}\bzeta,
\  
((\eta^0)^{-1/2}  \wh{\bzeta}_{\eps, 2})_\tau \vert_{\partial \O}=0.
\end{equation}
 Substituting  $\omega = \wh{\phi}_{\eps,2}$ in   \eqref{lemm7_q}, we arrive at the estimate
 \begin{equation}
\label{lemm9_q}
\| (\eta^0)^{1/2} \nabla \wh{\phi}_{\eps,2}  \|_{L_2(\O)} \leqslant \|\eta\|_{L_\infty}^{1/2} \|\eta^{-1}\|_{L_\infty}^{1/2} 
\| \bzeta \|_{L_2(\O)}.
\end{equation}
Consequently,
\begin{equation}
\label{lemm10_q}
\| \wh{\bzeta}_{\eps,2}   \|_{L_2(\O)} \leqslant 2 \|\eta\|_{L_\infty}^{1/2} \|\eta^{-1}\|_{L_\infty}^{1/2} \| \bzeta \|_{L_2(\O)}.
\end{equation}
Relations \eqref{lemm8_q} and \eqref{lemm10_q} show that $\wh{\bzeta}_{\eps,2} \in 
\Dom \wh{\mathfrak l}_0 \subset H^1(\O;\C^3)$ and (see \eqref{coercive_q})
\begin{equation}
\label{lemm11_q}
\| \wh{\bzeta}_{\eps,2}   \|_{H^1(\O)} \leqslant \wh{c}_2 ( \|\rot (\eta^\eps)^{-1/2}\bzeta \|_{L_2(\O)} + \| \bzeta \|_{L_2(\O)}).
\end{equation}
The constant $\wh{c}_2$ depends only on the norms $\| \eta \|_{L_\infty}$, $\| \eta^{-1} \|_{L_\infty}$, and the domain $\O$.

We put  $\wh{\bzeta}_\eps = \wh{\bzeta}_{\eps,1} + \wh{\bzeta}_{\eps,2}$. Then \eqref{lemm3_q} and \eqref{lemm8_q}
imply  \eqref{lemmm1_q}. Combining \eqref{lemm4_q} and \eqref{lemm10_q}, we obtain estimate  
 \eqref{lemmm2_q} with the constant  ${\mathfrak C}_{10}= 3 \|\eta\|^{1/2}_{L_\infty} \| \eta^{-1}\|^{1/2}_{L_\infty}$. Finally, 
 \eqref{lemm5_q} and \eqref{lemm11_q} yield \eqref{lemmm3_q} with the constant ${\mathfrak C}_{11}=\max \{\wh{c}_1, \wh{c}_2\}$.
\end{proof}

\begin{remark}
Under the assumptions of Lemma {\rm \ref{lem7.1}}, we have
\begin{equation}
\label{lemm12_q}
 (\eta^0)^{-1/2} \wh{\bzeta}_\eps - (\eta^\eps)^{-1/2} \bzeta = \nabla (\wh{\phi}_{\eps,1} - \wh{\phi}_{\eps,2}).
 \end{equation}
\end{remark}

\subsection{Estimate for the functional  ${\mathcal Q}_\eps[\bzeta]$}
Denote the first summand in  \eqref{pogr0_q} by ${\mathcal T}_\eps[\bzeta]$:
\begin{equation}
\label{lemm13_qq}
{\mathcal T}_\eps[\bzeta] =
 ((\1+Y_\mu^\eps) (\mu^0)^{-1} \rot   (\eta^0)^{-1/2}
 S_\eps   (\wt{\bphi}_0 + \wt{\bups}_\eps), \rot (\eta^\eps)^{-1/2} \bzeta )_{L_2(\O)}
\end{equation}
and represent it as the sum of four terms:
\begin{equation}
\label{rhs1_q}
{\mathcal T}_\eps[\bzeta] = \sum_{l=1}^4 {\mathcal T}_\eps^{(l)}[\bzeta], \quad \bzeta \in \Dom \wh{\mathfrak l}_\eps,
 \end{equation}
where
\begin{align}
\nonumber
{\mathcal T}_\eps^{(1)}[\bzeta] &=
(Y_\mu^\eps (\mu^0)^{-1} \rot   (\eta^0)^{-1/2}
 S_\eps   (\wt{\bphi}_0 + \wt{\bups}_\eps), \rot (\eta^\eps)^{-1/2} \bzeta )_{L_2(\O)},
\\
\label{rhs3_q}
{\mathcal T}_\eps^{(2)}[\bzeta] &=
( (\mu^0)^{-1} \rot   (\eta^0)^{-1/2}
  {\bphi}_0, \rot (\eta^\eps)^{-1/2} \bzeta )_{L_2(\O)},
\\
\label{rhs4_q}
{\mathcal T}_\eps^{(3)}[\bzeta] &=
((\mu^0)^{-1} \rot   (\eta^0)^{-1/2}
 {\bups}_\eps, \rot (\eta^\eps)^{-1/2} \bzeta )_{L_2(\O)},
\\
\label{rhs4a_q}
{\mathcal T}_\eps^{(4)}[\bzeta] &=
( (\mu^0)^{-1} \rot   (\eta^0)^{-1/2}
 (S_\eps-I) (\wt{\bphi}_0+  \wt{\bups}_\eps), \rot (\eta^\eps)^{-1/2} \bzeta )_{L_2(\O)}.
\end{align}

Similarly to the proof of Lemma \ref{lem5.3}, we check that  
\begin{equation}
\label{t0_q}
|{\mathcal T}_\eps^{(1)}[\bzeta]| \leqslant {\mathfrak C}_{12} \eps^{1/2} \|\q \|_{L_2(\O)}
\|  \rot (\eta^\eps)^{-1/2} \bzeta \|_{L_2(\O)}, \quad 0< \eps \leqslant \eps_1. 
\end{equation}
The term \eqref{rhs4a_q} is estimated by analogy with  \eqref{t44}:
\begin{equation}
\label{t44_q}
|{\mathcal T}_\eps^{(4)}[\bzeta]| 
\leqslant {\mathfrak C}_{13} \eps \| \q \|_{L_2( \O)}  \| \rot (\eta^\eps)^{-1/2} \bzeta \|_{L_2(\O)}.
\end{equation}
The constants  ${\mathfrak C}_{12}$ and ${\mathfrak C}_{13}$ depend on the norms  
$\| \eta \|_{L_\infty}$,   $\| \eta^{-1} \|_{L_\infty}$,
$\| \mu \|_{L_\infty}$, $\| \mu^{-1} \|_{L_\infty}$, the domain $\O$, and the parameters of the lattice $\Gamma$.

We transform the term \eqref{rhs3_q}, using Lemma~\ref{lem7.1} and the fact that $\bphi_0$ is the  solution of problem \eqref{Mqq4}:
\begin{equation}
\begin{aligned}
\label{t5_q}
{\mathcal T}_\eps^{(2)}[\bzeta] &= 
((\mu^0)^{-1} \rot (\eta^0)^{-1/2} \bphi_0, \rot (\eta^0)^{-1/2} \wh{\bzeta}_\eps)_{L_2(\O)}
\\
&= - ( \bphi_0,  \wh{\bzeta}_\eps)_{L_2(\O)} + i ((\eta^0)^{-1/2} \q,  \wh{\bzeta}_\eps)_{L_2(\O)}.
\end{aligned}
\end{equation}
The term \eqref{rhs4_q} is transformed similarly:
\begin{equation}
\label{t7_q}
\begin{aligned}
{\mathcal T}_\eps^{(3)}[\bzeta] &= 
((\mu^0)^{-1} \rot (\eta^0)^{-1/2} \bups_\eps, \rot (\eta^0)^{-1/2} \wh{\bzeta}_\eps)_{L_2(\O)}
\\
&=- ( \bups_\eps, \wh{\bzeta}_\eps)_{L_2(\O)} + i ((\eta^0)^{-1/2} \q_\eps, \wh{\bzeta}_\eps)_{L_2(\O)}.
\end{aligned}
\end{equation}
We have taken into account that  $\bups_\eps$ is the solution of problem  \eqref{Mqq11a}.

Combining \eqref{pogr0_q}, \eqref{lemm13_qq}, \eqref{rhs1_q}, \eqref{t5_q}, and \eqref{t7_q}, we arrive at the following representation:
\begin{equation}
\label{t9_q}
\begin{aligned}
{\mathcal Q}_\eps [\bzeta] =& {\mathcal T}_\eps^{(1)}[\bzeta] + {\mathcal T}_\eps^{(4)}[\bzeta]
+ ((W_\eta^\eps)^* S_\eps (\wt{\bphi}_0 + \wt{\bups}_\eps ), \bzeta)_{L_2(\O)}
\\
& - ({\bphi}_0 , \wh{\bzeta}_\eps)_{L_2(\O)} - ({\bups}_\eps , \wh{\bzeta}_\eps)_{L_2(\O)} + i ((\eta^0)^{-1/2} {\q}_\eps , \wh{\bzeta}_\eps)_{L_2(\O)}
\\
&+ i ({\q} , (\eta^0)^{-1/2} \wh{\bzeta}_\eps - (\eta^\eps)^{-1/2} \bzeta)_{L_2(\O)}.
\end{aligned}
\end{equation}

By \eqref{lemm12_q}, the last term in \eqref{t9_q} is equal to zero: 
\begin{equation}
\label{t10_q}
(\q, (\eta^0)^{-1/2} \wh{\bzeta}_\eps - (\eta^\eps)^{-1/2} \bzeta)_{L_2(\O)} =
(\q, \nabla (\wh{\phi}_{\eps,1} - \wh{\phi}_{\eps,2}))_{L_2(\O)}=0,
\end{equation}
because $\q \in J(\O)$ and $\wh{\phi}_{\eps,1} - \wh{\phi}_{\eps,2} \in H^1_0(\O)$ (see  \eqref{5.1}).

From \eqref{Mqq15} it follows that the third summand in \eqref{t9_q} is represented as
\begin{equation}
\label{t11_q}
\begin{aligned}
&((W^\eps_\eta)^* S_\eps (\wt{\bphi}_0 + \wt{\bups}_\eps),\bzeta)_{L_2(\O)} =
( \wt{\eta}^\eps (\eta^0)^{-1/2} S_\eps (\wt{\bphi}_0 + \wt{\bups}_\eps), (\eta^\eps)^{-1/2}\bzeta)_{L_2(\O)}
\\
&= 
{\mathcal T}_\eps^{(5)}[\bzeta] + {\mathcal T}_\eps^{(6)}[\bzeta]
+ 
(  (\eta^0)^{1/2}  ({\bphi}_0 + {\bups}_\eps), (\eta^\eps)^{-1/2}\bzeta)_{L_2(\O)},
\end{aligned}
\end{equation}
where 
\begin{align}
\label{t12_q}
{\mathcal T}_\eps^{(5)}[\bzeta] &=( (\wt{\eta}^\eps - \eta^0) (\eta^0)^{-1/2} S_\eps (\wt{\bphi}_0 + \wt{\bups}_\eps), (\eta^\eps)^{-1/2}\bzeta)_{L_2(\O)},
\\
\label{t13_q}
{\mathcal T}_\eps^{(6)}[\bzeta] &= ( (\eta^0)^{1/2} (S_\eps -I) (\wt{\bphi}_0 + \wt{\bups}_\eps), (\eta^\eps)^{-1/2}\bzeta)_{L_2(\O)}.
\end{align}

Now, relations  \eqref{t9_q}--\eqref{t11_q} imply that
\begin{equation}
\label{ttt9_q}
\begin{aligned}
&{\mathcal Q}_\eps[\bzeta]=
{\mathcal T}_\eps^{(1)}[\bzeta] + {\mathcal T}_\eps^{(4)}[\bzeta] 
 + {\mathcal T}_\eps^{(5)}[\bzeta] + {\mathcal T}_\eps^{(6)}[\bzeta]
 \\
 &+ i ((\eta^0)^{-1/2} \q_\eps, \wh{\bzeta}_\eps)_{L_2(\O)}
+ ((\eta^0)^{1/2}  ({\bphi}_0 + {\bups}_\eps), (\eta^\eps)^{-1/2} \bzeta - (\eta^0)^{-1/2} \wh{\bzeta}_\eps)_{L_2(\O)}.
\end{aligned}
\end{equation}
By \eqref{lemm12_q}, the last term in  \eqref{ttt9_q} is equal to zero: 
\begin{equation}
\label{ttt10_q}
\begin{aligned}
&((\eta^0)^{1/2}  ({\bphi}_0 + {\bups}_\eps), (\eta^\eps)^{-1/2} \bzeta - (\eta^0)^{-1/2} \wh{\bzeta}_\eps)_{L_2(\O)}
\\
&=((\eta^0)^{1/2}  ({\bphi}_0 + {\bups}_\eps), \nabla (\wh{\phi}_{\eps,2} - \wh{\phi}_{\eps,1}))_{L_2(\O)}=0,
\end{aligned}
\end{equation}
since  $ (\eta^0)^{1/2}({\bphi}_0 + {\bups}_\eps)  \in J(\O)$ and $\wh{\phi}_{\eps,2} - \wh{\phi}_{\eps,1} \in H^1_0(\O)$.

According to \eqref{ttt9_q} and \eqref{ttt10_q}, 
\begin{align}
\label{ttt11_q}
& {\mathcal Q}_\eps[\bzeta]=
{\mathcal T}_\eps^{(1)}[\bzeta] + {\mathcal T}_\eps^{(4)}[\bzeta] 
 + {\mathcal T}_\eps^{(5)}[\bzeta] + {\mathcal T}_\eps^{(6)}[\bzeta]
 +{\mathcal T}_\eps^{(7)}[\bzeta],
 \\
 \label{ttt12_q}
 & {\mathcal T}_\eps^{(7)}[\bzeta] :=
 i ((\eta^0)^{-1/2} \q_\eps,  \wh{\bzeta}_\eps)_{L_2(\O)}.
\end{align}

Similarly to the proof of Lemma~\ref{lem5.5} and estimate \eqref{t30},
we obtain the following estimates for the terms \eqref{t12_q} and \eqref{t13_q}:
\begin{align}
\label{t19_q}
\begin{split}
|&{\mathcal T}_\eps^{(5)}[\bzeta]| \leqslant {\mathfrak C}_{14} \eps^{1/2} \| \q \|_{L_2(\O)}
\\
&\times 
\left( \|\bzeta\|_{L_2(\O)} + \|\div (\eta^\eps)^{1/2} \bzeta\|_{L_2(\O)}
+ \|\rot (\eta^\eps)^{-1/2} \bzeta\|_{L_2(\O)} \right),\quad 0< \eps \le \eps_1,
\end{split}
\\
\label{t30_q}
&| {\mathcal T}_\eps^{(6)}[\bzeta] | \leqslant {\mathfrak C}_{15} \eps 
\| \q \|_{L_2(\O)} \| \bzeta\|_{L_2(\O)}.
\end{align}
The constants ${\mathfrak C}_{14}$ and ${\mathfrak C}_{15}$ depend on the norms  
$\| \eta \|_{L_\infty}$,   $\| \eta^{-1} \|_{L_\infty}$,
$\| \mu \|_{L_\infty}$, $\| \mu^{-1} \|_{L_\infty}$, the domain  $\O$, and the parameters of the lattice $\Gamma$.

The term \eqref{ttt12_q} is analyzed by analogy with Subsection \ref{sec5.5} (but there are some differences).
Using \eqref{right}, we represent the term  \eqref{ttt12_q} as 
\begin{equation}\label{t31_q}
 {\mathcal T}_\eps^{(7)}[\bzeta] =
 i ((\eta^0)^{-1/2} {\mathcal P}_{\eta^0} S_\eps (Y_\eta^\eps)^* \wt{\q}, \wh{\bzeta}_\eps)_{L_2(\O)}.
\end{equation}

\begin{remark}\label{rem7.6}
Let ${\mathcal P}_{\eta^0}$ be the orthogonal projection of  $L_2(\O;(\eta^0)^{-1})$ onto $J(\O)$.
It is easily seen that   ${\mathcal P}_{\eta^0} \f \in H^1(\O;\C^3)$ provided that  $\f \in H^1(\O;\C^3)$.
We have
\begin{equation}\label{t32_q}
 \| {\mathcal P}_{\eta^0} \f \|_{H^1(\O)} \leqslant \wh{\mathfrak c} \| \f \|_{H^1(\O)}.
\end{equation}
The constant $\wh{\mathfrak c}$ depends only on the norms  $\| \eta \|_{L_\infty}$, $\| \eta^{-1} \|_{L_\infty}$, and the domain~$\O$.
\end{remark}

 Since the operator ${\mathcal P}_{\eta^0}$ is selfadjoint in the weighted space  $L_2(\O;(\eta^0)^{-1})$,
 the functional \eqref{t31_q} can be represented in the form 
\begin{equation}\label{t33_q}
 {\mathcal T}_\eps^{(7)}[\bzeta] =
 i ((\eta^0)^{-1}  S_\eps (Y_\eta^\eps)^* \wt{\q},
 {\mathcal P}_{\eta^0} (\eta^0)^{1/2}  \wh{\bzeta}_\eps)_{L_2(\O)}.
\end{equation}

Let ${\theta}_\eps$ is the cut-off function satisfying \eqref{srezka}.
We write the term \eqref{t33_q} as the sum of two summands:
\begin{align}\label{t34_q}
 &{\mathcal T}_\eps^{(7)}[\bzeta] = \wh{\mathcal T}_\eps^{(7)}[\bzeta] + \wt{\mathcal T}_\eps^{(7)}[\bzeta],
 \\
\label{t35_q} 
&\wh{\mathcal T}_\eps^{(7)}[\bzeta]: =
 i ((\eta^0)^{-1}  S_\eps (Y_\eta^\eps)^* \wt{\q},
 {\theta}_\eps {\mathcal P}_{\eta^0} (\eta^0)^{1/2} \wh{\bzeta}_\eps)_{L_2(\O)},
\\
\label{t36_q}
&\wt{\mathcal T}_\eps^{(7)}[\bzeta] :=
 i ((\eta^0)^{-1}  S_\eps (Y_\eta^\eps)^* \wt{\q}, (1- {\theta}_\eps)
 {\mathcal P}_{\eta^0} (\eta^0)^{1/2} \wh{\bzeta}_\eps)_{L_2(\O)}.
 \end{align}

 The term \eqref{t35_q} is estimated with the help of  Proposition \ref{prop f^eps S_eps}, \eqref{eff6b}, and Lemma~\ref{lem01}:
 $$
 \begin{aligned}
 &| \wh{\mathcal T}_\eps^{(7)}[\bzeta] | \leqslant \|\eta^{-1}\|_{L_\infty} \| S_\eps (Y_\eta^\eps)^* \wt{\q} \|_{L_2(\R^3)}
 \| {\mathcal P}_{\eta^0} (\eta^0)^{1/2} \wh{\bzeta}_\eps \|_{L_2(B_{2\eps})}
  \\
  &\leqslant \eps^{1/2} 
  \|\eta \|^{1/2}_{L_\infty} \|\eta^{-1}\|^{3/2}_{L_\infty} \beta_0^{1/2} \|\q\|_{L_2(\O)}
  \|{\mathcal P}_{\eta^0} (\eta^0)^{1/2} \wh{\bzeta}_\eps \|_{H^1(\O)}, \quad 0<  \eps \le \eps_0.
 \end{aligned}
 $$
Taking  \eqref{lemmm3_q} and \eqref{t32_q} into account, we arrive at
\begin{equation}
\label{q00_q}
 \begin{aligned}
 &| \wh{\mathcal T}_\eps^{(7)}[\bzeta] | \leqslant 
 {\mathfrak C}_{16} \eps^{1/2} \| \q \|_{L_2(\O)}
 \\
 &\times 
 \left( \| \bzeta\|_{L_2(\O)} + \| \div (\eta^\eps)^{1/2}\bzeta\|_{L_2(\O)}
 + \| \rot (\eta^\eps)^{-1/2} \bzeta\|_{L_2(\O)} \right),\quad 0< \eps \le \eps_0,
 \end{aligned}
 \end{equation}
where ${\mathfrak C}_{16} =  \wh{\mathfrak c} \|\eta \|_{L_\infty} \|\eta^{-1}\|^{3/2}_{L_\infty}  \beta_0^{1/2} {\mathfrak C}_{11}$.

Now, we consider the term  \eqref{t36_q}.
The function $(1- {\theta}_\eps){\mathcal P}_{\eta^0} (\eta^0)^{1/2} \wh{\bzeta}_\eps$ belongs to $H^1(\O;\C^3)$ and is equal to zero in
the $\eps$-neighborhood of the boundary $\partial \O$. We extend this function by zero to $\R^3 \setminus \O$; the extended function 
is denoted by $\wh{\p}_\eps$.  Note that  $\wh{\p}_\eps \in H^1(\R^3;\C^3)$ and
 $\operatorname{supp} \wh{\p}_\eps \subset \O \setminus B_\eps$.
The term  \eqref{t36_q} can be represented as 
\begin{equation*}
\wt{\mathcal T}_\eps^{(7)}[\bzeta] :=
 i ( {\q}, Y_\eta^\eps S_\eps (\eta^0)^{-1} \wh{\p}_\eps )_{L_2(\O)}.
\end{equation*}
The columns of the matrix $Y_\eta^\eps$ are given by $\eps \nabla \Phi_j^\eps$, $j=1,2,3$.
Consequently,
\begin{equation*}
\wt{\mathcal T}_\eps^{(7)}[\bzeta] =
  i \eps \sum_{j=1}^3 ( {\q}, \nabla \left( \Phi_j^\eps [S_\eps (\eta^0)^{-1} \wh{\p}_\eps]_j\right) )_{L_2(\O)}
- i \eps \sum_{j=1}^3 ( {\q},  \Phi_j^\eps \nabla [S_\eps (\eta^0)^{-1} \wh{\p}_\eps]_j )_{L_2(\O)}.
\end{equation*}
The first summand on the right is equal to zero, because  $\q \in J(\O)$ and $\Phi_j^\eps [S_\eps (\eta^0)^{-1} \wh{\p}_\eps]_j \in H^1_0(\O)$.
(Here we use the property of the smoothing operator $S_\eps$: the function $S_\eps f$ is equal to zero outside the $\eps$-neighborhood of $\operatorname{supp} f$).  Next, by analogy with the proof of  \eqref{q3}, we obtain
\begin{equation}
\label{q3_q}
\begin{aligned}
&| \wt{\mathcal T}_\eps^{(7)}[\bzeta] | \leqslant {\mathfrak C}_{17} \eps^{1/2}
\| \q \|_{L_2(\O)}
 \\
 &\times 
 \left( \| \bzeta\|_{L_2(\O)} + \| \div (\eta^\eps)^{1/2}\bzeta \|_{L_2(\O)} 
 + \| \rot (\eta^\eps)^{-1/2}\bzeta \|_{L_2(\O)}\right),\quad 0< \eps \le \eps_0.
\end{aligned}
\end{equation}
The constant ${\mathfrak C}_{17}$ depends on $\|\eta\|_{L_\infty}$, $\|\eta^{-1} \|_{L_\infty}$, 
the parameters of the lattice $\Gamma$, and the domain~$\O$.

Now, relations  \eqref{t34_q}, \eqref{q00_q}, and \eqref{q3_q} imply that
\begin{equation}
\label{q4_q}
\begin{aligned}
&| {\mathcal T}_\eps^{(7)}[\bzeta] | \leqslant ({\mathfrak C}_{16} + {\mathfrak C}_{17}) \eps^{1/2}
\| \q \|_{L_2(\O)}
 \\
 &\times \left( \| \bzeta\|_{L_2(\O)} + \| \div (\eta^\eps)^{1/2}\bzeta \|_{L_2(\O)} 
 + \| \rot (\eta^\eps)^{-1/2}\bzeta \|_{L_2(\O)}\right),\quad 0< \eps \le \eps_0.
 \end{aligned}
 \end{equation}

\subsection{Taking the boundary condition into account. Completion of the proof of Theorem \ref{lem6.6}}
As a result, relations \eqref{t0_q}, \eqref{t44_q}, \eqref{ttt11_q}, \eqref{t19_q}, \eqref{t30_q}, and \eqref{q4_q} 
imply the following estimate for the functional~\eqref{pogr0_q}:
\begin{equation}
\label{q5_q}
| {\mathcal Q}_\eps [\bzeta] | \leqslant {\mathfrak C}^\circ \eps^{1/2}
\| \q \|_{L_2(\O)}
 \left( \| \bzeta\|_{L_2(\O)} + \| \div (\eta^\eps)^{1/2}\bzeta \|_{L_2(\O)} 
 + \| \rot (\eta^\eps)^{-1/2}\bzeta \|_{L_2(\O)}\right),\ 0<  \eps \le \eps_1,
 \end{equation}
where ${\mathfrak C}^\circ={\mathfrak C}_{12} + {\mathfrak C}_{13} + {\mathfrak C}_{14} + {\mathfrak C}_{15}
+ {\mathfrak C}_{16} + {\mathfrak C}_{17}$.

It remains to take the boundary condition  \eqref{pogr3_q} into account.
By \eqref{Mqq18}, we have
\begin{equation*}
\begin{aligned}
   (\eta^\eps)^{-1/2}{\btheta}_\eps &= (\1+ Y_\eta^\eps) S_\eps (\eta^0)^{-1/2} (\wt{\bphi}_0 + \wt{\bups}_\eps)
 + \eps \sum_{l=1}^3 (\eta^\eps)^{-1/2}\wh{\Lambda}_l^\eps S_\eps D_l  (\wt{\bphi}_0 + \wt{\bups}_\eps)
   \\
   &  = (\eta^0)^{-1/2} ({\bphi}_0 + {\bups}_\eps) +
(S_\eps -I) (\eta^0)^{-1/2} (\wt{\bphi}_0 + \wt{\bups}_\eps)
\\
& + \eps \sum_{l=1}^3 (\eta^\eps)^{-1/2}\wh{\Lambda}_l^\eps S_\eps D_l  (\wt{\bphi}_0 + \wt{\bups}_\eps)
+ \eps \sum_{j=1}^3 \nabla \left( \Phi_j^\eps  S_\eps b_{\eps,j}\right) 
- \eps \sum_{j=1}^3 \Phi_j^\eps S_\eps \nabla b_{\eps,j},
   \end{aligned}
\end{equation*}
where  $b_{\eps,j}= [(\eta^0)^{-1/2} (\wt{\bphi}_0 + \wt{\bups}_\eps)]_j$.
We have taken into account that  $Y_\eta^\eps$ is the matrix with the columns
$\eps \nabla \Phi_j^\eps$, $j=1,2,3$. The first summand on the right satisfies the homogeneous boundary condition
$((\eta^0)^{-1/2} ({\bphi}_0 + {\bups}_\eps))_\tau \vert_{\partial \O}=0$. Therefore, we have
\begin{equation}\label{aaa2}
(   (\eta^\eps)^{-1/2}{\btheta}_\eps)_\tau \vert_{\partial \O} = (\a_\eps)_\tau \vert_{\partial \O}, 
\end{equation}
where
\begin{align}
\label{aaa3}
\a_\eps &= \sum_{k=1}^4 \a_\eps^{(k)},
\\
\label{aaa4}
\a_\eps^{(1)} &= (S_\eps -I) (\eta^0)^{-1/2} (\wt{\bphi}_0 + \wt{\bups}_\eps),
\\
\label{aaa5}
\a_\eps^{(2)} &= 
\eps {\theta}_\eps \sum_{l=1}^3 (\eta^\eps)^{-1/2}\wh{\Lambda}_l^\eps S_\eps D_l  (\wt{\bphi}_0 + \wt{\bups}_\eps)
\\
\label{aaa6}
\a_\eps^{(3)} &= \eps \sum_{j=1}^3 \nabla \left( {\theta}_\eps \Phi_j^\eps  S_\eps b_{\eps,j}\right),
\\
\label{aaa7}
\a_\eps^{(4)} &= - \eps {\theta}_\eps \sum_{j=1}^3 \Phi_j^\eps S_\eps \nabla b_{\eps,j}.
\end{align}
Here ${\theta}_\eps$ is the cut-off function satisfying \eqref{srezka}.

\begin{lemma}
For $0< \eps \le \eps_1$ we have
\begin{align}
\label{aaa8}
\|\a_\eps \|_{L_2(\O)} \le {\mathfrak C}_{18} \eps^{1/2} \| \q\|_{L_2(\O)},
\\
\label{aaa9}
\| \rot \a_\eps \|_{L_2(\O)} \le {\mathfrak C}_{19} \eps^{1/2} \| \q\|_{L_2(\O)}.
\end{align}
The constants ${\mathfrak C}_{18}$ and ${\mathfrak C}_{19}$ depend on $\|\eta\|_{L_\infty}$,  
$\|\eta^{-1}\|_{L_\infty}$, $\|\mu\|_{L_\infty}$, $\|\mu^{-1}\|_{L_\infty}$, the parameters of the lattice $\Gamma$, and the domain~$\O$.
\end{lemma}

\begin{proof}
The term \eqref{aaa4} is estimated with the help of Proposition \ref{prop_Seps - I} and \eqref{Mqq11}, \eqref{Mqq14}:
\begin{equation}
\label{aaa10}
\| \a_\eps^{(1)}\|_{L_2(\O)} 
\le \eps r_1 \| \eta^{-1}\|^{1/2}_{L_\infty} \| \wt{\bphi}_0 + \wt{\bups}_\eps \|_{H^1(\R^3)}
\le {\mathfrak C}_{18}^{(1)} \eps \| \q\|_{L_2(\O)},
\end{equation}
where ${\mathfrak C}_{18}^{(1)} = r_1 \| \eta^{-1}\|^{1/2}_{L_\infty} ({\mathfrak C}_1 + {\mathfrak C}_2)$.
The term \eqref{aaa5} is estimated by Proposition
\ref{prop f^eps S_eps} and relations \eqref{Mqq11}, \eqref{Mqq14}, \eqref{Mqq17a}:
\begin{equation}
\label{aaa11}
\| \a_\eps^{(2)}\|_{L_2(\O)} 
\le \eps \| \eta^{-1}\|^{1/2}_{L_\infty}  C_{\wh{\Lambda}} \sqrt{3}
\| \wt{\bphi}_0 + \wt{\bups}_\eps \|_{H^1(\R^3)}
\le {\mathfrak C}_{18}^{(2)} \eps \| \q\|_{L_2(\O)},
\end{equation}
where ${\mathfrak C}_{18}^{(2)} =  \| \eta^{-1}\|^{1/2}_{L_\infty}   C_{\wh{\Lambda}} \sqrt{3}({\mathfrak C}_1 + {\mathfrak C}_2)$.

By analogy with the proof of estimate  \eqref{t2a},
we consider the term \eqref{aaa6}:
$$
\eps \nabla \left( {\theta}_\eps \Phi_j^\eps  S_\eps b_{\eps,j}\right)=
(\eps \nabla  {\theta}_\eps) \Phi_j^\eps  S_\eps b_{\eps,j}
+ {\theta}_\eps (\nabla  \Phi_j)^\eps  S_\eps b_{\eps,j}
+ \eps  {\theta}_\eps \Phi_j^\eps  S_\eps \nabla b_{\eps,j}.
$$
The first term is estimated with the help of \eqref{srezka}, Lemma~\ref{lem02}, and \eqref{efff}.
The second term is estimated by using \eqref{srezka}, Lemma~\ref{lem02}, and \eqref{eff6b}.
The third term is estimated by Proposition \ref{prop f^eps S_eps} and \eqref{efff}.
We also take the inequalities  \eqref{Mqq11} and \eqref{Mqq14} into account. 
These arguments imply the following estimate
\begin{equation}
\label{aaa12}
\| \a_\eps^{(3)}\|_{L_2(\O)} 
\le {\mathfrak C}_{18}^{(3)} \eps^{1/2} \| \q\|_{L_2(\O)}, \quad 0< \eps \le \eps_1, 
\end{equation}
where the constant ${\mathfrak C}_{18}^{(3)}$ depends on  $\|\eta\|_{L_\infty}$,  
$\|\eta^{-1}\|_{L_\infty}$, $\|\mu\|_{L_\infty}$, $\|\mu^{-1}\|_{L_\infty}$, the parameters of the lattice $\Gamma$, and the domain $\O$.

The term \eqref{aaa7} is estimated by using Proposition \ref{prop f^eps S_eps} and relations \eqref{efff}, \eqref{Mqq11}, \eqref{Mqq14}:
\begin{equation}
\label{aaa13}
\| \a_\eps^{(4)}\|_{L_2(\O)} 
\le \eps  (2r_0)^{-1} \|\eta \|_{L_\infty}^{1/2} \|\eta^{-1} \|_{L_\infty} \sqrt{3} \| \wt{\bphi}_0 + \wt{\bups}_\eps\|_{H^1(\R^3)}
\le
{\mathfrak C}_{18}^{(4)} \eps \| \q\|_{L_2(\O)}, 
\end{equation}
where ${\mathfrak C}_{18}^{(4)} = (2r_0)^{-1} \|\eta \|_{L_\infty}^{1/2} \|\eta^{-1} \|_{L_\infty} \sqrt{3} 
({\mathfrak C}_1 + {\mathfrak C}_2)$.

Now, relations  \eqref{aaa3}, \eqref{aaa10}--\eqref{aaa13} imply the required estimate \eqref{aaa8} with the constant
${\mathfrak C}_{18} = \sum_{k=1}^4 {\mathfrak C}_{18}^{(k)}$.

Consider the curl of the function \eqref{aaa3}:
\begin{equation}
\label{aaa14}
\rot \a_\eps = 
\rot \a_\eps^{(1)} + \rot \a_\eps^{(2)} + \rot \a_\eps^{(4)}.
\end{equation}
We have taken into account that the curl of the function \eqref{aaa6} is equal to zero.
The first summand is estimated with the help of Proposition  \ref{prop_Seps - I} and relations \eqref{Mqq11}, \eqref{Mqq14}:
\begin{equation}
\label{aaa15}
\| \rot \a_\eps^{(1)}\|_{L_2(\O)} 
\le \eps r_1 \| \eta^{-1}\|^{1/2}_{L_\infty} \| \wt{\bphi}_0 + \wt{\bups}_\eps \|_{H^2(\R^3)}
\le {\mathfrak C}_{19}^{(1)} \eps \| \q\|_{L_2(\O)},
\end{equation}
where ${\mathfrak C}_{19}^{(1)} = r_1 \| \eta^{-1}\|^{1/2}_{L_\infty} ({\mathfrak C}_1 + {\mathfrak C}_2)$.

Next, we have
$$
\begin{aligned}
\rot \a_\eps^{(2)} &= \sum_{l=1}^3
(\eps \nabla {\theta}_\eps) \times \left( (\eta^\eps)^{-1/2}\wh{\Lambda}_l^\eps S_\eps D_l  (\wt{\bphi}_0 + \wt{\bups}_\eps)\right) 
\\
&+
 {\theta}_\eps 
 \sum_{l=1}^3 [\rot \eta^{-1/2}\wh{\Lambda}_l]^\eps S_\eps D_l  (\wt{\bphi}_0 + \wt{\bups}_\eps)
 +  \eps \, {\theta}_\eps 
 \sum_{l,j =1}^3 b_j  (\eta^\eps)^{-1/2} \wh{\Lambda}_l^\eps S_\eps D_l D_j (\wt{\bphi}_0 + \wt{\bups}_\eps). 
\end{aligned}
$$
The first term is estimated with the help of  \eqref{srezka}, Lemma~\ref{lem02}, and \eqref{Mqq17a}.
The second one is estimated by using  \eqref{srezka}, Lemma \ref{lem02}, and \eqref{6.21a}.
The third summand is estimated by Proposition~\ref{prop f^eps S_eps} and 
\eqref{Mqq17a}. We also use the inequalities \eqref{Mqq11} and \eqref{Mqq14}.
These arguments imply the estimate
\begin{equation}
\label{aaa16}
\| \rot \a_\eps^{(2)}\|_{L_2(\O)} 
\le {\mathfrak C}_{19}^{(2)} \eps^{1/2} \| \q\|_{L_2(\O)}, \quad 0 <  \eps \le \eps_1,
\end{equation}
where the constant ${\mathfrak C}_{19}^{(2)}$ depends on  $\|\eta\|_{L_\infty}$,
$\|\eta^{-1} \|_{L_\infty}$, $\|\mu\|_{L_\infty}$, $\|\mu^{-1}\|_{L_\infty}$, the parameters of the lattice $\Gamma$, and the domain~$\O$.

Now, we consider the curl of the function \eqref{aaa7}:
$$
\rot 
\a_\eps^{(4)} = - \sum_{j=1}^3 (\eps \nabla {\theta}_\eps) \times  (\Phi_j^\eps S_\eps \nabla b_{\eps,j})
- {\theta}_\eps \sum_{j=1}^3 ( \nabla \Phi_j)^\eps \times (S_\eps \nabla b_{\eps,j}).
$$
The first term is estimated with the help of \eqref{srezka}, Lemma \ref{lem02}, and \eqref{efff}.
The second term is estimated by using 
\eqref{srezka}, Lemma \ref{lem02}, and \eqref{eff6b}. We also use  \eqref{Mqq11} and \eqref{Mqq14}.
This implies
\begin{equation}
\label{aaa17}
\| \rot \a_\eps^{(4)}\|_{L_2(\O)} 
\le {\mathfrak C}_{19}^{(4)} \eps^{1/2} \| \q\|_{L_2(\O)}, \quad 0 <  \eps \le \eps_1,
\end{equation}
where the constant ${\mathfrak C}_{19}^{(4)}$ depends on $\|\eta\|_{L_\infty}$,  
$\|\eta^{-1}\|_{L_\infty}$, $\|\mu\|_{L_\infty}$, $\|\mu^{-1}\|_{L_\infty}$, the parameters of the lattice $\Gamma$, and the domain~$\O$.

As a result, relations  \eqref{aaa14}--\eqref{aaa17} imply the required inequality \eqref{aaa9} with the constant 
${\mathfrak C}_{19}= {\mathfrak C}_{19}^{(1)} + {\mathfrak C}_{19}^{(2)} + {\mathfrak C}_{19}^{(4)}$.
\end{proof}

Let $\wh{\xi}_\eps\in H^1_0(\O)$ be the solution of the Dirichlet problem
\begin{equation*}
- \div \eta^\eps \nabla \wh{\xi}_\eps = 
\div \eta^\eps {\mathbf a}_\eps, \quad \wh{\xi}_\eps \vert_{\partial \O}=0.
\end{equation*}
The solution satisfies the identity
\begin{equation}
\label{Dir2}
\int_\O \langle \eta^\eps \nabla \wh{\xi}_\eps, \nabla \omega \rangle \, d\x  =
- \int_\O \langle \eta^\eps \a_\eps, \nabla \omega \rangle \, d\x,\quad \forall \omega \in H^1_0(\O).
\end{equation}
Substituting $\omega= \wh{\xi}_\eps$ in \eqref{Dir2} and using \eqref{aaa8}, we obtain
\begin{equation}
\label{Dir3}
\| ( \eta^\eps)^{1/2} \nabla \wh{\xi}_\eps \|_{L_2(\O)}
\le {\mathfrak C}_{18} \| \eta \|^{1/2}_{L_\infty} \eps^{1/2}\| \q \|_{L_2(\O)},
\quad 0 < \eps \le \eps_1.
\end{equation}

Let ${\mathbf h}_\eps = (\eta^\eps)^{1/2} \a_\eps + ( \eta^\eps)^{1/2} \nabla \wh{\xi}_\eps$. 
Then
\begin{equation}
\label{aaa18}
\rot (\eta^\eps)^{-1/2}{\mathbf h}_\eps = \rot \a_\eps, \quad 
\div (\eta^\eps)^{1/2}{\mathbf h}_\eps = 0, 
\quad 
((\eta^\eps)^{-1/2}{\mathbf h}_\eps)_\tau\vert_{\partial \O}=
({\mathbf a}_\eps)_\tau\vert_{\partial \O}.
\end{equation}
By \eqref{aaa8} and \eqref{Dir3}, 
\begin{equation}
\label{aaa19}
\| {\mathbf h}_\eps \|_{L_2(\O)} 
\le {\mathfrak C}_{20} \eps^{1/2} \| \q\|_{L_2(\O)}, \quad 0 < \eps \le \eps_1.
\end{equation}
where ${\mathfrak C}_{20} = 2 {\mathfrak C}_{18} \|\eta\|^{1/2}_{L_\infty}$.

Now, relations \eqref{pogr1_q}--\eqref{pogr3_q}, \eqref{aaa2}, and \eqref{aaa18} imply that 
$\wh{\s}_\eps - {\mathbf h}_\eps \in \Dom \wh{\mathfrak l}_\eps$ and
\begin{equation}
\label{aaa20}
\wh{\mathfrak l}_\eps [\wh{\s}_\eps - {\mathbf h}_\eps, \bzeta] +  (\wh{\s}_\eps - {\mathbf h}_\eps, \bzeta)_{L_2(\O)}
= \wt{\mathcal Q}_\eps[\bzeta], \quad \forall \bzeta \in \Dom \wh{\mathfrak l}_\eps,
\end{equation}
where
\begin{equation*}
\wt{\mathcal Q}_\eps[\bzeta] = {\mathcal Q}_\eps[\bzeta]
- ((\mu^\eps)^{-1} \rot {\mathbf a}_\eps, \rot (\eta^\eps )^{-1/2} \bzeta)_{L_2(\O)}
- ( {\mathbf h}_\eps,  \bzeta)_{L_2(\O)}.
\end{equation*}
By \eqref{q5_q}, \eqref{aaa9}, and  \eqref{aaa19}, we have
\begin{equation}
\label{aaa22}
| \wt{\mathcal Q}_\eps[\bzeta] | \le 
\wt{\mathfrak C}^\circ \eps^{1/2} \| \q \|_{L_2(\O)}
\left( \| \bzeta \|_{L_2(\O)} + \| \div (\eta^\eps)^{1/2} \bzeta \|_{L_2(\O)} 
+ \| \rot (\eta^\eps)^{-1/2} \bzeta \|_{L_2(\O)} \right)
\end{equation}
for $0< \eps \le \eps_1$, where 
$\wt{\mathfrak C}^\circ = {\mathfrak C}^\circ + {\mathfrak C}_{20} + \|\mu^{-1}\|_{L_\infty} {\mathfrak C}_{19}$.
Substituting $\bzeta = \wh{\s}_\eps - {\mathbf h}_\eps$ in \eqref{aaa20} and using  \eqref{aaa22}, we obtain
\begin{equation}
\label{aaa23}
\begin{aligned}
&\| \wh{\s}_\eps - {\mathbf h}_\eps \|_{L_2(\O)} +
\| \div (\eta^\eps)^{1/2}(\wh{\s}_\eps - {\mathbf h}_\eps)\|_{L_2(\O)}
\\
&+
\| \rot (\eta^\eps)^{-1/2}(\wh{\s}_\eps - {\mathbf h}_\eps)\|_{L_2(\O)}
\leqslant {\mathfrak C}_{7}' \eps^{1/2}  \| \q \|_{L_2(\O)}
\end{aligned}
\end{equation}
for $0< \eps \le \eps_1$, where  ${\mathfrak C}_{7}' = 3 \max \{ 1, \|\mu\|_{L_\infty} \} \wt{\mathfrak C}^\circ$.

Combining  \eqref{aaa9}, \eqref{aaa18}, \eqref{aaa19}, and \eqref{aaa23}, we arrive at the required inequality
\eqref{pogr6a_q} with the constant
${\mathfrak C}_7 = {\mathfrak C}_7' + {\mathfrak C}_{19} + {\mathfrak C}_{20}$. 
This \textit{completes the proof of Theorem} \ref{lem6.6}.

\end{document}